\documentclass[11pt]{article}
\usepackage{smile}

\usepackage{fullpage}
\usepackage{pdflscape}
\usepackage{framed}
\usepackage{mdframed}
\usepackage{enumerate}
\usepackage[inline]{enumitem}
\usepackage[T1]{fontenc}
\usepackage{moresize}
\usepackage{bm}
\usepackage{dsfont}
\usepackage{amsmath}
\usepackage{amssymb}
\usepackage{amsthm}
\usepackage{amsfonts}
\usepackage{array}
\usepackage{mathtools} % \prescript
\usepackage{extarrows}
\usepackage{stackrel}
\usepackage[nodisplayskipstretch]{setspace}
\usepackage{color}
\usepackage[usenames,dvipsnames]{xcolor}
\usepackage{cancel}
\usepackage{soul}
\usepackage{xfrac}
\usepackage{siunitx}
\usepackage{graphicx}
\usepackage{float}
\usepackage{rotating}
\usepackage{subcaption}
\usepackage{overpic}
\usepackage[all]{xy}
\usepackage{tikz}
\usetikzlibrary{patterns}
\usepackage{booktabs}
\usepackage{dcolumn}
\usepackage{multirow}
\usepackage{diagbox}
\usepackage{verbatim}
\usepackage{listings}
\usepackage{algorithm}
\usepackage{algpseudocode}
\usepackage{fancyvrb}
\usepackage{hyperref}
\usepackage[round]{natbib}
\usepackage{sectsty}

\hypersetup{
    bookmarks=true,         % show bookmarks bar?
    unicode=false,          % non-Latin characters in Acrobat’s bookmarks
    pdftoolbar=true,        % show Acrobat’s toolbar?
    pdfmenubar=true,        % show Acrobat’s menu?
    pdffitwindow=false,     % window fit to page when opened
    pdfstartview={FitH},    % fits the width of the page to the window
    pdftitle={My title},    % title
    pdfauthor={Author},     % author
    pdfsubject={Subject},   % subject of the document
    pdfcreator={Creator},   % creator of the document
    pdfproducer={Producer}, % producer of the document
    pdfkeywords={key1, key2}, % list of keywords
    pdfnewwindow=true,      % links in new PDF window
    colorlinks=true,        % false: boxed links; true: colored links
    linkcolor=blue,         % color of internal links (change box color with linkbordercolor)
    citecolor=blue,         % color of links to bibliography
    filecolor=blue,         % color of file links
    urlcolor=cyan           % color of external links
}

\def\tr{\mathop{\text{tr}}\kern.2ex}

\def\P{{\mathbb P}}
\def\E{{\mathbb E}}
\def\R{{\mathbb R}}
\def\Z{{\mathbb Z}}

\renewcommand{\Pr}{{\mathbb P}}

\newcommand{\textand}{,~}
\newcommand{\mbinom}[2]{\Bigl(\mkern1mu\begin{array}{@{}c@{}}#1\\#2\end{array}\mkern1mu\Bigr)}
\newcolumntype{L}[1]{>{\raggedright\let\newline\\\arraybackslash\hspace{0pt}}m{#1}}
\newcolumntype{C}[1]{>{  \centering\let\newline\\\arraybackslash\hspace{0pt}}m{#1}}
\newcolumntype{R}[1]{>{ \raggedleft\let\newline\\\arraybackslash\hspace{0pt}}m{#1}}
\newcolumntype{d}[1]{D{.}{.}{#1}}
\newcolumntype{H}{>{\setbox0=\hbox\bgroup}c<{\egroup}@{}}
\newcolumntype{Z}{>{\setbox0=\hbox\bgroup}c<{\egroup}@{\hspace*{-\tabcolsep}}}

\makeatletter

\newcommand*{\@rowstyle}{}

\newcommand*{\rowstyle}[1]{% sets the style of the next row
  \gdef\@rowstyle{#1}%
  \@rowstyle\ignorespaces%
}

\newcolumntype{=}{% resets the row style
  >{\gdef\@rowstyle{}}%
}

\newcolumntype{+}{% adds the current row style to the next column
  >{\@rowstyle}%
}

\makeatother

\numberwithin{equation}{section}
\newtheorem{theorem}{Theorem}[section]
\newtheorem{lemma}{Lemma}[section]
\newtheorem{proposition}{Proposition}[section]
\newtheorem{assumption}{Assumption}[section]
\newtheorem{corollary}{Corollary}[section]

\newtheorem{innercustomgeneric}{\customgenericname}
\providecommand{\customgenericname}{}
\newcommand{\newcustomtheorem}[2]{%
  \newenvironment{#1}[1]
  {%
   \renewcommand\customgenericname{#2}%
   \renewcommand\theinnercustomgeneric{##1}%
   \innercustomgeneric
  }
  {\endinnercustomgeneric}
}
\newcustomtheorem{customtheorem}{Theorem}
\newcustomtheorem{customassumption}{Assumption}
\newcustomtheorem{customlemma}{Lemma}
\newcustomtheorem{customexample}{Examples}
\theoremstyle{definition}
\newtheorem{example}{Example}[section]
\newtheorem{remark}{Remark}[section]

\newcommand{\nb}[1]{{{#1}}}

\begin{document}

\setlength{\abovedisplayskip}{5pt}
\setlength{\belowdisplayskip}{5pt}
\setlength{\abovedisplayshortskip}{5pt}
\setlength{\belowdisplayshortskip}{5pt}
\hypersetup{colorlinks,breaklinks,urlcolor=blue,linkcolor=blue}

\title{\LARGE High-dimensional consistent independence testing with maxima of rank correlations}

\author{
Mathias Drton\thanks{Department of Mathematics, Technical University
  of Munich, 85748 Garching b. M\"unchen, Germany; e-mail: {\tt mathias.drton@tum.de}},
~~Fang Han\thanks{Department of Statistics, University of Washington, Seattle, WA 98195, USA; e-mail: {\tt fanghan@uw.edu}},~~
Hongjian Shi\thanks{Department of Statistics, University of Washington, Seattle, WA 98195, USA; e-mail: {\tt hongshi@uw.edu}}}

%~~}
%}

\date{}

\maketitle

\begin{abstract} 
  Testing mutual independence for high-dimensional observations is a
  fundamental statistical challenge. Popular tests based on linear and
  simple rank correlations are known to be incapable of detecting
  non-linear, non-monotone relationships, calling for methods that can
  account for such dependences.
  % ones.
  % Some recent developments have partly answered this call via introducing $L_2$-type tests built on more involved correlation measures such as Hoeffding's $D$ and distance covariance. However, an important counterpart of them, the maximum-type tests, is still absent from the literature. This paper fills this notable gap, introducing test statistics that are constructed as maxima of pairwise rank correlations.
  To address this challenge, we propose a family of tests that are
  constructed using maxima of pairwise rank correlations that permit
  consistent assessment of pairwise independence.  Built upon a newly
  developed Cram\'{e}r-type moderate deviation theorem for degenerate
  U-statistics, our results cover a variety of rank correlations
  including Hoeffding's $D$, Blum--Kiefer--Rosenblatt's $R$, and
  Bergsma--Dassios--Yanagimoto's $\tau^*$. The proposed tests are
  distribution-free {in the class of multivariate distributions with continuous margins}, implementable without the need for permutation,
  and are shown to be rate-optimal against sparse alternatives under
  the Gaussian copula
  model. As a by-product of the study, we reveal an identity between the aforementioned three rank correlation statistics, and hence make a step towards proving a  conjecture of Bergsma and Dassios. %We also make a step towards closing a conjecture raised in \citet{MR3178526}.%, which stimulates developments of methods accounting for such dependences.

\end{abstract}


%\noindent This supplementary material provides all the technical proofs and additional simulation results.

\section{Introduction}

Let $\mX=(X_1,\ldots,X_p)^\top$ be a random vector taking values in
$\R^p$ and having all univariate marginal distributions continuous.
This paper is concerned with testing the null hypothesis
\begin{align}\label{eq:H0}
H_0: X_1,\ldots,X_p \text{ are mutually independent},
\end{align}
based on $n$ independent realizations $\mX_1,\ldots,\mX_n$ of
$\mX$. Testing $H_0$ is a core problem in multivariate statistics that
has attracted the attention of statisticians for decades; see e.g.\
the exposition in \citet[Chap.~9]{MR1990662} or
\citet[Chap.~11]{MR652932}.  Traditional methods such as the
likelihood ratio test, Roy's largest root test \citep{MR0092296}, and
Nagao's $L_2$-type test \citep{MR0339405} target the case
where the dimension $p$ is small and perform poorly when $p$ is
comparable to or even larger than $n$.  A line of recent work seeks to
address this issue and develops tests that are suitable for modern
applications involving data with large dimension $p$.  This high-dimensional regime is in the focus of our work, which develops
distribution theory based on asymptotic regimes where $p=p_n$
increases to infinity with $n$.

Many tests of independence in high dimensions have been proposed
recently.  For example, \citet{MR2572444} and \citet{MR3127857} derived
corrected likelihood ratio tests for Gaussian data.  Using covariance/correlation statistics such as Pearson's $r$, Spearman's
$\rho$, and Kendall's $\tau$, \citet{MR2988403}, \citet{MR3641402}, \citet{MR3757515},
and \citet{bao2017tracy} proposed revised versions of Roy's largest
root test.  \citet{MR2234197} and \citet{MR3766953} derived corrected
Nagao's $L_2$-type tests.  Finally, \citet{MR2052906},
\citet{MR2327033}, and \citet{MR3737306} proposed tests using the
magnitude of the largest pairwise correlation
statistics. %They showed, under $H_0$, the proposed test statistics will all weakly converge to Gumbel distributions.
Subsequently we shall refer to tests of this latter type as
maximum-type tests.

The aforementioned approaches are largely built on linear and simple
rank correlations. %such as Pearson's $r$ and Kendall's $\tau$.
These, however, are incapable of detecting more complicated
non-linear, non-monotone dependences as \citet{MR0029139} noted in his
seminal paper.  Recent work thus proposed the use of consistent rank
\citep{MR3178526}, kernel-based \citep{NIPS2007_3201,MR3744710}, and
distance covariance/correlation statistics \citep{MR2382665}.
However, % compared to linear and simple rank correlation based tests,
much less is known about high-dimensional tests of $H_0$ that use
these more involved statistics.  Notable exceptions include
\citet{MR3766953} and \citet{MR3798874}. There, the authors combined
Nagao's $L_2$-type methods with rank and distance covariance
statistics that in a tour de force are shown to
weakly converge to a Gaussian limit under the null. In addition,
\citet{MR3798874} proved that an infeasible version of their test is
rate-optimal against a Gaussian dense alternative (Gaussian
distribution with equal correlation), while still little is known about optimality of 
Leung and Drton's.

In this paper, we derive maximum-type tests that are counterparts of
Leung--Drton and Yao--Zhang--Shao $L_2$-type ones.  As  noted in
\citet{MR3737306}, \citet{MR3766953},  and \citet{MR3798874},
% compared to their $L_2$-type counterparts,
maximum-type tests will be more sensitive to  strong but sparse
dependence.
% , a setting of strong interest in many statistical applications.
% This paper hence also answers our peers' call for developing such type of tests. %corresponding to Leung-Drton and Yao-Zhang-Shao $L_2$-type ones. 
{Designed to assess pairwise independence consistently,} our tests are formed using statistics based on pairwise rank correlation
measures such as Hoeffding's $D$ \citep{MR0029139},
Blum--Kiefer--Rosenblatt's $R$ \citep{MR0125690}, and
Bergsma--Dassios--Yanagimoto's $\tau^*$ \citep{MR3178526,zbMATH03366369}. 
{In particular,} assuming the pair of random variables $X_i$ and $X_j$ to have
{a joint distribution that is not only continuous but also absolutely continuous,} these measures all satisfy the following three desirable
properties summarized in \citet{MR3842884}:
\begin{description}
  \setlength{\itemindent}{-.5em}
\item[\rm\em I-consistency.] If  $X_i$ and $X_j$ are independent, the correlation measure is zero.
\item[\rm\em D-consistency.] If  $X_i$ and $X_j$ are dependent, the correlation measure is nonzero.
\item[\rm\em Monotonic invariance.] The correlation measure is invariant to monotone transformations. 
\end{description}

{We remark that invariance under invertible (and not just monotonic) transformations was considered in work on self-equitable measures of dependence \citep{MR3200177}. This leads to notions of mutual information whose estimates are different from and usually more challenging to handle than the rank correlation measures we consider here; see \citet{MR3992389} and references therein.  Indeed, as we shall review in Section \ref{sec:rank}, the aforementioned correlation measures are naturally estimated via U-statistics, which despite being degenerate have important special properties.}

%As we shall review in Section \ref{sec:rank}, the aforementioned correlation statistics can all be treated as instances of degenerate U-statistics, with important special properties.

The contributions of our work are threefold.  First, we prove that all
the maximum-type test statistics {proposed} in Section~\ref{sec:tests}
have a null distribution that converges to a (non-standard) Gumbel distribution under
high-dimensional asymptotics. {This is in contrast to the results in \citet{MR3737306}, where those rank correlation measures that permit consistent assessment of pairwise independence are excluded from the analysis. This exclusion is due to the lack of necessary probability tools like   Cram\'er-type moderate deviation bounds for degenerate U-statistics, which are newly developed in this paper.} {Additionally,} no distributional assumption {except for marginal  continuity} is
required for this result\nb{,} and the parameters for the Gumbel limit can be explicitly
given. % like second moment existing.
% Hence, there is no need of permutation to implement the tests.
This allows one to avoid permutation analysis in problems of larger scale.
Second, we conduct a power analysis and give explicit conditions on
a sparse local alternative under which our proposed tests have power
tending to one.  Third, we show that the maximum-type tests based on
Hoeffding's $D$, Blum--Kiefer--Rosenblatt's $R$, and
Bergsma--Dassios--Yanagimoto's $\tau^*$ are all rate-optimal in the class
of Gaussian (copula) distributions with sparse and strong dependence as characterized in the power analysis. {To our knowledge this is the first rate-optimality result for a feasible test that permits consistent assessment of pairwise independence.} These
results are developed in Section~\ref{sec:theory}.  The theoretical
advantages of our tests are highlighted in simulation studies
(Section~\ref{sec:simulation}).  {Lastly, we note that, as an interesting by-product
of the study, we give an identity among the above three statistics
that helps make a step towards proving Bergsma--Dassios's conjecture
about general D-consistency of $\tau^*$.  This observation, along with other discussions, is given in Section \ref{sec:discussion}. }
All proofs {and additional simulation results} are deferred to a supplement.

\paragraph{Notation}

The sets of real, integer, and positive integer numbers are denoted
$\R$, $\Z$, and $\Z^+$, respectively.  The cardinality of a set
$\mathcal{A}$ is written $\#\mathcal{A}$.  For $m\in\Z^+$, we define
$[m]=\{1,2,\ldots,m\}$ and write $\mathcal{P}_m$ for the set of all
$m!$ permutations of $[m]$. Let
$\mv=(v_1,\ldots,v_p)^\top\in\R^p$,
$\fM=[\fM_{jk}]\in \R^{p\times p}$, and $I,J$
be two subsets of $[p]$.  Then $\mv_I$ is the sub-vector of $\mv$ with entries
indexed by $I${, i.e., $\mv_I=(v_{i_1},v_{i_2},\ldots,v_{i_{\# I}})^\top$ with $i_1<i_2<\cdots<i_{\# I}$ and $\{i_1,\ldots,i_{\# I}\}=I$. Both} $\fM_{I,J}$ and $\fM[I,J]$ are used to refer
to the sub-matrix of $\fM$ with rows indexed by $I$ and columns
indexed by
$J$. %For $0<q<\infty$, let $\norm{\mv}_q:=(\sum |v_i|^q)^{1/q}$ be the vector $L_q$ norm.
%Let $\norm{\fM}_{\max}:=\max_{jk}|\fM_{jk}|$ be the matrix elementwise maximum norm. 
The matrix ${\rm diag}(\fM)\in \R^{p\times p}$ is the diagonal matrix
whose diagonal is the same as that of $\fM$.  We write $\fI_p$ and $\fJ_p$ for the
identity matrix and all-ones matrix in $\R^{p\times p}$, respectively.    For a function $f:\cX\to \R$, we define
$\norm{f}_{\infty}:=\max_{x\in\cX}|f(x)|$.  %For any function
%{$f:\R^p\to\R$}, we write $f(\mv)=(f(v_1),\ldots,f(v_p))^\top$.  
The greatest integer less than or equal to $x\in\R$ is denoted by 
$\lfloor x\rfloor$.  The symbol $\ind(\cdot)$ is used for indicator
functions.  For any two real sequences $\{a_n\}$ and $\{b_n\}$, we
write $a_n \lesssim b_n$, $a_n=O(b_n)$, or equivalently
$b_n \gtrsim a_n$, if there exists $C>0$ such that $|a_n|\leq C|b_n|$
for any large enough $n$. We write $a_n \asymp b_n$ if both
$a_n \lesssim b_n$ and $a_n\gtrsim b_n$ hold. Write $a_n=o(b_n)$ if
for any $c>0$, $|a_n|\leq c|b_n|$ holds for any large enough $n$.
Throughout, $c$ and $C$ refer to positive absolute constants
whose values may differ in different parts of the paper.

% \subsection{Paper organization}

% The rest of the paper is organized as follows. In Section \ref{sec:rank} we introduce the family of rank correlation measures that will be used in formulating our tests. Section \ref{sec:tests} gives the proposed tests, while Section \ref{sec:theory} gives the theoretical analysis, including the study of size, power, and rate-optimality. Section \ref{sec:simulation} provides some finite-sample investigation and highlights the empirical advantages of the proposed tests over existing ones. We end the paper by providing discussions, presented in Section \ref{sec:discussion}. 
% All proofs are relegated to a supplement.

\section{Rank correlations and degenerate U-statistics}\label{sec:rank}

This section introduces the pairwise rank correlations that
will later be aggregated in a maximum-type test of the independence
hypothesis in~(\ref{eq:H0}).  We present these correlations in a general
U-statistic framework.  In the sequel, unless otherwise stated, the
random vector $\mX$ is assumed to have continuous margins, that is,
its marginal distributions are continuous, though not necessarily absolutely
continuous.
% , that is, the marginal distribution functions of $X_1,\ldots,X_p$ are continuous. %We first introduce the notions of ranks and relative ranks.  For any two indices $j\ne k\in[p]$, similar to \citet{MR3737306}, we denote $Q_{ni}^j$ to be the rank of $X_{ij}$ in $\{X_{1j},\ldots,X_{nj}\}$. Since we are focused on continuous random vector $\mX$, there is no need to consider ties. We further define $\{R_{n1}^{jk},\ldots,R_{nn}^{jk}\}$ to be the relative ranks of the $k$-th entry to the $j$-th entry. In other words, for any $i\in[n]$, we define $R_{ni}^{jk}=Q_{ni'}^{k}$ as long as $Q_{ni'}^j=i$.

Let $\mX_1,\dots,\mX_n$ be {independent} copies of $\mX$, with
$\mX_i=(X_{i,1},\dots,X_{i,p})^\top$.  Let
$j\ne k\in[p]$, and let $h: (\R^{2})^m\to\R$ be a fixed kernel of order
$m$.  The kernel $h$ defines a \emph{U-statistic} of order $m$:
\begin{align}\label{eq:Uh}
\hat U_{jk}=\mbinom{n}{m}^{-1}\sum_{1\le i_1< i_2<\cdots< i_m\le n}
  h\Big\{
  \Big(\begin{matrix}
    X_{i_1,j} \\     X_{i_1,k}
  \end{matrix}\Big),\ldots,\Big(\begin{matrix}
    X_{i_m,j} \\     X_{i_m,k}
  \end{matrix}\Big)\Big\}.
\end{align}
For our purposes $h$ may always be assumed to be \emph{symmetric},
i.e., $ %\begin{equation*}
h(\mz_1,\ldots,\mz_m)=h(\mz_{\sigma(1)},\ldots,\mz_{\sigma(m)})$
%\end{equation*}
for all permutations $\sigma\in\mathcal{P}_m$ and $\mz_1,\ldots,\mz_m\in\R^2$. % with
%$\mz_i=(z_{i,1},z_{i,2})^\top$.  %Consider a kernel
%argument given by $\mz_1,\ldots,\mz_m\in\R^2$ with
%$\mz_i=(z_{i,1},z_{i,2})^\top$.   
Letting $\mz_i=(z_{i,1},z_{i,2})^\top$, if both vectors
$(z_{1,1},\ldots,z_{m,1})$ and $(z_{1,2},\ldots,z_{m,2})$ are free of ties, i.e., have marginal distinct entries,
then we have well-defined vectors of ranks $(r_{1,1},\ldots,r_{m,1})$
and $(r_{1,2},\ldots,r_{m,2})$, and we define $\mr_i=(r_{i,1},r_{i,2})^\top$
for $1\le i\le n$.  Now a kernel is \emph{rank-based} if
\begin{equation*}
h(\mz_1,\ldots,\mz_m)=h(\mr_1,\ldots,\mr_m)
\end{equation*}
for all $\mz_1,\ldots,\mz_m\in\R^2$ with $(z_{1,1},\ldots,z_{m,1})$ and
$(z_{1,2},\ldots,z_{m,2})$ free of ties.  In this case, we also say that the ``correlation'' statistic
$\hat U_{jk}$ as well as the corresponding ``correlation measure'' $\E\hat U_{jk}$ is rank-based.  

Rank-based statistics have many
appealing properties with regard to independence.  The following
three will be of particular importance for us.  Proofs can be
found in, e.g., Chapter 31 in \citet{kendall1979advanced},
Lemma C4 in the supplement of \citet{MR3737306}, and Lemma 2.1 in
\citet{MR3766953}. We also note that, in finite samples, the statistics $\{\hat
U_{jk};j< k\}$ are generally not mutually independent.

\begin{proposition}\label{prop:easy}
Under the null hypothesis in \eqref{eq:H0} and assuming continuous
margins, we have:  
\begin{enumerate}[label=(\roman*)]
\item\label{prop:easy-1} The {rank} statistics $\{\hat U_{jk}, j\ne k\}$ are all
  identically distributed and are distribution-free, i.e., the
  distribution of $\hat U_{jk}$ does not depend on the marginal
  distributions of $X_1,\ldots,X_p$;
\item\label{prop:easy-2} Fix any $j\in[p]$, then the {rank} statistics $\{\hat U_{jk},
  k\ne j\}$, are mutually independent;
\item\label{prop:easy-3} For any $j\ne k\in [p]$, the {rank} statistic $\hat U_{jk}$ is independent of $\{\hat U_{j'k'}; j', k'\not\in\{j,k\}, j'\ne k' \}$.
\end{enumerate}
\end{proposition}

%We note that to our knowledge, the question whether $\{\hat
%U_{jk};j< k\}$ are mutually independent remains open. 

Our focus will be on those rank-based correlation statistics and the corresponding measures that are induced by the kernel $h(\cdot)$ and are both I-
and D-consistent.  The kernels of these measures satisfy important additional
properties that we will assume in our general treatment.  Further
concepts concerning U-statistics are needed to state this assumption.
For any kernel $h(\cdot)$, any number $\ell\in[m]$, and any measure
$\Pr_{\mZ}$, we write
\begin{equation*}
h_{\ell}(\mz_1\ldots,\mz_{\ell}; \Pr_{\mZ}):=\E h(\mz_1\ldots,\mz_{\ell},\mZ_{\ell+1},\ldots,\mZ_m)
\end{equation*}
and 
\begin{align}\label{eq:han-hl}
    &h^{(\ell)}(\mz_1,\ldots,\mz_{\ell}; \Pr_{\mZ}):= \\
&~~~h_{\ell}(\mz_1,\ldots,\mz_{\ell};\Pr_{\mZ})-\E h-\sum_{k=1}^{\ell-1}\sum_{1\leq i_1<\cdots<i_k\leq\ell}h^{(k)}(\mz_{i_1},\ldots,\mz_{i_k};\Pr_{\mZ}), \notag
\end{align}
where $\mZ_1,\ldots,\mZ_m$ are $m$ independent random vectors with
distribution $\Pr_{\mZ}$ {and $\E h:=\E h(\mZ_1,\ldots,\mZ_m)$}.  The kernel as well as the corresponding U-statistic is \emph{degenerate} under $\Pr_{\mZ}$ if $h_1(\cdot)$
has variance zero.  We use the term
\emph{completely degenerate} to indicate that  the
variances of $h_1(\cdot),\ldots,h_{m-1}(\cdot)$ are all zero.
Finally, let
$\Pr_0$ be the uniform distribution on $[0,1]$, and write
$\Pr_0\otimes \Pr_0$ for its product measure, the uniform distribution
on $[0,1]^2$.  Note that by Proposition~\ref{prop:easy}\ref{prop:easy-1}, the study
of $\hat U_{jk}$ under independent
continuous margins $X_j$ and $X_k$ can be reduced to the case with $(X_j,X_k)^\top\sim \Pr_0\otimes \Pr_0$.
% With these
% definitions, we are ready to make the following major assumption on
% the kernel $h(\cdot)$.  

%the following three more properties on the kernel $h$. %In the sequel $\E_{H_0}(\cdot)$ will represent the expectation operator under $H_0$.

\begin{assumption}\label{assumption:key} The kernel $h$ is rank-based, symmetric, and has the following three properties: 
\begin{enumerate}[label=(\roman*)]
\item $h$ is bounded.
\item\label{assumption:key-2} $h$ is mean-zero and degenerate under independent continuous margins, i.e., {$\E \{h_1(\mZ_1; \Pr_0\otimes \Pr_0)\}^2=0$ as $\mZ_1\sim \Pr_0\otimes \Pr_0$.}
\item $h_2(\mz_1,\mz_2;\Pr_0\otimes \Pr_0)$ has uniformly bounded
  eigenfunctions, that is, it admits the expansion
\begin{equation*}
h_2(\mz_1,\mz_2;\Pr_0\otimes \Pr_0)=\sum_{v=1}^{\infty}\lambda_v\phi_v(\mz_1)\phi_v(\mz_2),
\end{equation*}
where $\{\lambda_v\}$ and $\{\phi_v\}$ are the eigenvalues and eigenfunctions satisfying the integral equation
\begin{equation*}
\E h_2(\mz_1,\mZ_2)\phi(\mZ_2)=\lambda\phi(\mz_1)~~~\text{for all }\mz_1\in\R^2,
\end{equation*}
with $\mZ_2\sim \Pr_0\otimes\Pr_0$, $\lambda_1\ge\lambda_2\ge\cdots\ge0$, $\Lambda:=\sum_{v=1}^{\infty}\lambda_v\in(0,\infty)$, and $\sup_v\lVert \phi_v\rVert_\infty<\infty$.
\end{enumerate}
\end{assumption}

The first boundedness property is satisfied for the commonly used rank
correlations, including Kendall's $\tau$, Spearman's $\rho$,
and many others.   The latter two properties are much more specific, but
exhibited by the classical rank correlation measures for which consistency
properties are known.  %\nb{To our knowledge, the first exception arises from a very recent development, namely, the rank correlation introduced in \citet{chatterjee2019new}}. 
We discuss the main examples below.
Note also that the assumption $\Lambda>0$ implies $\lambda_1>0$, so
that $h_2(\cdot)$ is not a constant function.

%We then move on to introduce several rank correlation measures that satisfy Assumption \ref{assumption:key}. %Starting from here until the end of this section, the data are assumed not necessarily absolutely continuous. 

\begin{example}[Hoeffding's $D$]\label{eg:hd}  From the symmetric kernel 
\begin{align*}
     & h_D(\mz_1,\dots,\mz_5):= \frac{1}{16}\sum_{(i_1,\dots,i_5)\in\mathcal{P}_5}\\
&~~~~ \Big[\Big\{\ind(z_{i_1,1}\leq z_{i_5,1})-\ind(z_{i_2,1}\leq z_{i_5,1})\Big\}\Big\{\ind(z_{i_3,1}\leq z_{i_5,1})-\ind(z_{i_4,1}\leq z_{i_5,1})\Big\}\Big]\\
&~~~~  \Big[\Big\{\ind(z_{i_1,2}\leq z_{i_5,2})-\ind(z_{i_2,2}\leq z_{i_5,2})\Big\}\Big\{\ind(z_{i_3,2}\leq z_{i_5,2})-\ind(z_{i_4,2}\leq z_{i_5,2})\Big\}\Big],
\end{align*}
we recover Hoeffding's $D$ statistic, which is a rank-based
U-statistic of order 5 and gives rise to the Hoeffding's $D$ correlation measure $\E h_D$.  The kernel $h_D(\cdot)$ satisfies the first two properties in
Assumption \ref{assumption:key} in view of the results in
\citet{MR0029139}. To verify the last property, we note that under
the measure $\Pr_0\otimes\Pr_0$,
$h_{D,2}(\cdot)$ is known to have
eigenvalues
\begin{equation*}
\lambda_{i,j; D}={3}/{(\pi^4i^2j^2)},\quad i,j\in\Z^+;
\end{equation*}
see, e.g., Proposition~7 in \citet{MR3842884} or 
Theorem~4.4 in \citet{MR3541972}.  The corresponding eigenfunctions are
\begin{equation*}
\phi_{i,j; D}\{(z_{1,1},z_{1,2})^\top\}=2\cos(\pi i
z_{1,1})\cos(\pi j z_{1,2}), \quad i,j\in\Z^+.
\end{equation*}
The eigenvalues are positive and sum to
$\Lambda_{D}:=\sum_{i,j}\lambda_{i,j; D}={1}/{12}$, and
$\sup_{i,j}\lVert\phi_{i,j;D}\rVert_{\infty}\le 2$. 
For any pair of random variables, the correlation measure $\E h_D\ge 0$ \citep[p.~547]{MR0029139}.
Furthermore, it has been proven that, once the pair is absolutely continuous 
in $\R^2$, the correlation measure $\E h_D=0$ if and only if the pair 
is independent \citep{MR0029139,zbMATH03366369}. 
This property, however, generally does not hold for 
discrete data or data generated from a bivariate distribution that is
continuous but not absolutely continuous; see Remark 1 in
\citet{zbMATH03366369} for a counterexample.
\end{example}

\begin{example}[Blum--Kiefer--Rosenblatt's $R$]\label{eg:bkr} The
  symmetric kernel 
\begin{align*}
 &    h_R(\mz_1,\dots,\mz_6):=  \frac{1}{32}\sum_{(i_1,\dots,i_6)\in\mathcal{P}_6}\\
      &~~~ \Big[\Big\{\ind(z_{i_1,1}\leq z_{i_5,1})-\ind(z_{i_2,1}\leq z_{i_5,1})\Big\}\Big\{\ind(z_{i_3,1}\leq z_{i_5,1})-\ind(z_{i_4,1}\leq z_{i_5,1})\Big\}\Big]\\
&~~~   \Big[\Big\{\ind(z_{i_1,2}\leq z_{i_6,2})-\ind(z_{i_2,2}\leq z_{i_6,2})\Big\}\Big\{\ind(z_{i_3,2}\leq z_{i_6,2})-\ind(z_{i_4,2}\leq z_{i_6,2})\Big\}\Big]
\end{align*}
yields Blum--Kiefer--Rosenblatt's $R$ statistic \citep{MR0125690}, which is a
rank-based U-statistic of order 6. One can verify the three properties
in Assumption \ref{assumption:key} similarly to Hoeffding's $D$ by
using that $h_{R,2}=2h_{D,2}$. In addition, for any pair of random
variables, the correlation measure $\E h_R\ge0$ with equality if and only if the pair is independent, {and no continuity assumption is needed at all};
cf.\  page 490 of \citet{MR0125690}.
\end{example}

\begin{example}[Bergsma--Dassios--Yanagimoto's $\tau^*$]\label{eg:bd}  \citet{MR3178526} introduced a rank correlation statistic as a U-statistic of order 4 with the symmetric kernel
\begin{align*}
h_{\tau^*}(\mz_1,\dots,\mz_4)\\
:= \frac{1}{16}\sum_{(i_1,\dots,i_4)\in\mathcal{P}_4} 
\Big\{\;&\ind(z_{i_1,1},z_{i_3,1}<z_{i_2,1},z_{i_4,1})+\ind(z_{i_2,1},z_{i_4,1}<z_{i_1,1},z_{i_3,1})\\[-.5em]
     -\;&\ind(z_{i_1,1},z_{i_4,1}<z_{i_2,1},z_{i_3,1})-\ind(z_{i_2,1},z_{i_3,1}<z_{i_1,1},z_{i_4,1})\Big\}\\
\Big\{\;&\ind(z_{i_1,2},z_{i_3,2}< z_{i_2,2},z_{i_4,2})+ \ind(z_{i_2,2},z_{i_4,2}< z_{i_1,2},z_{i_3,2})\\
     -\;&\ind(z_{i_1,2},z_{i_4,2}< z_{i_2,2},z_{i_3,2})-\ind(z_{i_2,2},z_{i_3,2}< z_{i_1,2},z_{i_4,2})\Big\}.
\end{align*}
Here,
$\ind(y_1,y_2<y_3,y_4):=\ind(y_1<y_3)\ind(y_1<y_4)\ind(y_2<y_3)\ind(y_2<y_4)$.
It holds that $h_{\tau^*,2}=3h_{D,2}$ and all properties in Assumption
\ref{assumption:key} also hold for
$h_{\tau^*}(\cdot)$.  %Furthermore, Theorem 1 in \citet{MR3178526} guaranteed that $\E h_{\tau^*}=0$ if and only if the pair is independent, provided that the joint distribution is discrete or continuous, or a mixture of both.
Theorem 1 in \citet{MR3178526} shows that for a pair of random
variables whose distribution is discrete, absolutely continuous, or a
mixture of both, the correlation measure $\E h_{\tau^*}\ge0$ where equality holds
if and only if the variables are independent.  It has been conjectured that this
fact is true for any distribution on $\R^2$. {In Section
  \ref{sec:tau} of this paper we make new progress along this track. This progress is based on early but apparently little known results of
  \citet{zbMATH03366369} that prompted us to add his name in reference
  to $\tau^*$.}
% To our knowledge, it is still unknown if the same
% conclusion applies to singular distributions such as the uniform
% distribution on the unit circle in $\R^2$.
\end{example}

\section{Maximum-type tests of mutual independence}\label{sec:tests}

We now turn to tests of the mutual independence hypothesis $H_0$
in~(\ref{eq:H0}).  As in \citet{MR3737306}, we propose maximum-type
tests.  However, in contrast to \citet{MR3737306}, we suggest the use of
consistent and rank-based correlations with the practical
choices being {the ones} from Examples~\ref{eg:hd}--\ref{eg:bd}.  As
these measures are all nonnegative, it is
appropriate to consider a one-sided test in which we aggregate
pairwise U-statistics $\hat U_{jk}$ in \eqref{eq:Uh} into the test statistic
\[
\hat M_n:=(n-1)\max_{j<k}\hat U_{jk}.
\]
We then reject $H_0$ if $\hat M_n$ is larger than a certain threshold.  
Note that we tacitly assumed $\hat U_{jk}=\hat U_{kj}$ when maximizing
over $j<k$;  this symmetry holds for any 
reasonable correlation statistic.
{We emphasize once more that, since the statistic is constructed
  based on pairs $\{X_{i,j},X_{i,k}\}_{i\in[n]}$, the proposed tests
  % of mutual independence
  are designed to assess pairwise independence consistently.}

By Proposition~\ref{prop:easy}\ref{prop:easy-1}, the statistic $\hat M_n$ is
distribution-free {in the class of multivariate distributions with continuous margins}.  An exact critical value for rejection of $H_0$ could thus
be approximated by Monte Carlo simulation.  However, as we will show,
extreme-value theory yields asymptotic critical values that avoid any
extra computation all the while giving good finite-sample control of
the test's size.  When presenting this theory, we write
$X\stackrel{\sf d}{=}Y$ if two random variables $X$ and $Y$ have the
same distribution, and we use  $\stackrel{\sf d}{\longrightarrow}$ to denote ``weak convergence''. 

If, under $H_0$, the studied statistic $(n-1)\hat U_{jk}$ weakly
converged to a chi-square distribution with one degree of freedom, as
in Theorems 1 and 2 of \citet{MR3737306}, then extreme-value theory
combined with Proposition \ref{prop:easy} would imply that a suitably
standardized version of $\hat M_n$ would weakly converge to a type-I Gumbel distribution with distribution function $\exp\{-(8\pi)^{-1/2}\exp(-y/2)\}$. %The threshold can then be explicitly calculated. 
However,  the degeneracy stated in Assumption \ref{assumption:key}\ref{assumption:key-2}
rules out this possibility.  Classical theory yields that instead of a single chi-square
variable, we encounter convergence to much more involved infinite weighted series \citep[Chap.~5.5.2]{MR595165}. 

\begin{proposition}\label{lem:easy}
  Let $\mX$ have continuous margins, and let  $j\ne k$.  If
  $h(\cdot)$ satisfies Assumption \ref{assumption:key}, then under $H_0$,
\[
\mbinom{m}{2}^{-1}(n-1)\hat U_{jk} \stackrel{\sf d}{\longrightarrow} \sum_{v=1}^{\infty}\lambda_v(\xi_v^2-1),
\]
where $\{\xi_v,v=1,2,\ldots\}$ are i.i.d.~standard Gaussian random variables. 
\end{proposition}

Note that the weak convergence result for degenerate U-statistics 
in Proposition~\ref{lem:easy} 
holds under much weaker conditions than Assumption \ref{assumption:key}; see the main theorem in \citet[Chap.~5.5.2]{MR595165} for detailed conditions.
Our intuition for the asymptotic forms of the maxima now comes from the following fact, {though the analysis of $\max_{j<k} \hat U_{jk}$ requires more refined techniques since $\{\hat U_{jk};j\leq k\}$ are in general not mutually independent}.

\begin{proposition}\label{lem:easy2} Let $Y_1,\ldots,Y_{d}$ be
  $d=p(p-1)/2$ {independent} copies of $\zeta\stackrel{\sf
    d}{=}\sum_{v=1}^{\infty}\lambda_v(\xi_v^2-1)$.   Then, as $p\to\infty$,
\[
\max_{j\in[d]} \frac{Y_{j}}{\lambda_1} -4\log p-(\mu_1-2)\log\log p+\frac{\Lambda}{\lambda_1} \stackrel{\sf d}{\longrightarrow} G.
\]
Here $G$ follows a Gumbel distribution with distribution function 
\[
\exp\Big\{-\frac{2^{\mu_1/2-2}\kappa}{\Gamma(\mu_1/2)}\exp\Big(-\frac{y}{2}\Big)\Big\},
\]
where $\mu_1$ is the multiplicity of the largest eigenvalue $\lambda_1$ in the sequence $\{\lambda_1,\lambda_2,\dots\}$, $\kappa:=\prod_{v= \mu_1+1}^{\infty}(1-\lambda_v/\lambda_1)^{-1/2},$ and $\Gamma(z):=\int_0^{\infty}x^{z-1}e^{-x}dx$ is the gamma function.
\end{proposition}

Obviously, when setting $\lambda_1=1,
\lambda_2=\lambda_3=\cdots=0$ in Proposition \ref{lem:easy2}, we recover the Gumbel distribution
derived by \citet{MR3737306}.  Based on Propositions \ref{lem:easy} and \ref{lem:easy2}, for any pre-specified significance level $\alpha\in(0,1)$, our proposed test is 
\begin{equation}\label{eq:general}
\mathsf{T}_{\alpha}:= \ind\Big\{\frac{n-1}{\lambda_1\binom{m}{2}}\max_{j<k}\hat U_{jk}-4\log p-(\mu_1-2)\log\log p+\frac{\Lambda}{\lambda_1}>Q_\alpha\Big\},
\end{equation}
where 
\begin{equation*}
Q_\alpha:=\log\frac{2^{\mu_1-4}\kappa^2}{\{\Gamma(\mu_1/2)\}^2}-2\log\log(1-\alpha)^{-1}
\end{equation*}
is the $1-\alpha$ quantile of the Gumbel distribution of distribution function $\exp\{-{2^{\mu_1/2-2}\kappa}/\Gamma(\mu_1/2)\cdot\exp(-y/2)\}$. {However, note that so far the test results merely
  from heuristic arguments. Theoretical justifications
  regarding the test's size and power under the high-dimensional regime will
  be given in Section \ref{sec:theory}.}

\begin{example}[``Extreme $D$'']\label{eg:hd2}
Hoeffding's $D$ statistic introduced in Example \ref{eg:hd} is
\[
\hat D_{jk}:=\mbinom{n}{5}^{-1}\sum_{\substack{i_1<\cdots<i_5}}h_D(\mX_{i_1,\{j,k\}},\dots,\mX_{i_5,\{j,k\}}).
\]
 According to \eqref{eq:general}, the corresponding test is
\[
\mathsf{T}_{D,\alpha}:= \ind\Big\{\frac{\pi^4(n-1)}{30}\max_{j<k}\hat D_{jk}-4\log p+\log\log p+\frac{\pi^4}{36}>Q_{D,\alpha}\Big\},
\]
where $Q_{D,\alpha}:=\log\{\kappa_D^2/(8\pi)\}-2\log\log(1-\alpha)^{-1}$ and 
\[
\kappa_D:=\Big\{2\prod_{n=2}^{\infty}\frac{\pi/n}{\sin(\pi/n)}\Big\}^{1/2}\approx 2.467.
\] 
\end{example}

\begin{example}[``Extreme $R$'']\label{eg:bkr2}
Blum--Kiefer--Rosenblatt's $R$ statistic from Example \ref{eg:bkr} is
\[
\hat R_{jk}:=\mbinom{n}{6}^{-1}\sum_{\substack{i_1<\cdots<i_6}}h_R(\mX_{i_1,\{j,k\}},\dots,\mX_{i_6,\{j,k\}}).
\]
According to \eqref{eq:general}, the corresponding test is
\[
\mathsf{T}_{R,\alpha}:= \ind\Big\{\frac{\pi^4(n-1)}{90}\max_{j<k}\hat R_{jk}-4\log p+\log\log p+\frac{\pi^4}{36}>Q_{R,\alpha}\Big\},
\]
where $Q_{R,\alpha}:=Q_{D,\alpha}$.
\end{example}

\begin{example}[``Extreme $\tau^*$'']\label{eg:bd2}
  Bergsma--Dassios--Yanagimoto's $\tau^*$ statistic from Example \ref{eg:bd} is
\[
\hat\tau^*_{jk}:=\mbinom{n}{4}^{-1}\sum_{\substack{i_1<\cdots<i_4}}h_{\tau^*}(\mX_{i_1,\{j,k\}},\dots,\mX_{i_4,\{j,k\}}).
\]
According to \eqref{eq:general}, it yields the test
\[
\mathsf{T}_{\tau^*,\alpha}:= \ind\Big\{\frac{\pi^4(n-1)}{54}\max_{j<k}\hat\tau^*_{jk}-4\log p+\log\log p+\frac{\pi^4}{36}>Q_{\tau^*,\alpha}\Big\},
\]
where $Q_{\tau^*,\alpha}:=Q_{D,\alpha}$.
\end{example}

% In all the previous three examples, there is no unknown parameter
% related to the distribution of $\mX$, and the tests are directly
% implementable, although alternative simulation-based tests could also
% be easily implemented as will be shown in more details by the end of
% Section \ref{sec:simulation}. This partly demonstrates the proposed
% tests' computational advantages over linear or distance
% covariance-based ones.
% We also note that, by the definitions of the kernels and the identity
% \eqref{eq:identity}, as long as there is no tie in the data, for any
% $j,k\in[p]$,
Note that, by the definitions of the kernels and the identity \eqref{eq:identity} that will be introduced in Section \ref{sec:tau}, as long as there is no tie in the data, for any $j,k\in[p]$, 
\begin{align}\label{eq:identity2}
\hat D_{jj}=\hat R_{jj}=\hat\tau^*_{jj}=1~~~\text{and}~~~3\hat D_{jk}+2\hat R_{jk}=5\hat\tau^*_{jk}.
\end{align}
%In addition, for any $j, k\in[p]$, we have $\hat D_{jk}=\hat D_{kj}$, $\hat R_{jk}=\hat R_{kj}$, and $\hat\tau^*_{jk}=\hat\tau^*_{kj}$.

\begin{remark}
{In applying the above tests we have intrinsically assumed that
  there are no ties among the entries $X_{1,j},\ldots,X_{n,j}$ for
  each $j\in[p]$. This is based on the assumption that $\mX=(X_1,\ldots,X_p)^{\top}$ has continuous margins. In practice, however, data in finite accuracy might feature ties or may indeed be drawn from a distribution that is not of a continuous margin. In such cases, conducting the above tests on the original data may distort the size. To fix this, as was discussed in Remark 2.1 in \citet{MR3491123}, one may break the ties randomly so that the above tests remain distribution-free. Also see Chapter 8 in \citet{MR3221959} for more discussions on how to break ties for rank-based tests.}
\end{remark}

\section{Theoretical analysis}\label{sec:theory}

This section provides theoretical justifications of the tests proposed
in Section \ref{sec:tests}. The section is split into two parts.  The
first part rigorously justifies the proposed asymptotic critical
values.  The second part gives a  power analysis and shows optimality properties. 

\subsection{Size control}

In this section, we derive the limiting distribution of the statistic
$\hat M_n$ under $H_0$.  The below Cram\'{e}r-type moderate deviation
theorem for degenerate U-statistics under a general probability measure is the
foundation of our theory.  There has been a large literature on
deriving the moderate deviation theorem for non-degenerate
U-statistics (see, for example, \citet{MR3498022} for some recent
developments) as well as Berry--Esseen-type bounds for degenerate
U-statistics (see \citet{MR1481126} and \citet{MR3161488} among many).
However, to our knowledge, the literature does not provide  a comparable moderate deviation theorem for degenerate U-statistics. %the following is the first moderate deviation theorem for degenerate U-statistics in the literature, and hence could also be of independent interest. 

\begin{theorem}[Cram\'{e}r-type moderate deviation for degenerate U-statistics]\label{thm:md}
Let $Z_1,\ldots,Z_n$ be (not necessarily continuous) i.i.d.~random
variables with distribution $\Pr_Z$. Consider the U-statistic
\[
\hat U_n=\mbinom{n}{m}^{-1}\sum_{1\le i_1<\cdots<i_m\le n}h(Z_{i_1},\dots,Z_{i_m}),
\]
where the kernel $h(\cdot)$ is symmetric and such that (i) $\lVert
h\rVert_\infty <\infty$, (ii) $h_1(Z_1;\Pr_Z)=0$ almost surely,
and (iii) $h_2(z_1,z_2;\Pr_Z)$ admits the  eigenfunction expansion,
\begin{equation*}
h_2(z_1,z_2; \Pr_Z)=\sum_{v=1}^{\infty}\lambda_v\phi_v(z_1)\phi_v(z_2),
\end{equation*}
with $\lambda_1\ge\lambda_2\ge\cdots\ge0$,
$\Lambda:=\sum_{v=1}^{\infty}\lambda_v\in(0,\infty)$, and
$\sup_v\lVert \phi_v\rVert_\infty <\infty$. 
{We then have, for any sequence of positive scalars $e_n\to 0$,
\begin{equation*}
\lim_{n\to\infty}\sup_{x_n\in[-\Lambda,e_nn^{\theta}]}\left|\frac{\Pr\Big\{\binom{m}{2}^{-1}(n-1)\hat U_n>x_n\Big\}}{\Pr\Big\{\sum_{v=1}^\infty\lambda_v(\xi_v^2-1)>x_n\Big\}}-1\right|=0,
\end{equation*}
}where $\{\xi_v, v=1,2,\ldots\}$ are i.i.d.~standard Gaussian, and $\theta$ is any absolute constant such that
\begin{equation}\label{eq:thetah}
\theta< 
{\sup\Big\{q\in [0,1/3): \;\sum\nolimits_{v>\lfloor
    n^{(1-3q)/5}\rfloor}\lambda_v=O(n^{-q})\Big\}}
\end{equation}
if infinitely many of eigenvalues $\lambda_v$ are nonzero, 
and $\theta=1/3$ otherwise.
% %&\text{if there are infinite nonzero } \{\lambda_v\},\\
% 1/3, &\text{otherwise},
% \end{cases}
% \end{equation}

% \begin{equation}\label{eq:thetah}
% a=\begin{cases}
% {\sup\Big\{q\in [0,1/3): \;\sum_{v>\lfloor
%     n^{(1-3q)/5}\rfloor}\lambda_v=O(n^{-q})\Big\}-\delta},&\text{if
%   infinitely many } \lambda_v \text{ are nonzero},\\
% %&\text{if there are infinite nonzero } \{\lambda_v\},\\
% 1/3, &\text{otherwise},
% \end{cases}
% \end{equation}
% with $\delta>0$ being an arbitrary constant.
\end{theorem}

In Theorem \ref{thm:md}, when there are only finitely many nonzero eigenvalues, \nb{the range $o(n^{1/3})$ is the standard one for Cram\'er-type moderate deviation.} When there are infinitely many nonzero eigenvalues, it is still unclear if the range $o(n^\theta)$ is the best possible one. It is certainly an interesting question to investigate the optimal range for degenerate U-statistics in the future. With the aid of Theorem \ref{thm:md} and combining it with Proposition
\ref{lem:easy2}, we can now show that, under $H_0$, even if
$p$ is exponentially larger than the sample size $n$, our maximum-type
test statistic still weakly converges to the Gumbel distribution
specified in Proposition \ref{lem:easy2}.   Hence, the proposed test $\mathsf{T}_\alpha$ in \eqref{eq:general} can effectively control the size. %But before that, let us first introduce another notation. Consider the kernel $h$ with the corresponding eigenvalues $\lambda_1,\lambda_2,\ldots$ outlined in Assumption \ref{assumption:key}. We then define the value $\theta=\theta(h)$ to be
%\begin{equation*}
%\theta=\begin{cases}
%{\displaystyle \sup\Big\{q\in [0,1/3): \liminf_{n\to\infty} \frac{n^{-q}}{\sum_{v=K+1}^{\infty}\lambda_v}>0,K=\lfloor n^{(1-3q)/5}\rfloor\Big\}-\delta},\mkern-200mu\\
%     &\text{if there are infinite nonzero } \lambda_v,\\
%1/3, &\text{otherwise},
%\end{cases}
%\end{equation*}
%in which $\delta$ is an arbitrarily pre-specified positive absolute constant. Intrinsically, $\theta$ depicts how fast the tail of the series $\{\lambda_v\}$ converges to 0. 

\begin{theorem}[Limiting null distribution]\label{thm:distr}
  Assume $X_1,\ldots,X_p$ are continuous and the independence
  hypothesis $H_0$ holds.  Let $\hat U_{jk}$, $j<k$, have a common
  kernel $h$ that satisfies Assumption \ref{assumption:key}.  Define
  the parameter $\theta$ as in \eqref{eq:thetah}.  Then if $p=p_n$ goes
  to infinity with $n$ such that $\log p=o(n^\theta)$,  it holds for
  {any absolute constant $y\in\mathbb{R}$} that
  \begin{align*}
    &\Pr\Big\{\frac{n-1}{\lambda_1\binom{m}{2}}\max_{j<k}\hat U_{jk}-4\log p-(\mu_1-2)\log\log p+\frac{\Lambda}{\lambda_1}\le y\Big\}\\
    =\; &\exp\Big\{-\frac{2^{\mu_1/2-2}\kappa}{\Gamma(\mu_1/2)}\exp\Big(-\frac{y}{2}\Big)\Big\}+o(1).
  \end{align*}
  Consequently, 
\[
\Pr_{H_0}(\mathsf{T}_{\alpha}=1)=\alpha+o(1),
\]
{where $\Pr_{H_0}$ represents the probability under the null hypothesis $H_0$.}
\end{theorem}

%   , Assumption \ref{assumption:key}, and assuming further
%   that $X_1,\ldots,X_p$ are continuous, $\log p=o(n^a)$ with the
%   parameter $a$ defined in \eqref{eq:thetah}, and $p=p_n$ goes to
%   infinity with $n$, we have, for any $y\in\R$,
% \begin{align*}
%     &\Pr\Big\{\frac{n-1}{\lambda_1\binom{m}{2}}\max_{j<k}\hat U_{jk}-4\log p-(\mu_1-2)\log\log p+\frac{\Lambda}{\lambda_1}\le y\Big\}\\
% =\; &\exp\Big\{-\frac{2^{\mu_1/2-2}\kappa}{\Gamma(\mu_1/2)}\exp\Big(-\frac{y}{2}\Big)\Big\}+o(1),
% \end{align*}
% which implies that 
% \[
% \Pr(T_{\alpha}=1\given H_0)=\alpha+o(1).
% \]
% \end{theorem}

%Theorem \ref{thm:distr} confirms that the tests we proposed for testing $H_0$ can all asymptotically control the size. In addition,  o

% In Theorem \ref{thm:distr} we require $\mX$ to have continuous margins
% so that the properties outlined in Assumption \ref{assumption:key} may
% apply.
{Note that the proof of Theorem \ref{thm:distr} uses the Chen-Stein
  method, via Theorem~1 of \citet{MR972770}, which is able to handle our case where the random variables are not mutually independent.}
We emphasize that our theory holds without any {distributional
assumption on $\mX$ except for marginal continuity}. This property of being distribution-free  {in the class of multivariate distributions with continuous margins} is
essentially shared by all rank-based correlation measures, but is
clearly not satisfied by other measures like linear or
distance covariance as was illustrated, for example,  by
\citet{MR2052906} and \citet{MR3798874}.   

As a simple consequence of Theorem \ref{thm:distr}, the following corollary shows that the tests in Examples \ref{eg:hd2}--\ref{eg:bd2} have asymptotically correct sizes, with $\theta$ being explicitly calculated.

\begin{corollary}\label{crl:size_eg}
Let $X_1,\ldots,X_p$ be continuous.  Let $p$ go to infinity with $n$
in such a way that $\log p=o(n^{1/8-\delta})$ for some arbitrarily
small pre-specified constant $\delta>0$.  Then
\begin{align*}
&\Pr_{H_0}(\mathsf{T}_{D,\alpha}=1)=\alpha+o(1),~~~
\Pr_{H_0}(\mathsf{T}_{R,\alpha}=1)=\alpha+o(1), & \\
\text{and}~~~
&\Pr_{H_0}(\mathsf{T}_{\tau^*,\alpha}=1)=\alpha+o(1).
\end{align*}
\end{corollary}

\subsection{Power analysis and rate-optimality}

We now investigate the power of the proposed tests from an asymptotic
minimax perspective.  The key ingredient is the choice of a
suitable distribution family as an alternative to the null hypothesis
in \eqref{eq:H0}. Recall the definition of $h^{(1)}(\cdot)$ in
\eqref{eq:han-hl}. For any kernel function $h(\cdot)$ and constants
$\gamma>0$ and $q\in\Z^+$, define a general $q$-dimensional (not necessarily continuous) distribution family as follows:
\begin{equation*}
\mathcal{D}(\gamma,q;h)\!:=\!\Big\{\mathcal{L}(\mX): \mX\in\mathbb{R}^q, \Var_{jk}\{h^{(1)}(\cdot;\Pr_{jk})\}\leq \gamma \E_{jk} h  ~~\text{for all }j\ne k\in[q] \Big\},
\end{equation*}
where $\mathcal{L}(\mX)$ is the distribution (law) of $\mX$, and $\Pr_{jk}$, $\E_{jk}(\cdot)$, and $\Var_{jk}(\cdot)$ stand for the probability measure, expectation, and variance operated on the bivariate distribution of $(X_j,X_k)^\top$, respectively. 

%The family $\mathcal{D}(\gamma,q;h)$ intrinsically characterizes the
%slope of the function $\Var_{jk}\{h^{(1)}(\cdot;\Pr_{jk})\}$ with
%regard to the dependence between $X_j$ and $X_k$ characterized by the ``correlation measure'' $\E_{jk}h$. Furthermore, once
%$X_j$ is independent of $X_k$, $\Pr_{jk}$ reduces to
%$\Pr_j\otimes\Pr_k$, the product measure of the two marginal measures
%$\Pr_j, \Pr_k$.  If $\Pr_j$ and  $\Pr_k$ are further continuous, the according variance is
%\begin{align}\label{eq:interesting}
%\Var_{jk}\{h^{(1)}(\cdot;\Pr_{j}\otimes\Pr_k)\}=0 = \E_{jk} h,
%\end{align}
%provided that $h(\cdot)$ satisfies the properties in Assumption
%\ref{assumption:key}. Therefore, in view of \eqref{eq:interesting}, it
%is reasonable  to expect that $\mathcal{D}(\gamma,q;h)$, for large
%enough $\gamma$, contains many interesting continuous multivariate distributions.

The family $\mathcal{D}(\gamma,q;h)$ intrinsically characterizes the
slope of the function $\Var_{jk}\{h^{(1)}(\cdot;\Pr_{jk})\}$ with
regard to the dependence between $X_j$ and $X_k$, characterized by the
``correlation measure'' $\E_{jk}h$.  {Intuitively, consider
  $\mathbb{E}_{jk}h$ as a rank correlation measure of dependence
  between $X_j$ and $X_k$.   When $X_j$ is independent of $X_k$, we have that
\[
{\rm Var}_{jk}\{h^{(1)}(\cdot;\mathbb{P}_{j}\otimes\mathbb{P}_k)\}=0 = \mathbb{E}_{jk} h
\]
as long as Assumption 2.1 holds for $h(\cdot)$. Therefore,
heuristically, as the dependence between $X_j$ and $X_k$ increases, it
is possible that the variance
{${\rm Var}_{jk}\{h^{(1)}(\cdot;\mathbb{P}_{jk})\}$}
will deviate from 0 at the same or a slower rate compared to 
$\mathbb{E}_{jk} h$.  Note that both parameters are nonnegative.
The next lemma firms up this intuition by establishing that the Gaussian family belongs to $\mathcal{D}(\gamma,q;h)$ for
all the kernels $h(\cdot)$ considered in Examples \ref{eg:hd} to \ref{eg:bd},
provided $\gamma$ is large enough.}
 
\begin{lemma}\label{lem:normal}
  There exists an absolute constant $\gamma>0$ such that for all
  $q\in\mathbb{Z}^+$, any $q$-dimensional Gaussian distribution is in
  $\mathcal{D}(\gamma,q;h_D)$, $\mathcal{D}(\gamma,q;h_R)$, and
  $\mathcal{D}(\gamma,q;h_{\tau^*})$.
\end{lemma}

Next we introduce a class of matrices indexed by a
positive constant $C$ as 
\begin{equation*}
\mathcal{U}_p(C):= \Big\{ \fM\in \R^{p\times p}:  \max_{j< k}\{M_{jk}\}\geq C(\log p/n)\Big\}.
\end{equation*}
Such matrices will define a ``sparse local alternative'' as considered also
in Section 4.1 in \citet{MR3737306}.   Note, however, that in our case
the scale is at the order of $\log p/n$ as opposed to  $(\log
p/n)^{1/2}$ in \citet{MR3737306}.  This is due to our statistics being
degenerate under independence. Hence, the variance
of $h^{(1)}(\cdot)$ is
zero under the null, while nonzero for these statistics investigated in \citet{MR3737306}. {It should also be noted that these two classes cannot
be directly compared;  intuitively the consistent measures are defined
on a squared scale when contrasted to the non-consistent measures. As will be shown later,
in the example of the Gaussian case, both classes correspond to a
condition on the Pearson correlation obeying the rate $(\log p /n)^{1/2}$. }

The following theorem now describes ``local alternatives'' under which
the power of our general test $\mathsf{T}_{\alpha}$ tends to one as {both $n$ and $p$ go
to infinity.}  

\begin{theorem}[Power analysis, general]\label{thm:power}
Given any $\gamma>0$ and a kernel $h(\cdot)$ satisfying Assumption \ref{assumption:key}, there exists some sufficiently large $C_{\gamma}$ depending on $\gamma$ such that
\begin{equation*}
\liminf_{n,p\to\infty}\inf_{\fU\in\mathcal{U}_p(C_{\gamma})}\P_{\fU}(\mathsf{T}_\alpha=1)=1,
\end{equation*}
where, for each specified $(n,p)$, the infimum is taken over all
distributions in $\mathcal{D}(\gamma,p;h)$ that have the matrix of
population dependence coefficients $\fU=[U_{jk}]$ in
$\mathcal{U}_p(C_{\gamma})$.  Here, $U_{jk}:=\E \hat U_{jk}$.
\end{theorem}

The proof of Theorem \ref{thm:power} only uses the Hoeffding decomposition for U-statistics, Bernstein's inequality
for the sample mean part, and Arcones and Gin\'e's inequality for the degenerate
U-statistics parts \citep{MR1235426}.  Consequently, we do not have to assume any continuity of $\mX$. The theorem immediately yields the following corollary, characterizing the local alternatives under which the three rank-based tests from Examples \ref{eg:hd2}--\ref{eg:bd2} have power tending to 1.

\begin{corollary}[Power analysis, examples]\label{crl:power_eg}
Given any $\gamma>0$, we have, for some sufficiently large $C_{\gamma}$ depending on $\gamma$,
\begin{align*}
&\liminf_{n,p\to\infty}\inf_{\fD\in\mathcal{U}_p(C_\gamma)}\P_{\fD}(\mathsf{T}_{D,\alpha}=1)=1,~~~
 \liminf_{n,p\to\infty}\inf_{\fR\in\mathcal{U}_p(C_\gamma)}\P_{\fR}(\mathsf{T}_{R,\alpha}=1)=1,\\
&\liminf_{n,p\to\infty}\inf_{\fT^*\in\mathcal{U}_p(C_\gamma)}\P_{\fT^*}(\mathsf{T}_{\tau^*,\alpha}=1)=1,
\end{align*}
where, for each specified $(n,p)$, the infima are taken over all
distributions in $\mathcal{D}(\gamma,p;h_D)$,
$\mathcal{D}(\gamma,p;h_R)$, and $\mathcal{D}(\gamma,p;h_{\tau^*})$ with
population dependence coefficient matrices $\fD=[D_{jk}]$,
$\fR=[R_{jk}]$, and $\fT^*=[\tau^*_{jk}]$ for $D_{jk}:=\E \hat D_{jk}$,
$R_{jk}:=\E \hat R_{jk}$, and $\tau^*_{jk}:=\E \hat\tau^*_{jk}$,
respectively.
\end{corollary}

%In the following we give justification on the rate-optimality of the proposed tests. 
We now turn to optimality of the proposed tests. There have been long
debates on the power of consistent rank-based tests compared to those
based on linear and simple rank correlation measures.  As a matter of fact, %Hoeffding himself has commented on this topic, stating that, if the Gaussian model is correct, then the power of Hoeffding's $D$-based test would be lower than those based on Pearson's $r$ (cf. Chapter 10 in \citet{hoeffding1948class}). %More recent results comparing the empirical powers of different tests under different distributions can be found in \citep{mudholkar2003conventional}. 
% Blum, Kiefer, and Rosenblatt
\citet{MR0125690} have given interesting comments on this topic,
stating that the required sample size for the bivariate independence
test based on $h_R(\cdot)$ is of the same order as that in common
parametric cases, hinting that even under a particular parametric model these
nonparametric consistent tests of independence can be as rate-efficient as
tests that specifically target the considered model. {\citet{MR3766953} and \citet{MR3737306}, among many others, derived rate-optimality results for rank-based tests. However, their results do not cover those that permit consistent assessment of pairwise independence.} %However, as remarked in \citet{mudholkar2003conventional}, beyond empirical investigations, still little is known about the power properties of these consist tests even in low dimensions. 
Recently, \citet{MR3798874} %derived another interesting result,
% showing
made a first step towards a minimax optimality result for consistent
tests of independence.  Their result shows an infeasible version of a
test based on distance covariance to be rate-optimal against a
Gaussian dense alternative.  However, it remained an open question if
there exists a feasible (consistent) test of mutual  independence in
high dimensions that is rate-optimal against certain alternatives.
Below we are able to  give an affirmative answer.

% showed that a test based on the distance covariance %statistic
% is actually rate-optimal against a Gaussian dense alternative. This is to our knowledge the first minimax optimality result on (consistent) tests of independence. It suggests that, in terms of power and regarding rates, it is possible that those nonparametric (consistent) tests of independence are comparable to their parametric counterparts. However, the test discussed in \citet{MR3798874} is infeasible and it is still an open question if there exists a feasible (consistent) test of mutual  independence in high dimensions that is rate-optimal against certain alternative. This paper gives an affirmative answer.

% For this, we
We shall focus on the proposed tests in Examples
\ref{eg:hd2}--\ref{eg:bd2} and show their rate-optimality in the
Gaussian model.  To this end, we define a new alternative class of
matrices 
\begin{align*}
\mathcal{V}(C)\!:=\! \Big\{ \fM\!\in\! \R^{p\times p}\!\!: \fM \succeq 0,  {\rm diag}(\fM)\!=\!\fI_p, \fM\!=\!\fM^\top, \max_{j\ne k}|M_{jk}|\geq C\sqrt{\frac{\log p}{n}}\Big\},
\end{align*}
where $\fM \succeq 0$ denotes positive semi-definiteness.  We
then have the following theorem as a consequence of Corollary
\ref{crl:power_eg}. It concerns the proposed tests' power under a
Gaussian model with  some nonzero pairwise correlations but for which
these  are decaying to zero as the sample size increases, {and is immediate from the fact that, as $(X_j,X_k)^\top$ is bivariately normal with correlation $\rho_{jk}$, we have 
\[D_{jk}, R_{jk}, \tau_{jk}^*\asymp \rho_{jk}^2~~{\rm as}~\rho_{jk}\to 0.\]
}%
Since the test statistics are all rank-based and thus invariant to monotone marginal transformations, extension of the following result to the corresponding Gaussian copula family with continuous margins is straightforward.
\begin{theorem}[Power analysis, Gaussian]\label{crl:normal}
For a sufficiently large absolute constant $C_0>0$, we have, {as long as $n,p\to\infty$},
\begin{align*}
&\inf_{\mSigma\in\mathcal{V}(C_0)}\P_{\mSigma}(\mathsf{T}_{D,\alpha}=1)=1-o(1), ~~~
 \inf_{\mSigma\in\mathcal{V}(C_0)}\P_{\mSigma}(\mathsf{T}_{R,\alpha}=1)=1-o(1), \\
\text{and}~~~
&\inf_{\mSigma\in\mathcal{V}(C_0)}\P_{\mSigma}(\mathsf{T}_{\tau^*,\alpha}=1)=1-o(1), 
\end{align*}
where infima are over centered Gaussian distributions
with (Pearson) covariance matrix $\mSigma=[\Sigma_{jk}]$.
\end{theorem}

The proof of Theorem \ref{crl:normal} is given in the supplement.  It relies on Lemma \ref{lem:normal}
and the fact that $D_{jk}, R_{jk}, \tau^*_{jk}\asymp \Sigma_{jk}^2$ as
$\Sigma_{jk}\to 0$. 
Combined with the following result from \citet{MR3737306}, Theorem
\ref{crl:normal} yields minimax rate-optimality of the tests in
Examples \ref{eg:hd2}--\ref{eg:bd2} against the sparse Gaussian alternative. 

\begin{theorem}[Rate optimality, Theorem 5 in \citeauthor{MR3737306},
  \citeyear{MR3737306}]\label{thm:opt}
  There exists an absolute constant $c_0>0$ such that for any number
  $\beta>0$ satisfying $\alpha+\beta<1$, in any asymptotic regime with
  $p\to\infty$ as $n\to\infty$ but $\log p/n=o(1)$, it holds for all
  sufficiently large $n$ and $p$ that
  \[
    \inf_{\overline{\mathsf{T}}_{\alpha}\in\cT_{\alpha}}\sup_{\mSigma\in\mathcal{V}(c_0)}\Pr_{\mSigma}(\overline{\mathsf{T}}_{\alpha}=0)\geq 1-\alpha-\beta.
  \]
  Here the infimum is taken over all size-$\alpha$ tests, and the
  supremum is taken over all centered Gaussian distributions with
  (Pearson) covariance matrix $\mSigma$.
\end{theorem}

\section{Simulation studies}\label{sec:simulation}

In this section we compare the finite-sample performance of the three 
tests (Extreme $D$, Extreme $R$, and Extreme $\tau^*$) from Section
\ref{sec:tests} to eight existing tests proposed in the literature  via Monte Carlo simulations. The first eight tests are rank-based and hence distribution-free {in the class of multivariate distributions with continuous margins}, while the other three tests are distribution-dependent:
\begin{itemize}[label={},itemindent=-2.5em]
{\item DHS$_D$: the maximum-type test in Example \ref{eg:hd2};
\item DHS$_R$: the maximum-type test in Example \ref{eg:bkr2};
\item DHS$_{\tau^*}$: the maximum-type test in Example \ref{eg:bd2};}
\item LD$_\tau$: the $L_2$-type test based on Kendall's $\tau$ \citep{MR3766953};
\item LD$_\rho$: the $L_2$-type test based on Spearman's $\rho$ \citep{MR3766953};
\item LD$_{\tau^*}$: the $L_2$-type test based on Bergsma--Dassios--Yanagimoto's $\tau^*$ \citep{MR3766953}; 
\item HCL$_\tau$: the maximum-type test based on Kendall's $\tau$ \citep{MR3737306};
\item HCL$_\rho$: the maximum-type test based on Spearman's $\rho$ \citep{MR3737306};
\item YZS: the $L_2$-type test based on the distance covariance statistic \citep{MR3798874};
\item SC: the $L_2$-type test based on Pearson's $r$ \citep{MR2234197};
\item CJ: the maximum-type test based on Pearson's $r$ \citep{MR2850210}.
\end{itemize}

\subsection{Computational aspects}

% We discuss computation of the discussed tests in this section.
Throughout this section %we are focused on a general bivariate sample
$\{\mz_i=(z_{i,1},z_{i,2})^\top\}_{i\in[n]}$ %be
is a bivariate sample
that contains no tie.
% , %, which means that $X_{1,j}\ldots,X_{n,j}$ are distinct and so are $X_{1,k},\ldots,X_{n,k}$. This 
% a condition that is naturally satisfied under our settings since $\mX$ is assumed  to have continuous margins. 
We first discuss how to compute the U-statistics $\hat D$, $\hat R$,
and $\hat\tau^*$ for Hoeffding's $D$, Blum--Kiefer--Rosenblatt's $R$, and
Bergsma--Dassios--Yanagimoto's $\tau^*$, respectively.   As we review
below, efficient algorithms
are available for $\hat D$ and $\hat\tau^*$.  The value of $\hat R$ may
then be found using the relation in \eqref{eq:identity2}.

% In view of
% %\eqref{eq:identity}
% \eqref{eq:identity2}, it suffices to compute any two of the three to
% determine all.  We will thus compute $\hat D$ and $\hat\tau^*$,
% % the statistics for Hoeffding's $D$ and
% % Bergsma--Dassios--Yanagimoto's $\tau^*$,
% for which efficient algorithms
% are available.
% % than Blum--Kiefer--Rosenblatt's $R$, and hence we only have to compute Hoeffding's $D$ and Bergsma--Dassios--Yanagimoto's $\tau^*$.

%For computing Hoeffding's $D$ statistic $\hat D$,
%For efficient computation of $\hat D$,
\citet{MR0029139} himself
observed that $\hat D$ can be computed in $O(n\log n)$ time via the following formula
\begin{equation*}
\frac{\hat D}{30}=\frac{P-2(n-2)Q+(n-2)(n-3)S}{n(n-1)(n-2)(n-3)(n-4)}.
\end{equation*}
Here
\begin{align*}
P&:=\sum_{i=1}^n (r_i-1)(r_i-2)(s_i-1)(s_i-2),\\
Q&:=\sum_{i=1}^n (r_i-2)(s_i-1)c_i,~~~
S:=\sum_{i=1}^n c_i(c_i-1),
\end{align*}
and $r_i$ and $s_i$ are the ranks of $z_{i,1}$ among
$\{z_{1,1},\ldots,z_{n,1}\}$ and $z_{i,2}$ among
$\{z_{1,2},\ldots,z_{n,2}\}$, respectively.   Moreover, $c_i$ is the
number of pairs $\mz_{i'}$ for which $z_{i',1}<z_{i,1}$ and
$z_{i',2}<z_{i,2}$.
%The according time complexity is thus $O(n\log n)$.
%Recently, \citet{MR3842884} developed this algorithm of computing Hoeffding's $D$ using range-tree method, whose asymptotic run-time is $O(n\log n)$. %This results in total $O(p^2n\log n)$ time in computing $\max_{1\le j<k\le p}\hat D_{jk}$. 

\citet{MR3481807} and \citet{heller2016computing} proposed algorithms
for efficient computation of the Bergsma--Dassios--Yanagimoto statistic
$\hat\tau^*$.  % $\tau^*$ in $O(n^2)$ time.
% We sketch their idea below. Since there is almost surely no tie,
Without loss of generality, let $z_{1,1}<\cdots<z_{n,1}$, i.e.,
$r_i=i$. \citet{MR3481807} proved that 
$2\hat\tau^*/3=N_c/\binom{n}{4}-1/3$ with
\begin{equation*}
N_c=\sum_{3\le \ell<\ell'\le n}\mbinom{\fB_{<}[\ell,\ell']}{2}+\mbinom{\fB_{>}[\ell,\ell']}{2},
\end{equation*}
where for all $\ell<\ell'$,
\begin{align*}
\fB_{<}[\ell,\ell']&:=\#\{i:i\in[\ell-1], z_{i,2}<\min(z_{\ell,2},z_{\ell',2})\}\\
\text{and}~~~
\fB_{>}[\ell,\ell']&:=\#\{i:i\in[\ell-1], z_{i,2}\mkern-1mu>\mkern-1mu\max(z_{\ell,2},z_{\ell',2})\}.
\end{align*}
\citet{MR3481807} went on to give an algorithm to compute these
counts, and thus $\hat\tau^*$, in
$O(n^2\log n)$ time with little
memory use.  \citet{heller2016computing} showed that the computation
time can be further lowered to $O(n^2)$ via calculation of  the following
matrix based on the empirical %cumulative
distribution of the ranks $r_i$ and $s_i$,
\begin{equation*}
\fB[r,s]:=\sum_{i=1}^n \ind(r_i\le r,s_i\le s),~~~0\le r,s\le n.
\end{equation*}
Here, $\fB[r,0]:=0$ and $\fB[0,s]:=0$.   We may then find
$\fB_{<}[\ell,\ell']=\fB[\ell-1,\min(s_\ell,s_{\ell'})-1]$ and
$\fB_{>}[\ell,\ell']=\ell-\fB[\ell,\max(s_\ell,s_{\ell'})]$ for all
$\ell<\ell'$; recall that $s_i$ is the rank of $z_{i,2}$ in
$\{z_{1,2},\ldots,z_{n,2}\}$.
% Therefore, $\hat\tau^*$ can be calculated in $O(n^2)$ time.
As a consequence, formula \eqref{eq:identity2} now also yields an
$O(n^2)$ algorithm for $\hat R$.
% in $O(n^2)$ time directly by the formula \eqref{eq:identity2}.

Regarding  other competing statistics, note that %By the argument
                                %above, HL$_{\tau^*}$ can be calculated
                                %in $O(n^2)$ as well.
Pearson's $r$ and Spearman's $\rho$ can be naively computed in time
$O(n)$ and $O(n\log n)$, respectively.  \citet{zbMATH03224091} proposed
an efficient algorithm for computing Kendall's $\tau$ that has time
complexity $O(n\log n)$.  Finally, 
the algorithm of \citet{MR3556612} computes the distance covariance
statistic in $O(n\log n)$ time.  

{Table \ref{tab:time} shows empirical computation times for the
  considered statistics on 1,000 bivariate samples of size $n=100,
  200, 400$, and 800, respectively randomly generated as
  i.i.d.~standard bivariate normal.  The timings are based on
  available functions in \texttt{R}.  Pearson's $r$ and
  Spearman's $\rho$ were computed using the basic \texttt{cor()}
  function, with option \texttt{method="spearman"} for $\rho$.
  Kendall's $\tau$ was computed with the function \texttt{cor.fk()}
  from package \texttt{pcaPP},  
  Hoeffding's $D$ with \texttt{hoeffD()} from
  \texttt{SymRC},  Bergsma--Dassios--Yanagimoto's $\tau^*$ with
  \texttt{tStar()} from \texttt{TauStar}, and the distance
  covariance with \texttt{dcov2d()} from \texttt{energy}.
  Blum--Kiefer--Rosenblatt's $\hat R$ was then obtained using identity \eqref{eq:identity2}, and its computation time is thus omitted.
  All experiments are conducted on a laptop with 
  a 2.6 GHz Intel Core i5 processor and a 8 GB memory.
}

{
\renewcommand{\tabcolsep}{1.5pt}
\renewcommand{\arraystretch}{1.0}
\begin{table}[t]
\centering
\caption{A comparison of computation time for all the correlation statistics considered. 
The computation time here is the averaged elapsed time (in milliseconds) of 1,000 replicates of a single experiment.}
\label{tab:time}{
\footnotesize
\begin{tabular}{cC{0.7in}C{0.5in}
                 C{0.6in}C{0.7in}C{0.6in}
                 C{0.6in}}
$n$  &  Hoeffding's $D$  & BDY's $\tau^*$ & 
        Pearson's $r$ & Spearman's $\rho$ & Kendall's $\tau$ & 
        distance correlation \\
100  &  0.270  & 0.167 & 0.060 & 0.121 & 0.064 & 0.667 \\
200  &  0.962  & 0.543 & 0.080 & 0.144 & 0.085 & 1.194 \\
400  &  4.419 & 2.364 & 0.099 & 0.206 & 0.106 & 2.313 \\
800  &  9.683  & 20.860 & 0.103 & 0.327 & 0.148 & 4.410 \\
\end{tabular}} 
%Results are averaged over $5000$ simulated data sets.
\end{table} 
}

{While the above statistics can all be computed efficiently using
  special purpose algorithms, our theory also covers
  general rank-based statistics for which only a naive algorithm that
  follows the U-statistic definition may be available.  
  The complexity of computing the  statistic could then be  a high
  degree polynomial of the sample size.   We note that in this case,
  it may become necessary to use resampling and subsampling techniques
  to decrease computational effort, as was done by \citet[Section 4]{MR3178526} when applying their statistics before
  efficient algorithms for its computation were developed.

}
%, and Kendall's $\tau$ are $O(n)$, $O(n\log n)$, and $O(n\log n)$  respectively. 

\subsection{Simulation results}

We evaluate the empirical sizes and powers of the eleven competing tests
introduced above for both Gaussian and non-Gaussian distributions.
The values reported below are based on $5,000$ simulations at the
nominal significance level of $0.05$, with sample size
$n\in\{100,200\}$ and dimension $p\in\{50,100,200,400,800\}$.  All
data sets are generated as an i.i.d.~sample from the distribution
specified for the $p$-dimensional random vector $\mX$.

We investigate the sizes of the tests in four settings, where 
$\mX=(X_1,\ldots,X_p)^\top$ has mutually independent entries.  {In the following, with slight abuse of notation, we write $f(\mv)=(f(v_1),\ldots,f(v_p))^\top$ for any univariate function $f:\mathbb{R}\to\mathbb{R}$ and $\mv=(v_1,\ldots,v_p)^\top\in\mathbb{R}^p$. } 

\begin{example}\label{eg:sim-size} 
  \mbox{ }
% $\mX=(X_1,\ldots,X_p)^\top\in\R^p$ is generated as follows such that $X_1,\ldots,X_p$ are mutually independent.
\begin{enumerate}[label=(\alph*)]
\item\label{sim:1a}
  %The data are generated from a standard Gaussian distribution:
  $\mX\sim N_p(0,\fI_p)$  (standard Gaussian).
\item\label{sim:1b}
  % The data are generated from a Gaussian copula family:
  $\mX=\mW^{1/3}$ with $\mW\sim N_p(0,\fI_p)$  (light-tailed Gaussian copula).
\item\label{sim:1c}
  % The data are generated from a Gaussian copula family:
  $\mX=\mW^3$ with $\mW\sim N_p(0,\fI_p)$  (heavy-tailed Gaussian copula).
\item\label{sim:1d} %($t$-distribution)
  % The components $X_1,\ldots,X_p$ are independent and identically distributed as Student's $t$-distribution with degrees of freedom $3$.
  $X_1,\ldots,X_p$ are i.i.d.~with a $t$-distribution with 3 degrees of
  freedom.
\end{enumerate}
\end{example}
The simulated sizes of the eight {rank-based} tests are reported in Table~\ref{tab:size-1}. Those of the three distribution-dependent tests are given in Table~\ref{tab:size-2}.
As expected, the tests derived from Gaussianity (SC, CJ) fail to
control the size for heavy-tailed distributions.  In contrast, the
other tests control the size effectively in most circumstances.  A
slight size inflation is observed for DHS$_D$ at small sample size, which can be addressed using Monte Carlo approximation to set the critical value.
In addition, when considering different pairs of $(n,p)$ in Table \ref{tab:size-1}, as long as $n$ and $p$ grow simultaneously, a trend to the nominal level 0.05 is clear; e.g., as $(n,p)$ grows from $(100,200)$ to $(200,400)$, the empirical size of ${\rm DHS}_D$ changes from 0.076 to 0.064, that of ${\rm DHS}_R$ changes from 0.028 to 0.040, and that of ${\rm DHS}_{\tau^*}$ changes from 0.036 to 0.045. These phenomena back up Corollary \ref{crl:size_eg}, and this trend persists in more simulations as $n$ and $p$ become even larger.

{
\renewcommand{\tabcolsep}{2.5pt}
\renewcommand{\arraystretch}{1.0}
\begin{table}[t]
\centering
\caption{Empirical sizes of the eight rank-based tests in Example \ref{eg:sim-size}}
\label{tab:size-1}{
\footnotesize
\begin{tabular}{ccC{0.375in}C{0.375in}C{0.375in}
                  C{0.35in}C{0.35in}C{0.35in}
                  C{0.35in}C{0.35in}}
$n$ & $p$ &  DHS$_D$ & DHS$_R$ & DHS$_{\tau^*}\!$ & 
             LD$_{\tau}$ & LD$_{\rho}$ & LD$_{\tau^*}$ & 
             HCL$_{\tau}$ & HCL$_{\rho}$\\
%    &     &  \multicolumn{8}{c}{Results for Example \ref{eg:sim-size}\ref{sim:1a}} \\
100 &  50 & 0.070 & 0.042 & 0.047 & 0.054 & 0.048 & 0.056 & 0.037 & 0.028\\
    & 100 & 0.073 & 0.035 & 0.042 & 0.055 & 0.047 & 0.066 & 0.034 & 0.021\\
    & 200 & 0.076 & 0.028 & 0.036 & 0.058 & 0.050 & 0.059 & 0.028 & 0.015\\
    & 400 & 0.084 & 0.025 & 0.035 & 0.054 & 0.045 & 0.065 & 0.025 & 0.012\\
    & 800 & 0.088 & 0.021 & 0.032 & 0.055 & 0.049 & 0.062 & 0.023 & 0.008\\
200 &  50 & 0.054 & 0.042 & 0.044 & 0.048 & 0.044 & 0.051 & 0.037 & 0.034\\
    & 100 & 0.057 & 0.042 & 0.044 & 0.052 & 0.047 & 0.052 & 0.038 & 0.032\\
    & 200 & 0.059 & 0.038 & 0.042 & 0.052 & 0.050 & 0.055 & 0.037 & 0.032\\
    & 400 & 0.064 & 0.040 & 0.045 & 0.051 & 0.048 & 0.053 & 0.038 & 0.027\\
    & 800 & 0.065 & 0.034 & 0.040 & 0.051 & 0.047 & 0.055 & 0.034 & 0.024\\
\end{tabular}}
%Results are averaged over $5000$ simulated data sets.
\end{table}
}

{
\renewcommand{\tabcolsep}{0.5pt}
\renewcommand{\arraystretch}{1.0}
\begin{table}[t]
\centering
\caption{Empirical sizes of the three distribution-dependent tests in Example \ref{eg:sim-size}}
\label{tab:size-2}{
\footnotesize
\begin{tabular}{ccC{0.35in}C{0.35in}C{0.35in}c
                  C{0.35in}C{0.35in}C{0.35in}c
                  C{0.35in}C{0.35in}C{0.35in}c
                  C{0.35in}C{0.35in}C{0.35in}}
$n$ & $p$ &  YZS & SC & CJ && 
             YZS & SC & CJ && 
             YZS & SC & CJ && 
             YZS & SC & CJ \\
    &     &  \multicolumn{3}{c}{Results for Case \ref{sim:1a}} &&
             \multicolumn{3}{c}{Results for Case \ref{sim:1b}} &&
             \multicolumn{3}{c}{Results for Case \ref{sim:1c}} &&
             \multicolumn{3}{c}{Results for Case \ref{sim:1d}} \\
100 &  50 & 0.048 & 0.051 & 0.029 && 0.052 & 0.052 & 0.036 && 0.055 & 0.210 & 0.974 && 0.055 & 0.081 & 0.479\\
    & 100 & 0.054 & 0.052 & 0.018 && 0.048 & 0.047 & 0.032 && 0.052 & 0.206 & 1.000 && 0.053 & 0.083 & 0.781\\
    & 200 & 0.059 & 0.051 & 0.013 && 0.055 & 0.055 & 0.024 && 0.052 & 0.207 & 1.000 && 0.058 & 0.089 & 0.974\\
    & 400 & 0.049 & 0.049 & 0.011 && 0.053 & 0.051 & 0.022 && 0.052 & 0.210 & 1.000 && 0.055 & 0.089 & 1.000\\
    & 800 & 0.050 & 0.045 & 0.005 && 0.050 & 0.048 & 0.018 && 0.055 & 0.222 & 1.000 && 0.051 & 0.092 & 1.000\\
200 &  50 & 0.050 & 0.044 & 0.032 && 0.050 & 0.052 & 0.040 && 0.054 & 0.194 & 0.955 && 0.050 & 0.086 & 0.527\\
    & 100 & 0.049 & 0.049 & 0.029 && 0.049 & 0.051 & 0.036 && 0.048 & 0.190 & 1.000 && 0.052 & 0.089 & 0.850\\
    & 200 & 0.053 & 0.049 & 0.030 && 0.052 & 0.053 & 0.035 && 0.055 & 0.193 & 1.000 && 0.050 & 0.085 & 0.996\\
    & 400 & 0.051 & 0.049 & 0.022 && 0.050 & 0.048 & 0.035 && 0.050 & 0.193 & 1.000 && 0.050 & 0.091 & 1.000\\
    & 800 & 0.050 & 0.053 & 0.018 && 0.051 & 0.053 & 0.033 && 0.052 & 0.188 & 1.000 && 0.049 & 0.088 & 1.000\\
\end{tabular}}
%Results are averaged over $5000$ simulated data sets.
\end{table}
}

In order  to study the power properties of the different statistics,
we consider three sets of examples.  %that mimics the corresponding ones from \citet{MR3798874} and \citet{MR3737306}. The first set of examples is extracted from \citet{MR3798874}.
%In the first set (Example \ref{eg:sim-power1}), the signal is {relatively} dense.  The other two sets of examples focus on sparse settings.
We remark that, regarding the power, for $L_2$-type and
  maximum-type tests, one cannot dominate the other; compare the power
  analyses in Section 3.3 in \citet{MR3174618} and Section 5.2 in
  \citet{MR3766953}. To reflect this, we consider two sets of examples that
  focus on relatively sparse settings (modified based on \citet{MR3798874} and \citet{MR3737306}) but also include a very dense third
  setup \nb{(modified based on \citet{MR3766953} with an adjustment to dimension as suggested in \citet[Theorems~1 and 4]{MR3160557})}.

%The remaing two types of examples focus on sparse settings.
% The first is a modification to the first one by adding more independent random variables in order to make the pairs of dependent random variables sparser.

\begin{example}\label{eg:sim-power2}
    \mbox{ }
\begin{enumerate}[label=(\alph*)]
%\item\label{sim:3a} The data are constructed as $\mX=(\mX_1^\top,\mX_2^\top)^\top$, where 
%\[
%\mX_1=(\momega^\top,\sin(2\pi\momega)^\top,\cos(2\pi\momega)^\top,\sin(4\pi\momega)^\top,\cos(4\pi\momega)^\top)^\top\in\R^{p_1}
%\]
%in which $\momega\sim N_{p_1/5}(0,\fI_{p_1/5})$ with $p_1=p/5$, and $\mX_2\sim N_{p-p_1}(0,\fI_{p-p_1})$ are independent of $\mX_1$.
\item\label{sim:3a'} The data are generated as $\mX=(\mX_1^\top,\mX_2^\top)^\top$, where 
\[
\mX_1=(\momega^\top,\sin(2\pi\momega)^\top,\cos(2\pi\momega)^\top,\sin(4\pi\momega)^\top,\cos(4\pi\momega)^\top)^\top\in\R^{10}
\]
with $\momega\sim N_{2}(0,\fI_{2})$, and $\mX_2\sim
N_{p-10}(0,\fI_{p-10})$ independent of $\mX_1$.
%\item\label{sim:3b} The data are constructed as $\mX=(\mX_1^\top,\mX_2^\top)^\top$, where 
%\[
%\mX_1=(\momega^\top,\log(\momega^2)^\top)^\top\in\R^{p_1}
%\]
%in which $\momega\sim N_{p_1/2}(0,\fI_{p_1/2})$ with $p_1=p/5$, and $\mX_2\sim N_{p-p_1}(0,\fI_{p-p_1})$ are independent of $\mX_1$.
\item\label{sim:3b'} The data are generated as $\mX=(\mX_1^\top,\mX_2^\top)^\top$, where 
\[
\mX_1=(\momega^\top,\log(\momega^2)^\top)^\top\in\R^{10}
\]
with $\momega\sim N_{5}(0,\fI_{5})$, and $\mX_2\sim
N_{p-10}(0,\fI_{p-10})$ independent of $\mX_1$.
\end{enumerate}
\end{example}

% The last set of examples depicts another  sparse setting where only a small portion of the pairs are not independent. 

\begin{example}\label{eg:sim-power3}
    \mbox{ }
\begin{enumerate}[label=(\alph*)]
\item\label{sim:4a} The data are drawn as $\mX\sim N_p(0,\fR^*)$ with
  $\fR^*$ generated as follows: Consider a random matrix $\mDelta$
  with all but eight random nonzero entries. We select the locations
  of four nonzero entries randomly from the upper triangle of
  $\mDelta$, each with a magnitude randomly drawn from the uniform
  distribution in $[0, 1]$. The other four nonzero entries in the
  lower triangle are determined to make $\mDelta$ symmetric. Finally,
\[
\fR^*=(1+\delta)\fI_p+\mDelta, 
\]
where $\delta=\{-\lambda_{\min}(\fI_p+\mDelta)+0.05\}\cdot \ind\{\lambda_{\min}(\fI_p+\mDelta)\le 0\}$ and $\lambda_{\min}(\cdot)$ denotes the smallest eigenvalue of the input.
\item\label{sim:4b} The data are drawn as $\mX=\sin(2\pi
  \mZ^{1/3}/3)$, where $\mZ\sim N_p(0,\fR^*)$ with $\fR^*$ as in \ref{sim:4a}.
\item\label{sim:4c} The data are drawn as $\mX=\sin(\pi \mZ^3/4)$, where $\mZ\sim N_p(0,\fR^*)$ with $\fR^*$ as in \ref{sim:4a}.
%\item\label{sim:4d} The data $X$ are generated from a multivariate $t$-distribution with mean $0$ and covariance matrix $R^∗$ as defined above. 
\end{enumerate}
\end{example}

\begin{example}\label{eg:sim-power4}
    %\mbox{ }
    The data are drawn as $\mX\sim N_p(0,\fR^*)$, where
  $\fR^*=(1-\varrho)\fI_p+\varrho\fJ_p$ with $\varrho$ such that \nb{
    \begin{enumerate}[label=(\alph*)]      
\item\label{sim:5a}  $\binom{p}{2}(2\arcsin\varrho/\pi)^2=p/n$;
\item\label{sim:5b}  $\binom{p}{2}(2\arcsin\varrho/\pi)^2=(3/2)\cdot p/n$;
\item\label{sim:5c}  $\binom{p}{2}(2\arcsin\varrho/\pi)^2=2p/n$.
\end{enumerate} }
\end{example}

The powers for  Examples \ref{eg:sim-power2}--\ref{eg:sim-power4} are
reported in Tables \ref{tab:power1}--\ref{tab:power4}.  Several
observations stand out.  First, throughout the sparse examples,
we found that the proposed tests have the highest powers on average.
%followed by LD$_{\tau^*}$ that performs well except for Example~\ref{eg:sim-power3}. 
Among the three proposed tests, the power of
DHS$_D$ is highest on average, followed by DHS$_{\tau^*}$.  Recall,
however, that DHS$_D$ can be subject to slight size inflation.
%Second, Table \ref{tab:power1} (the first half) shows that, under a relatively dense alternative
%case, the powers of the proposed tests, YZS, and LD$_{\tau^*}$ are all
%perfect, while the performance of the other methods is considerably worse. 
Second, focusing on the results in Example~\ref{eg:sim-power2}, 
we note that, as more independent components are
added, the power of YZS significantly decreases. This is as expected
and indicates that YZS is less powerful in detection of sparse
dependences. In addition, both HCL$_\tau$ and HCL$_\rho$ perform
unsatisfactorily in Example~\ref{eg:sim-power2}, indicating that they are powerless
in detecting the considered non-linear, non-monotone dependences, an observation that was also made in \citet{MR3798874}. 
Fourth, Tables \ref{tab:power1} and \ref{tab:power3} jointly confirm
the intuition that, for sparse alternatives, the proposed maximum-type
tests dominate $L_2$-type ones including both YZS and LD$_{\tau^*}$,
especially when $p$ is large. 
% This is as expected since the proposed tests are all based on
% maximum-type statistics and hence are very sensitive to such spiked
% correlations.
In addition, we note that, under the setting of Example
\ref{eg:sim-power3}, the performances of HCL$_\tau$ and HCL$_\rho$ are
the second best to the proposed consistent rank-based tests,
indicating that there exist cases in which simple rank correlation
measures like Kendall's $\tau$ and Spearman's $\rho$ can still detect
aspects of non-linear non-monotone dependences. {Fifth, under a Gaussian parametric model, Table \ref{tab:power3} (the first part) shows that  CJ, the maximum-type test based on Pearson's $r$,  indeed outperforms all others, though the difference between it and
  the proposed rank-based ones is small.} {Lastly, Table \ref{tab:power4} shows that, as the signals are rather dense, $L_2$-type tests dominate the maximum-type ones, confirming the intuition and also the theoretical findings that $L_2$-type ones are more powerful in the dense setting.}

{
\renewcommand{\tabcolsep}{1.5pt}
\renewcommand{\arraystretch}{1.0}
\begin{table}[t]
\centering
\caption{Empirical powers of the eleven competing tests in Example \ref{eg:sim-power2}}{
\footnotesize
\begin{tabular}{ccC{0.37in}C{0.37in}C{0.37in}
                  C{0.35in}C{0.35in}C{0.35in}
                  C{0.35in}C{0.35in}
                  C{0.35in}C{0.35in}C{0.35in}}
$n$ & $p$ &  DHS$_D$ & DHS$_R$ & DHS$_{\tau^*}\!$ & 
             LD$_{\tau}$ & LD$_{\rho}$ & LD$_{\tau^*}$ & 
             HCL$_{\tau}$ & HCL$_{\rho}$ & 
             YZS & SC & CJ \\
    &     &  \multicolumn{11}{c}{Results for Example \ref{eg:sim-power2}\ref{sim:3a'}} \\
100 &  50 & 1.000 & 1.000 & 1.000 & 0.058 & 0.049 & 1.000 & 0.089 & 0.033 & 0.442 & 0.047 & 0.024\\
    & 100 & 1.000 & 1.000 & 1.000 & 0.055 & 0.045 & 1.000 & 0.070 & 0.025 & 0.156 & 0.049 & 0.018\\
    & 200 & 1.000 & 1.000 & 1.000 & 0.052 & 0.046 & 1.000 & 0.049 & 0.017 & 0.071 & 0.048 & 0.011\\
    & 400 & 1.000 & 1.000 & 1.000 & 0.058 & 0.049 & 0.973 & 0.043 & 0.014 & 0.057 & 0.050 & 0.011\\
    & 800 & 1.000 & 0.827 & 1.000 & 0.061 & 0.052 & 0.520 & 0.029 & 0.009 & 0.054 & 0.050 & 0.007\\
200 &  50 & 1.000 & 1.000 & 1.000 & 0.053 & 0.045 & 1.000 & 0.099 & 0.038 & 0.955 & 0.053 & 0.033\\
    & 100 & 1.000 & 1.000 & 1.000 & 0.055 & 0.051 & 1.000 & 0.080 & 0.038 & 0.435 & 0.050 & 0.032\\
    & 200 & 1.000 & 1.000 & 1.000 & 0.048 & 0.045 & 1.000 & 0.060 & 0.028 & 0.142 & 0.045 & 0.023\\
    & 400 & 1.000 & 1.000 & 1.000 & 0.052 & 0.047 & 1.000 & 0.049 & 0.023 & 0.078 & 0.048 & 0.023\\
    & 800 & 1.000 & 1.000 & 1.000 & 0.057 & 0.052 & 1.000 & 0.044 & 0.020 & 0.053 & 0.050 & 0.021\\
\vspace{-.5em}\\
    &     &  \multicolumn{11}{c}{Results for Example \ref{eg:sim-power2}\ref{sim:3b'}} \\
100 &  50 & 1.000 & 1.000 & 1.000 & 0.065 & 0.049 & 1.000 & 0.106 & 0.037 & 0.984 & 0.049 & 0.026\\
    & 100 & 1.000 & 1.000 & 1.000 & 0.054 & 0.046 & 1.000 & 0.078 & 0.026 & 0.660 & 0.046 & 0.020\\
    & 200 & 1.000 & 1.000 & 1.000 & 0.059 & 0.052 & 1.000 & 0.055 & 0.018 & 0.266 & 0.051 & 0.014\\
    & 400 & 1.000 & 1.000 & 1.000 & 0.059 & 0.052 & 0.996 & 0.039 & 0.014 & 0.107 & 0.046 & 0.010\\
    & 800 & 1.000 & 0.897 & 1.000 & 0.059 & 0.051 & 0.642 & 0.030 & 0.007 & 0.067 & 0.052 & 0.005\\
200 &  50 & 1.000 & 1.000 & 1.000 & 0.062 & 0.053 & 1.000 & 0.120 & 0.042 & 1.000 & 0.050 & 0.033\\
    & 100 & 1.000 & 1.000 & 1.000 & 0.053 & 0.047 & 1.000 & 0.087 & 0.040 & 0.996 & 0.045 & 0.036\\
    & 200 & 1.000 & 1.000 & 1.000 & 0.051 & 0.047 & 1.000 & 0.061 & 0.030 & 0.729 & 0.045 & 0.023\\
    & 400 & 1.000 & 1.000 & 1.000 & 0.053 & 0.050 & 1.000 & 0.050 & 0.023 & 0.272 & 0.053 & 0.023\\
    & 800 & 1.000 & 1.000 & 1.000 & 0.047 & 0.044 & 1.000 & 0.042 & 0.021 & 0.102 & 0.046 & 0.016\\    
\end{tabular}} 
\label{tab:power1}
%Results are averaged over $5000$ simulated data sets.
\end{table}
}

{
\renewcommand{\tabcolsep}{1.5pt}
\renewcommand{\arraystretch}{1.0}
\begin{table}
\centering
\caption{Empirical powers of the eleven competing tests in Example \ref{eg:sim-power3}}{
\footnotesize
\begin{tabular}{ccC{0.37in}C{0.37in}C{0.37in}
                  C{0.35in}C{0.35in}C{0.35in}
                  C{0.35in}C{0.35in}
                  C{0.35in}C{0.35in}C{0.35in}}
$n$ & $p$ &  DHS$_D$ & DHS$_R$ & DHS$_{\tau^*}\!$ & 
             LD$_{\tau}$ & LD$_{\rho}$ & LD$_{\tau^*}$ & 
             HCL$_{\tau}$ & HCL$_{\rho}$ & 
             YZS & SC & CJ \\
    &     &  \multicolumn{11}{c}{Results for Example \ref{eg:sim-power3}\ref{sim:4a}}\\ % multivariate normal distribution} \\
100 &  50 & 0.967 & 0.962 & 0.964 & 0.705 & 0.586 & 0.946 & 0.970 & 0.966 & 0.555 & 0.624 & 0.973\\
    & 100 & 0.959 & 0.952 & 0.954 & 0.392 & 0.259 & 0.914 & 0.960 & 0.956 & 0.252 & 0.283 & 0.962\\
    & 200 & 0.950 & 0.938 & 0.942 & 0.161 & 0.107 & 0.840 & 0.950 & 0.943 & 0.109 & 0.115 & 0.950\\
    & 400 & 0.936 & 0.924 & 0.928 & 0.089 & 0.064 & 0.727 & 0.938 & 0.931 & 0.064 & 0.073 & 0.941\\
    & 800 & 0.931 & 0.911 & 0.918 & 0.061 & 0.049 & 0.539 & 0.929 & 0.916 & 0.051 & 0.051 & 0.931\\
200 &  50 & 0.991 & 0.991 & 0.991 & 0.912 & 0.891 & 0.988 & 0.993 & 0.992 & 0.871 & 0.906 & 0.993\\
    & 100 & 0.984 & 0.985 & 0.985 & 0.728 & 0.627 & 0.974 & 0.988 & 0.987 & 0.579 & 0.650 & 0.989\\
    & 200 & 0.984 & 0.983 & 0.983 & 0.408 & 0.278 & 0.954 & 0.987 & 0.985 & 0.255 & 0.299 & 0.988\\
    & 400 & 0.986 & 0.983 & 0.983 & 0.166 & 0.110 & 0.917 & 0.986 & 0.985 & 0.111 & 0.115 & 0.989\\
    & 800 & 0.980 & 0.976 & 0.978 & 0.073 & 0.060 & 0.839 & 0.983 & 0.980 & 0.058 & 0.063 & 0.986\\
\vspace{-.5em}\\
    &     &  \multicolumn{11}{c}{Results for Example \ref{eg:sim-power3}\ref{sim:4b}}\\ % $X=\sin(2\pi Z^{1/3}/3)$} \\
100 &  50 & 0.759 & 0.642 & 0.687 & 0.244 & 0.167 & 0.623 & 0.623 & 0.553 & 0.277 & 0.260 & 0.786\\
    & 100 & 0.747 & 0.624 & 0.670 & 0.131 & 0.091 & 0.555 & 0.607 & 0.540 & 0.131 & 0.125 & 0.758\\
    & 200 & 0.720 & 0.583 & 0.635 & 0.082 & 0.062 & 0.444 & 0.578 & 0.502 & 0.080 & 0.075 & 0.714\\
    & 400 & 0.702 & 0.557 & 0.615 & 0.065 & 0.054 & 0.333 & 0.549 & 0.471 & 0.060 & 0.061 & 0.678\\
    & 800 & 0.679 & 0.512 & 0.577 & 0.057 & 0.048 & 0.218 & 0.517 & 0.431 & 0.052 & 0.051 & 0.638\\
200 &  50 & 0.897 & 0.843 & 0.866 & 0.423 & 0.343 & 0.825 & 0.810 & 0.767 & 0.577 & 0.550 & 0.928\\
    & 100 & 0.880 & 0.819 & 0.846 & 0.248 & 0.170 & 0.753 & 0.784 & 0.732 & 0.287 & 0.273 & 0.912\\
    & 200 & 0.855 & 0.789 & 0.818 & 0.128 & 0.088 & 0.670 & 0.757 & 0.714 & 0.129 & 0.128 & 0.891\\
    & 400 & 0.849 & 0.768 & 0.799 & 0.074 & 0.059 & 0.571 & 0.743 & 0.689 & 0.065 & 0.064 & 0.875\\
    & 800 & 0.820 & 0.738 & 0.772 & 0.051 & 0.045 & 0.450 & 0.713 & 0.654 & 0.053 & 0.051 & 0.852\\
\vspace{-.5em}\\
    &     &  \multicolumn{11}{c}{Results for Example \ref{eg:sim-power3}\ref{sim:4c}}\\ % $X=\sin(\pi Z^3/4)$} \\
100 &  50 & 0.654 & 0.579 & 0.608 & 0.209 & 0.137 & 0.541 & 0.582 & 0.513 & 0.111 & 0.106 & 0.365\\
    & 100 & 0.656 & 0.566 & 0.599 & 0.109 & 0.072 & 0.464 & 0.580 & 0.502 & 0.071 & 0.064 & 0.344\\
    & 200 & 0.635 & 0.527 & 0.571 & 0.069 & 0.055 & 0.364 & 0.539 & 0.455 & 0.056 & 0.051 & 0.311\\
    & 400 & 0.617 & 0.496 & 0.546 & 0.068 & 0.059 & 0.256 & 0.516 & 0.421 & 0.053 & 0.058 & 0.277\\
    & 800 & 0.597 & 0.455 & 0.507 & 0.055 & 0.049 & 0.164 & 0.487 & 0.370 & 0.055 & 0.049 & 0.238\\
200 &  50 & 0.824 & 0.789 & 0.803 & 0.396 & 0.302 & 0.750 & 0.785 & 0.753 & 0.238 & 0.211 & 0.606\\
    & 100 & 0.812 & 0.773 & 0.788 & 0.219 & 0.143 & 0.681 & 0.768 & 0.732 & 0.113 & 0.100 & 0.570\\
    & 200 & 0.792 & 0.752 & 0.767 & 0.101 & 0.072 & 0.596 & 0.750 & 0.711 & 0.063 & 0.059 & 0.543\\
    & 400 & 0.776 & 0.728 & 0.744 & 0.070 & 0.054 & 0.499 & 0.730 & 0.689 & 0.058 & 0.057 & 0.513\\
    & 800 & 0.755 & 0.699 & 0.723 & 0.052 & 0.048 & 0.360 & 0.699 & 0.646 & 0.044 & 0.051 & 0.473\\
\end{tabular}}
\label{tab:power3} 
%Results are averaged over $5000$ simulated data sets.
\end{table}
}

{
\renewcommand{\tabcolsep}{1.5pt}
\renewcommand{\arraystretch}{1.0}
\begin{table}
\centering
\caption{ Empirical powers of the eleven competing tests in Example \ref{eg:sim-power4}}{
\footnotesize
\begin{tabular}{ccC{0.37in}C{0.37in}C{0.37in}
                  C{0.35in}C{0.35in}C{0.35in}
                  C{0.35in}C{0.35in}
                  C{0.35in}C{0.35in}C{0.35in}}
$n$ & $p$ &  DHS$_D$ & DHS$_R$ & DHS$_{\tau^*}\!$ & 
             LD$_{\tau}$ & LD$_{\rho}$ & LD$_{\tau^*}$ & 
             HCL$_{\tau}$ & HCL$_{\rho}$ & 
             YZS & SC & CJ \\
    &     &  \multicolumn{11}{c}{Results for Example \ref{eg:sim-power4}\ref{sim:5a}} \\
100 &  50 & 0.102 & 0.068 & 0.074 & 0.532 & 0.524 & 0.350 & 0.062 & 0.046 & 0.474 & 0.578 & 0.042\\
    & 100 & 0.104 & 0.056 & 0.066 & 0.578 & 0.560 & 0.361 & 0.052 & 0.036 & 0.492 & 0.620 & 0.033\\
    & 200 & 0.096 & 0.035 & 0.048 & 0.583 & 0.565 & 0.343 & 0.037 & 0.022 & 0.488 & 0.620 & 0.018\\
    & 400 & 0.104 & 0.040 & 0.050 & 0.542 & 0.534 & 0.320 & 0.038 & 0.018 & 0.471 & 0.610 & 0.012\\
    & 800 & 0.095 & 0.018 & 0.032 & 0.570 & 0.552 & 0.344 & 0.027 & 0.007 & 0.487 & 0.620 & 0.005\\
200 &  50 & 0.104 & 0.080 & 0.086 & 0.564 & 0.544 & 0.357 & 0.081 & 0.072 & 0.478 & 0.614 & 0.068\\
    & 100 & 0.073 & 0.052 & 0.059 & 0.590 & 0.580 & 0.357 & 0.054 & 0.043 & 0.509 & 0.654 & 0.052\\
    & 200 & 0.085 & 0.061 & 0.064 & 0.594 & 0.585 & 0.336 & 0.052 & 0.040 & 0.488 & 0.652 & 0.040\\
    & 400 & 0.075 & 0.040 & 0.049 & 0.604 & 0.591 & 0.332 & 0.038 & 0.028 & 0.498 & 0.668 & 0.024\\
    & 800 & 0.067 & 0.036 & 0.044 & 0.586 & 0.573 & 0.320 & 0.034 & 0.027 & 0.488 & 0.640 & 0.026\\
\vspace{-.5em}\\
    &     &  \multicolumn{11}{c}{Results for Example \ref{eg:sim-power4}\ref{sim:5b}} \\
100 &  50 & 0.130 & 0.078 & 0.086 & 0.792 & 0.782 & 0.554 & 0.076 & 0.064 & 0.722 & 0.836 & 0.055\\
    & 100 & 0.110 & 0.056 & 0.062 & 0.808 & 0.800 & 0.584 & 0.052 & 0.035 & 0.746 & 0.848 & 0.032\\
    & 200 & 0.099 & 0.046 & 0.060 & 0.810 & 0.800 & 0.553 & 0.042 & 0.026 & 0.738 & 0.850 & 0.021\\
    & 400 & 0.110 & 0.030 & 0.041 & 0.808 & 0.797 & 0.587 & 0.034 & 0.014 & 0.738 & 0.854 & 0.012\\
    & 800 & 0.098 & 0.020 & 0.033 & 0.816 & 0.804 & 0.579 & 0.023 & 0.008 & 0.745 & 0.872 & 0.006\\
200 &  50 & 0.116 & 0.094 & 0.098 & 0.802 & 0.801 & 0.546 & 0.103 & 0.084 & 0.718 & 0.858 & 0.098\\
    & 100 & 0.098 & 0.072 & 0.076 & 0.827 & 0.822 & 0.571 & 0.075 & 0.062 & 0.768 & 0.878 & 0.058\\
    & 200 & 0.063 & 0.040 & 0.042 & 0.848 & 0.840 & 0.570 & 0.036 & 0.030 & 0.764 & 0.888 & 0.030\\
    & 400 & 0.070 & 0.048 & 0.055 & 0.834 & 0.829 & 0.578 & 0.042 & 0.032 & 0.752 & 0.883 & 0.030\\
    & 800 & 0.081 & 0.036 & 0.046 & 0.866 & 0.862 & 0.560 & 0.041 & 0.028 & 0.788 & 0.907 & 0.030\\
\vspace{-.5em}\\
    &     &  \multicolumn{11}{c}{Results for Example \ref{eg:sim-power4}\ref{sim:5c}} \\
100 &  50 & 0.157 & 0.102 & 0.116 & 0.904 & 0.900 & 0.731 & 0.093 & 0.069 & 0.864 & 0.926 & 0.076\\
    & 100 & 0.124 & 0.067 & 0.082 & 0.914 & 0.909 & 0.738 & 0.058 & 0.036 & 0.878 & 0.943 & 0.042\\
    & 200 & 0.115 & 0.051 & 0.059 & 0.918 & 0.913 & 0.748 & 0.046 & 0.028 & 0.880 & 0.947 & 0.018\\
    & 400 & 0.112 & 0.034 & 0.046 & 0.930 & 0.926 & 0.738 & 0.038 & 0.017 & 0.888 & 0.954 & 0.009\\
    & 800 & 0.101 & 0.030 & 0.039 & 0.927 & 0.924 & 0.744 & 0.029 & 0.012 & 0.879 & 0.946 & 0.012\\
200 &  50 & 0.120 & 0.100 & 0.098 & 0.935 & 0.932 & 0.740 & 0.110 & 0.098 & 0.894 & 0.952 & 0.118\\
    & 100 & 0.107 & 0.082 & 0.085 & 0.941 & 0.939 & 0.740 & 0.072 & 0.066 & 0.892 & 0.960 & 0.065\\
    & 200 & 0.096 & 0.062 & 0.072 & 0.962 & 0.960 & 0.768 & 0.064 & 0.048 & 0.930 & 0.976 & 0.046\\
    & 400 & 0.077 & 0.042 & 0.046 & 0.964 & 0.962 & 0.792 & 0.037 & 0.028 & 0.930 & 0.978 & 0.024\\
    & 800 & 0.090 & 0.043 & 0.054 & 0.956 & 0.956 & 0.776 & 0.044 & 0.028 & 0.922 & 0.980 & 0.016\\
\end{tabular}}
\label{tab:power4}
%Results are averaged over $5000$ simulated data sets.
\end{table}
}

We end this section with a discussion of the simulation-based
approach. In view of Proposition~\ref{prop:easy}, the distributions of
rank-based test statistics are invariant to the generating
distribution, and hence we may use simulations to approximate the exact distribution of 
\begin{equation*}
S := \frac{n-1}{\lambda_1\binom{m}{2}}\max_{j<k}\hat U_{jk}-4\log p-(\mu_1-2)\log\log p+\frac{\Lambda}{\lambda_1}.
\end{equation*}
In detail, we pick a large integer $M$ to be the number of
independent replications. For each $t\in[M]$, compute $S^{(t)}$ as the value of
$S$ for an $n\times p$
data matrix $\fX^{(t)}\in\R^{n\times p}$ drawn as having i.i.d.~Uniform(0,1) entries.   Let
% $\hat F_{n,p;M}(\cdot)$
% be the
% empirical distribution function of $\{S^{(t)}\}_{t\in[M]}$, i.e.,
%\begin{equation*}
$\hat F_{n,p;M}(y)=\frac{1}{M}\sum_{t=1}^{M}\ind\{S^{(t)}\leq y\}$,
$y\in\R$,
% .
% \end{equation*}
be the resulting empirical distribution function.  %The Dvoretzky--Kiefer--Wolfowitz inequality
%\citep{MR0083864,MR1062069} for discontinuous distribution functions
%\citep[Theorem 11.6]{MR2724368} then guarantees  that %, for each $(n,p)$,
%\begin{equation*}
%\Pr\Big\{\sup_{y\in\R}|\hat F_{n,p;M}(y)-F_{n,p}(y)|>\Big(\frac{\log M}{M}\Big)^{1/2}\Big\}\le\frac{2}{M^2},
%\end{equation*}
%where $F_{n,p}(y)$ is %be the population counterpart, i.e.,
%the distribution
%function of $S$ under $H_0$.
%% of $\hat F_{n,p;M}(y)$. 
%% The Dvoretzky--Kiefer--Wolfowitz inequality
%% \citep{MR0083864,MR1062069} for discontinuous distribution functions
%% \citep[Theorem 11.6]{MR2724368} then guarantees  that %, for each $(n,p)$,
%% \begin{equation*}
%% \Pr\Big\{\sup_{y\in\R}|\hat F_{n,p;M}(y)-F_{n,p}(y)|>\Big(\frac{\log M}{M}\Big)^{1/2}\Big\}\le\frac{2}{M^2}.
%% \end{equation*}
For a specified
significance
level  $\alpha\in(0,1)$, we may now use the simulated quantile $\hat
Q_{\alpha,n,p;M}:=\inf\{y\in\R:\hat F_{n,p;M}(y)\geq 1-\alpha\}$ to form
the test
% Therefore, for any pre-specified significance level $\alpha\in(0,1)$, we propose the following exact test using the simulated-based rejection threshold:
\begin{equation*}
\mathsf{T}_{\alpha}^{\text{exact}}:= \ind\Big\{\frac{n-1}{\lambda_1\binom{m}{2}}\max_{j<k}\hat U_{jk}-4\log p-(\mu_1-2)\log\log p+\frac{\Lambda}{\lambda_1}>\hat Q_{\alpha,n,p;M}\Big\}.
\end{equation*}
The test becomes exact in the large $M$ limit, immediately by the Dvoretzky--Kiefer--Wolfowitz inequality
% \citep{MR0083864,MR1062069}
for {empirical} distribution functions
\citep[e.g.,][Theorem 11.6]{MR2724368}, and is shown explicitly in the following proposition.
% where $\hat Q_{\alpha,n,p;M}:=\inf\{y\in\R:\hat F_{n,p;M}(y)\geq 1-\alpha\}$ is the empirical $1-\alpha$ quantile of $\hat F_{n,p;M}(\cdot)$. 

\begin{proposition}
Under the independence
  hypothesis $H_0$, for each $(n,p)$, we have with probability at least $1-2/M^2$ that
\begin{align*}
\sup_{\alpha\in[0,1]}\Big|\Pr\Big[S>\hat Q_{\alpha,n,p;M}\Big|\{\fX^{(t)}\}_{t=1}^{M}\Big]
- \Big\{1-\hat F_{n,p;M}(\hat Q_{\alpha,n,p;M})\Big\}\Big|\le\Big(\frac{\log M}{M}\Big)^{1/2}.
\end{align*}
\end{proposition}

Table \ref{tab:size-power} in the supplement gives the sizes and powers of the proposed tests with simulation-based critical values ($M=5,000$).  
The table shows results only for Examples~\ref{eg:sim-size}, \ref{eg:sim-power3}, and \ref{eg:sim-power4} as the simulated powers under Example~\ref{eg:sim-power2} were all perfectly one.  
It can be observed that all sizes are now well controlled, with powers of the proposed
tests only slightly different from the ones without using simulation. {An alternative to the simulation-based approach would be a permutation-based approach, but we find simulation based on the
pivotal null distribution simpler to analyze and with the advantage that
approximation errors can be made arbitrarily small via larger Monte
Carlo samples.}

{
\section{Discussion}\label{sec:discussion}

\subsection{Discussion of Assumption \ref{assumption:key}}\label{sec:key}

Assumption \ref{assumption:key} plays a key role in our analysis. It
synthesizes crucial properties satisfied by the three rank correlation statistics from Examples \ref{eg:hd}--\ref{eg:bd}.

From a more general perspective, one might ask whether there is an
exact relation between Assumption \ref{assumption:key} and the
properties of I- and D-consistency summarized in \citet{MR3842884}. As
a matter of fact, to our knowledge, most existing test statistics
(including rank-based, distance covariance-based, and kernel-based
ones) that permit consistent assessment of pairwise independence are
asymptotically equivalent to U-statistics with the corresponding
kernels degenerate under the null, which echoes Assumption 2.1(ii).
%We conjecture that this degeneracy is a feature of all test statistics that are I- and D-consistent. 
 The only exception is a new rank correlation measure that was just proposed \citep{chatterjee2019new}, whose limiting distribution is normal. Its analysis uses the permutation theory and, in particular, is not based on the U-statistic framework.
Assumption 2.1(iii), on the other hand, is much more specific and related to the particular properties of rank-based consistent tests. This assumption, however, is key to the establishment of Theorem \ref{thm:distr}.

\subsection{Discussion of $\tau^*$}\label{sec:tau}

In this section we give new perspectives on Bergsma--Dassios--Yanagimoto's correlation measure $\tau^*:=\E h_{\tau^*}$, introduced in Example \ref{eg:bd}.
\citet{MR0029139} stated a problem about the relationship between
equiprobable rankings and independence that was solved by
\citet{zbMATH03366369}.  In the proof of his Proposition 9,
\citet{zbMATH03366369} presented a correlation measure that is proportional to $\tau^*$ 
of Bergsma--Dassios if the pair is absolutely continuous. 
Accordingly, we term the correlation   
``Bergsma--Dassios--Yanagimoto's $\tau^*$''.
% for a better reflection of historical developments.  
Yanagimoto's key relation gives rise to an interesting identity between Hoeffding's $D$,
Blum--Kiefer--Rosenblatt's $R$, and Bergsma--Dassios--Yanagimoto's
$\tau^*$ statistics.  This identity appears to be unknown in the literature.
In detail,  if $\mz_1,\ldots,\mz_6\in\R^2$ have no tie among their first and their second entries, respectively, then
\begin{align}\label{eq:identity}
    &3\cdot \mbinom{6}{5}^{-1}\sum_{1\leq i_1<\cdots<i_5\leq 6}h_D(\mz_{i_1},\ldots,\mz_{i_5})+2h_R(\mz_1,\ldots,\mz_6)\\
=\; &5\cdot \mbinom{6}{4}^{-1}\sum_{1\leq i_1<\cdots<i_4\leq 6}h_{\tau^*}(\mz_{i_1},\ldots,\mz_{i_4}). \notag
\end{align}
Equation \eqref{eq:identity} can be easily verified by calculating
all $6!$ entrywise permutations of $\{1,2,\ldots,6\}$, but may be false when ties exist.
Using the identity, we can make a step towards proving the conjecture
raised in \citet{MR3178526}, that is, for an arbitrary random pair $(Z_1, Z_2)^\top\in\R^2$, do we have $\E h_{\tau^*}\geq 0$ with equality if and only if $Z_1$ and $Z_2$ are independent?
%\end{quote}
%It extends the results in Theorem 1 in \citet{MR3178526} to random
% vectors of continuous margins but not (absolutely) continuous
% bivariate distribution.

\begin{theorem}\label{thm:ya}
For any random vector $\mZ=(Z_1,Z_2)^\top\in \R^2$ with continuous marginal distributions, we have $\E h_{\tau^*}\geq 0$ and the equality holds if and only if $Z_1$ is independent of $Z_2$.
\end{theorem}

Similarly, a monotonicity property of $\E h_D$ and $\E h_R$
proved by \citet[Sec.~2]{zbMATH03366369} extends to $\E h_{\tau^*}$.  
We state the Gaussian version of this property.

\begin{theorem}
  \label{thm:increasing}
  If $\mZ=(Z_1,Z_2)^\top\in\R^2$ is bivariate Gaussian with (Pearson) correlation
  $\rho$, then $\E h_D$ and $\E h_R$ and, thus, also $\E h_{\tau^*}$ are
  increasing functions of $|\rho|$.
\end{theorem}

Theorem \ref{thm:ya} complements the results in Theorem 1 in
\citet{MR3178526} to include random vectors with continuous margins 
and a bivariate joint distribution that is continuous (implied by marginal continuity)
but need not be {absolutely} continuous. Such an example of
distribution on $\mathbb{R}^2$ that has continuous margins but is not
absolutely continuous has been constructed in Remark 1 in
\citet{zbMATH03366369}, where it is used to  illustrate an inconsistency problem about Hoeffding's $D$. A simpler example is the uniform distribution on the unit circle in $\mathbb{R}^2$. For this, we revisit a comment of \citet{MR3842884} who noted that based on existing literature ``it is not guaranteed that
  $\E h_{\tau^*} > 0$ when $(X,Y)^\top$ is generated uniformly on the unit
  circle in $\R^2$.''  We are able to calculate the values of $D$ and
  $R$ for this example and, thus, can deduce the value of $\tau^*$.
% Using the connections The following proposition calculates the value $\E h_{\tau^*}$ as well as $\E h_D, \E h_R$ under this scenario via the use of the identity \eqref{eq:identity}. 

\begin{proposition}\label{prop:circ}
For $(X,Y)^\top$ following the uniform distribution on the unit circle in $\R^2$, we have
$\E h_D=\E h_R=\E h_{\tau^*}=1/16$.
\end{proposition}

\bibliographystyle{apalike}
\bibliography{DHS_ams}

\newpage

\appendix

\section{Technical proofs}

We first introduce more notation. For $x\in\R$, let
$x_{+}$ denote the positive part of $x$, defined as $\max\{x,0\}$.
 For any vector $\mv\in\mathbb{R}^p$, we denote $\|\mv\|$ as its Euclidean norm.
We define the $L^{\infty}$ norm of a random variable as $\lVert X\rVert_{\infty}=\inf\{t\ge 0:|X|\le t\text{ a.s.}\}$, the $\psi_2$ (sub-gaussian) norm as $\lVert X\rVert_{\psi_2}=\inf\{t>0:\E\exp(X^2/t^2)\le 2\}$, and the $\psi_1$ (sub-exponential) norm as $\lVert X\rVert_{\psi_1}=\inf\{t>0:\E\exp(|X|/t)\le 2\}$.
For any measure $\Pr_{Z}$ and kernel $h$, we let
$H_n^{(\ell)}(\cdot;\Pr_{Z})$ be the U-statistic based on the
completely degenerate kernel $h^{(\ell)}(\cdot;\Pr_{Z})$ from~(\ref{eq:han-hl}):
\begin{equation}
H_n^{(\ell)}(\cdot;\Pr_{Z}):=\mbinom{n}{\ell}^{-1}\sum_{1\le i_1< i_2<\cdots< i_\ell\le n} h^{(\ell)}\Big(Z_{i_1},\ldots,Z_{i_\ell};\Pr_{Z}\Big).
\end{equation}

\subsection{Proofs for Section \ref{sec:tests} of the main paper}

\begin{proof}[Proof of Proposition \ref{lem:easy2}]
Since $Y_1,\ldots,Y_{d}$ are i.i.d. realizations of $\zeta$, we have
\begin{equation}\label{eqn:bor1}
\Pr\Big(\max_{j\in[d]} Y_j\le x\Big)=\{\Pr(\zeta\le x)\}^d=\{F_{\zeta}(x)\}^{d}=\{1-\overline F_{\zeta}(x)\}^{d},
\end{equation}
where
\begin{equation}
\overline F_\zeta(x):=\Pr(\zeta >x)=\frac{\kappa}{\Gamma(\mu_1/2)}\Big(\frac{x+\Lambda}{2\lambda_1}\Big)^{\mu_1/2-1}\exp\Big(-\frac{x+\Lambda}{2\lambda_1}\Big)\{1+o(1)\}
\end{equation}
for $x>-\Lambda$ as $x\to\infty$ by Equation (6) in \citet{zbMATH03179069}.
Take $x=4\lambda_1\log p +\lambda_1(\mu_1-2)\log\log
p-\Lambda+\lambda_1 y$.  Noticing that $x\to\infty$ as $p\to\infty$ and recalling $d=p(p-1)/2$, we obtain
\begin{align}\label{eqn:bor2}
d\cdot\overline F_\zeta(x)&=\frac{p(p-1)}{2}\frac{\kappa}{\Gamma(\mu_1/2)}\Big(\frac{x+\Lambda}{2\lambda_1}\Big)^{\mu_1/2-1}\exp\Big(-\frac{x+\Lambda}{2\lambda_1}\Big)\{1+o(1)\}\notag\\
&=\frac{p(p-1)}{2}\frac{\kappa}{\Gamma(\mu_1/2)}(2\log p)^{\mu_1/2-1}\exp\Big\{-2\log p -\Big(\frac{\mu_1}{2}-1\Big)\log\log p -\frac{y}{2}\Big\}\{1+o(1)\}\notag\\
&=\frac{2^{\mu_1/2-2}\kappa}{\Gamma(\mu_1/2)}\exp\Big(-\frac{y}{2}\Big)\{1+o(1)\}.
\end{align}
Combing \eqref{eqn:bor1} and \eqref{eqn:bor2}, we deduce that
\begin{equation*}
\Pr(\max_{j\in[d]} Y_j\le x)=\{1-\overline F_{\zeta}(x)\}^{d}\to\exp\Big\{-\lim_{d\to\infty}d\cdot\overline F_\zeta(x)\Big\}=\exp\Big\{-\frac{2^{\mu_1/2-2}\kappa}{\Gamma(\mu_1/2)}\exp\Big(-\frac{y}{2}\Big)\Big\},
\end{equation*}
which concludes the proof of  the lemma.
\end{proof}

\subsection{Proofs for Section \ref{sec:theory} of the main paper}

\subsubsection{Proof of Theorem \ref{thm:md}}

\begin{proof}[Proof of Theorem \ref{thm:md}]

%It is easily derived from the results in Chapter 5.5.2 of \citet{MR595165} that $(n-1)\hat U_n$ weakly converges to$\binom{m}{2}\sum_{v=1}^{\infty}\lambda_v(\xi_v^2-1)$, which implies, { for any bounded sequence $\{x_n\}$},
%\begin{equation*}
%\frac{\Pr\Big\{\binom{m}{2}^{-1}(n-1)\hat U_n>x_n\Big\}}{\Pr\Big\{\sum_{v=1}^\infty\lambda_v(\xi_v^2-1)>x_n\Big\}}=1+o(1).
%\end{equation*}
%For the proof we may thus only consider, without loss of generality, 
%{ We only consider those sequences $\{x_n\}$ that will go to infinity and $\theta$'s that are positive. The case of bounded $\{x_n\}$ and $\theta=0$ follows the same proof structure with minor modifications, and is thus omitted.}   
We proceed in two steps,
proving first the case $m=2$ and then generalizing to 
$m\ge2$.  For notational convenience we introduce the constants $b_1:=\lVert h\rVert_\infty <\infty$ and $b_2:=\sup_v\lVert \phi_v\rVert_\infty <\infty$. 
\smallskip

{\bf Step I. }  Suppose $m=2$.
We start with the scenario that there are infinitely many nonzero
eigenvalues. For a large enough integer $K$ to be specified later, we
define the ``truncated'' kernel  of $h_2(z_1,z_2;\Pr_Z)$ as
$h_{2,K}(z_1,z_2;\Pr_Z)=\sum_{v=1}^{K}\lambda_v\phi_v(z_1)\phi_v(z_2)$,
with corresponding U-statistic
\begin{equation*}
\hat U_{K,n}:=\mbinom{n}{2}^{-1}\sum_{1\le i<j\le n}h_{2,K}(Z_i,Z_j;\Pr_Z).
\end{equation*}
For simpler presentation, define $Y_{v,i}=\phi_v(Z_i)$ for all $v=1,2,\ldots$ and $i\in[n]$. In view of the expansions of $h_{2,K}(\cdot)$ and $h_2(\cdot)$, $\hat U_{K,n}$ and $\hat U_n$ can be written as
\begin{align*}
\hat U_{K,n}&=\frac{1}{n-1}\Big\{\sum_{v=1}^K\lambda_v \Big(n^{-1/2}\sum_{i=1}^n Y_{v,i}\Big)^2-\sum_{v=1}^K\lambda_v \Big(\frac{\sum_{i=1}^n Y_{v,i}^2}{n}\Big)\Big\}\\
\text{and}~~~\hat U_n&=\frac{1}{n-1}\Big\{\sum_{v=1}^{\infty}\lambda_v \Big(n^{-1/2}\sum_{i=1}^n Y_{v,i}\Big)^2-\sum_{v=1}^{\infty}\lambda_v \Big(\frac{\sum_{i=1}^n Y_{v,i}^2}{n}\Big)\Big\}.
\end{align*}
We now quantify the approximation accuracy of $\hat U_{K,n}$ to $\hat U_n$. Using Slutsky's argument, we obtain, 
\begin{align}\label{eqn:slu1}
     & \Pr\Big\{(n-1)\hat U_n\ge x_n\Big\}=\Pr\Big\{\sum_{v=1}^{\infty}\lambda_v \Big(n^{-1/2}\sum_{i=1}^n Y_{v,i}\Big)^2-\sum_{v=1}^{\infty}\lambda_v \Big(\frac{\sum_{i=1}^n Y_{v,i}^2}{n}\Big)\ge x_n\Big\}\notag\\
\le\;& \Pr\Big\{\sum_{v=1}^{K}\lambda_v \Big(n^{-1/2}\sum_{i=1}^n Y_{v,i}\Big)^2-\sum_{v=1}^{K}\lambda_v \Big(\frac{\sum_{i=1}^n Y_{v,i}^2}{n}\Big)\ge x_n-\epsilon_1\Big\}
       +\Pr\Big\{\Big|(n-1)(\hat U_n-\hat U_{K,n})\Big|\ge \epsilon_1\Big\}\notag\\
\le\;& \Pr\Big\{\sum_{v=1}^{K}\lambda_v \Big(n^{-1/2}\sum_{i=1}^n Y_{v,i}\Big)^2-\sum_{v=1}^{K}\lambda_v \ge x_n-\epsilon_1-\epsilon_2\Big\}
       +\Pr\Big\{\Big|(n-1)(\hat U_n-\hat U_{K,n})\Big|\ge \epsilon_1\Big\}\notag\\
     & +\Pr\Big\{\Big|\sum_{v=1}^{K}\lambda_v\frac{\sum_{i=1}^n(Y_{v,i}^2-1)}{n}\Big|\ge\epsilon_2\Big\},
\end{align}
where $\epsilon_1,\epsilon_2$ are constants to be specified later.

The first term on the right-hand side of \eqref{eqn:slu1} may be
controlled using Za\u{\i}tsev's multivariate moderate deviation
theorem.  For this, we require a dimension-free bound on
$\sum_{v=1}^{K}u_v\lambda_v^{1/2}(n^{-1/2} Y_{v,i})$ for any
$\bm{u}\in\R^K$ satisfying $\lVert \bm{u}\rVert=1$.  Indeed, we have
\begin{align*}
\Big\lVert \sum_{v=1}^{K}u_v\lambda_v^{1/2}\frac{Y_{v,i}}{n^{1/2}}\Big\rVert_{\infty}
\le \sum_{v=1}^K|u_v|\lambda_v^{1/2}\frac{\lVert  Y_{v,i}\rVert_{\infty}}{n^{1/2}}
\le \Big(\sum_{v=1}^K u_v^2\Big)^{1/2}\Big(\sum_{v=1}^K \lambda_v\Big)^{1/2}n^{-1/2}b_2
\le n^{-1/2}\Lambda^{1/2} b_2.
\end{align*}
Thus all assumptions in Theorem 1.1 in \citet{MR876255} are satisfied with the $\tau$ in his Equation (1.5) chosen to be $n^{-1/2}\Lambda^{1/2} b_2$. We obtain the following bound:
\begin{align}\label{eqn:zai}
     & \Pr\Big\{\sum_{v=1}^{K}\lambda_v \Big(n^{-1/2}\sum_{i=1}^n Y_{v,i}\Big)^2-\sum_{v=1}^{K}\lambda_v \ge x_n-\epsilon_1-\epsilon_2\Big\}\notag\\
=  \;& \Pr\Big[\Big\{\sum_{v=1}^{K}\Big(\lambda_v^{1/2}\sum_{i=1}^nn^{-1/2} Y_{v,i}\Big)^2\Big\}^{1/2} \ge \Big(x_n-\epsilon_1-\epsilon_2+\sum_{v=1}^{K}\lambda_v\Big)_{+}^{1/2}\Big]\notag\\
\le\;& \Pr\Big[\Big\{\sum_{v=1}^{K}(\lambda_v^{1/2} \xi_v)^2\Big\}^{1/2} \ge \Big(x_n-\epsilon_1-\epsilon_2+\sum_{v=1}^{K}\lambda_v\Big)_{+}^{1/2}-\epsilon_3\Big]+c_1K^{5/2}\exp\Big\{-\frac{\epsilon_3}{c_2K^{5/2}(n^{-1/2}\Lambda^{1/2} b_2)}\Big\}\notag\\
=  \;& \Pr\Big[\sum_{v=1}^{K}\lambda_v \xi_v^2 \ge \Big\{\Big(x_n-\epsilon_1-\epsilon_2+\sum_{v=1}^{K}\lambda_v\Big)_{+}^{1/2}-\epsilon_3\Big\}_{+}^2\Big]+c_1K^{5/2}\exp\Big(-\frac{n^{1/2}\epsilon_3}{c_2\Lambda^{1/2} b_2K^{5/2}}\Big),
\end{align}
where $\epsilon_3$ is a constant to be specified later.
%This finishes the derivation on the first term in \eqref{eqn:slu1}. 
Combining \eqref{eqn:slu1} and \eqref{eqn:zai}, we find using
Slutsky's argument once again that
\begin{align}\label{eqn:slu2}
     & \Pr\Big\{\sum_{v=1}^{\infty}\lambda_v \Big(n^{-1/2}\sum_{i=1}^n Y_{v,i}\Big)^2-\sum_{v=1}^{\infty}\lambda_v \Big(\frac{\sum_{i=1}^n Y_{v,i}^2}{n}\Big)\ge x_n\Big\}\notag\\
\le\;& \Pr\Big[\sum_{v=1}^{\infty}\lambda_v (\xi_v^2-1) \ge \Big\{\Big(x_n-\epsilon_1-\epsilon_2+\sum_{v=1}^{K}\lambda_v\Big)_{+}^{1/2}-\epsilon_3\Big\}_{+}^2-\sum_{v=1}^{K}\lambda_v-\epsilon_4\Big]\notag\\
     & +\Pr\Big\{\Big|(n-1)(\hat U_n-\hat U_{K,n})\Big|\ge \epsilon_1\Big\}
       +\Pr\Big\{\Big|\sum_{v=1}^{K}\lambda_v\frac{\sum_{i=1}^n(Y_{v,i}^2-1)}{n}\Big|\ge\epsilon_2\Big\}\notag\\
     & +c_1K^{5/2}\exp\Big(-\frac{n^{1/2}\epsilon_3}{c_2\Lambda^{1/2} b_2K^{5/2}}\Big)
       +\Pr\Big\{\Big|\sum_{v=K+1}^{\infty}\lambda_v (\xi_v^2-1)\Big|\ge \epsilon_4\Big\},
\end{align}
where $\epsilon_4$ is another constant to be specified later.
In the following, we separately study the five terms on the right-hand
side of \eqref{eqn:slu2}, starting from the first term.

Let $\epsilon_n^*:=x_n-[\{(x_n-\epsilon_1-\epsilon_2+\sum_{v=1}^{K}\lambda_v)_{+}^{1/2}-\epsilon_3\}_{+}^2-\sum_{v=1}^{K}\lambda_v-\epsilon_4]$. %We then have $\epsilon\asymp \epsilon_1+\epsilon_2+\epsilon_3x^{1/2}+\epsilon_3^2+\epsilon_4$ and
% we have
Then
{
\begin{equation*}
\epsilon_n^*=\begin{cases}
\epsilon_1+\epsilon_2+2\epsilon_3(x_n-\epsilon_1-\epsilon_2+\sum_{v=1}^{K}\lambda_v)^{1/2}-\epsilon_3^2+\epsilon_4,
& \text{if }x_n+\sum_{v=1}^{K}\lambda_v\ge \epsilon_1+\epsilon_2+\epsilon_3^2,\\
x_n+\sum_{v=1}^{K}\lambda_v+\epsilon_4,
& \text{otherwise},
\end{cases}
\end{equation*}}
and
\begin{align}
\Pr\Big\{\sum_{v=1}^{\infty}\lambda_v (\xi_v^2-1) \ge x_n-\epsilon_n^{*}\Big\}\le \Pr\Big\{\sum_{v=1}^{\infty}\lambda_v (\xi_v^2-1) \ge x_n\Big\}+(\epsilon_n^*)_{+}\cdot\max_{x'\in [x_n-(\epsilon_n^*)_{+},x_n]} p_\zeta(x')\label{eqn:bndi}
\end{align}
where $p_{\zeta}(x)$ is the density of the random variable $\zeta:=\sum_{v=1}^{\infty}\lambda_v (\xi_v^2-1)$.

We turn to the second term in \eqref{eqn:slu2}.  Proposition 2.6.1
and Example 2.5.8 in \citet{MR3837109}
yield that %{\color{black} (compute $c_3=8$, $c_4=c_3\cdot 1=8$)}
\begin{equation*}
\Big\lVert n^{-1/2}\sum_{i=1}^n Y_{v,i}\Big\rVert_{\psi_2}^2
\le {\color{black} 8}n^{-1}\sum_{i=1}^n \Big\lVert Y_{v,i}\Big\rVert_{\psi_2}^2
\le {\color{black} 8(\log 2)^{-1}}b_2^2
\le {\color{black} 12}b_2^2.
\end{equation*}
Applying the triangle inequality and Lemma 2.7.6 in
\citet{MR3837109}, we deduce that
\begin{align*}
\Big\lVert \sum_{v=K+1}^{\infty}\lambda_v\Big(n^{-1/2}\sum_{i=1}^n Y_{v,i}\Big)^2\Big\rVert_{\psi_1} &\le \sum_{v=K+1}^{\infty}\lambda_v\Big\lVert \Big(n^{-1/2}\sum_{i=1}^n Y_{v,i}\Big)^2\Big\rVert_{\psi_1}\\
&=\sum_{v=K+1}^{\infty}\lambda_v
\Big\lVert n^{-1/2}\sum_{i=1}^n Y_{v,i}\Big\rVert_{\psi_2}^2\le {\color{black} 12}b_2^2\sum_{v=K+1}^{\infty}\lambda_v.
\end{align*}
Using Proposition 2.7.1 in \citet{MR3837109}, this is seen to further imply that, for any $\epsilon_1'>0$,
\begin{equation*}
\Pr\Big\{\sum_{v=K+1}^{\infty}\lambda_v\Big(n^{-1/2}\sum_{i=1}^n Y_{v,i}\Big)^2\ge \epsilon_1'\Big\}\le {\color{black} 2}\exp\Big(-\frac{\epsilon_1'}{{\color{black} 12}b_2^2\sum_{v=K+1}^{\infty}\lambda_v}\Big).
\end{equation*}
Noting that
$$\left|\sum_{v=K+1}^{\infty}\lambda_v\Big(n^{-1}\sum_{i=1}^n Y_{v,i}^2\Big)\right|\le
b_2^2\sum_{v=K+1}^{\infty}\lambda_v,$$
we obtain, for any $\epsilon_1>b_2^2\sum_{v=K+1}^\infty \lambda_v$,
\begin{align} 
     &\Pr\Big\{\Big|(n-1)(\hat U_n-\hat U_{K,n})\Big|\ge \epsilon_1\Big\}\notag\\
\le\;&
       \Pr\Big\{\Big|\sum_{v=K+1}^{\infty}\lambda_v\Big(n^{-1/2}\sum_{i=1}^n
       Y_{v,i}\Big)^2\Big|+\Big|\sum_{v=K+1}^{\infty}\lambda_v\Big(n^{-1}\sum_{i=1}^n
       Y_{v,i}^2\Big)\Big|\ge \epsilon_1\Big\}\notag\\
  \label{eqn:bnd1}
\le\;&\Pr\Big\{\Big|\sum_{v=K+1}^{\infty}\lambda_v\Big(n^{-1/2}\sum_{i=1}^n Y_{v,i}\Big)^2\Big|\ge \epsilon_1-b_2^2\sum_{v=K+1}^\infty \lambda_v\Big\}
\le     2e^{1/{\color{black} 12}}\exp\Big(-\frac{\epsilon_1}{{\color{black} 12}b_2^2\sum_{v=K+1}^{\infty}\lambda_v}\Big).
\end{align}
%This finishes the bound of the second term in \eqref{eqn:slu2}.

We next study the third term in \eqref{eqn:slu2}. Again, Proposition 2.6.1 and Example 2.5.8 in \citet{MR3837109} give
\begin{equation*}
\Big\lVert n^{-1}\sum_{i=1}^n(Y_{v,i}^2-1)\Big\rVert_{\psi_2}^2\le 8n^{-2}\sum_{i=1}^n\Big\lVert Y_{v,1}^2-1\Big\rVert_{\psi_2}^2\le {\color{black} 12}n^{-1}(b_2^2+1)^2,
\end{equation*}
which further yields
\begin{equation*}
\Big\lVert \sum_{v=1}^{K}\lambda_v\sum_{i=1}^n\frac{Y_{v,i}^2-1}{n}\Big\rVert_{\psi_2}\le \sum_{v=1}^{K}\lambda_v \Big\lVert n^{-1}\sum_{i=1}^n(Y_{v,i}^2-1)\Big\rVert_{\psi_2}\le {\color{black} 12}^{1/2}n^{-1/2}\Lambda(b_2^2+1).
\end{equation*}
Using Proposition 2.5.2 in \citet{MR3837109}, we have, for any $\epsilon_2>0$,
\begin{equation}\label{eqn:bnd2}
\Pr\bigg(\Big|\sum_{v=1}^{K}\lambda_v\sum_{i=1}^n\frac{Y_{v,i}^2-1}{n}\Big|\ge\epsilon_2\bigg)\le 2\exp\Big\{-\frac{n\epsilon_2^2}{{\color{black} 48}\Lambda^2(b_2^2+1)^2}\Big\}.
\end{equation}

The fourth term in \eqref{eqn:slu2} is explicit, and it remains to
bound the fifth and last term.  Since $\xi_v$ is a sub-gaussian random
variable, $\xi_v^2-1$ is sub-exponential. %Actually we can show that
One readily verifies
$\lVert \xi_v^2-1\rVert_{\psi_1}\le {\color{black} 4}$, %as follows:
%\begin{equation*}
%\E\exp(|Z^2-1|/4)\le\E\exp\{(Z^2-1)/4\}+e^{1/4}\Pr(|Z|\le 1)=(1-2/4)^{-1/2}e^{-1/4}+e^{1/4}\{\Phi(1)-\Phi(-1)\}<2.
%\end{equation*}
and accordingly
\begin{equation*}
\Big\lVert \sum_{v=K+1}^{\infty}\lambda_v (\xi_v^2-1)\Big\rVert_{\psi_1}
\le\sum_{v=K+1}^{\infty}\lambda_v\lVert \xi_v^2-1\rVert_{\psi_1}
\le {\color{black} 4}\sum_{v=K+1}^{\infty}\lambda_v.
\end{equation*}
By Proposition 2.7.1 in \citet{MR3837109}, this further
implies that, for any $\epsilon_4>0$,
\begin{equation}\label{eqn:bnd4}
\Pr\Big\{\Big|\sum_{v=K+1}^{\infty}\lambda_v (\xi_v^2-1)\Big|\ge \epsilon_4\Big\}
\le 2\exp\Big(-\frac{\epsilon_4}{{\color{black} 4}\sum_{v=K+1}^{\infty}\lambda_v}\Big).
\end{equation}

We now specify the integer $K$ to be $\lfloor n^{(1-3\theta)/5}\rfloor$. By the definition of $\theta$, %we have 
%\begin{equation*}
%\liminf_{n\to\infty} \frac{n^{-a}}{\sum_{v=K+1}^{\infty}\lambda_v}>0. 
%\end{equation*}
%Thus, 
there exists a positive absolute  constant $C_{\theta}$ such that $\sum_{v=K+1}^{\infty}\lambda_v\le C_{\theta}n^{-\theta}$ for all sufficiently large $n$. %More specifically, if $\liminf_{n\to\infty} {n^{-\theta}}/{\sum_{v=K+1}^{\infty}\lambda_v}=C^*<\infty$, we can choose $C_{\theta}=2/C^*$; if $\liminf_{n\to\infty} {n^{-a}}/{\sum_{v=K+1}^{\infty}\lambda_v}=\infty$, $C_{\theta}$ can be chosen as any positive number.
Combining this fact and inequalities \eqref{eqn:slu2}--\eqref{eqn:bnd4}, 
%using the fact that $\sum_{v=K+1}^{\infty}\lambda_v\le C_{\theta}n^{-a}$ for all large enough $n$, 
we obtain
\begin{align}\label{eqn:fin}
     & \frac{\Pr\Big\{(n-1)\hat U_n >x_n\Big\}}{\Pr\Big\{\sum_{v=1}^{\infty}\lambda_v(\xi_v^2-1) >x_n\Big\}}-1\notag\\
\le\;& \{\overline F_\zeta(x_n)\}^{-1}\Big[(\epsilon_n^*)_{+}\cdot\max_{x'\in [x_n-(\epsilon_n^*)_{+},x_n]}p_\zeta(x') 
       +2e^{1/{\color{black} 12}}\exp\Big(-\frac{\epsilon_1}{{\color{black} 12}b_2^2\sum_{v=K+1}^{\infty}\lambda_v}\Big)\notag\\
     & +2\exp\Big\{-\frac{n\epsilon_2^2}{{\color{black} 48}\Lambda^2(b_2^2+1)^2}\Big\}
       +c_1K^{5/2}\exp\Big(-\frac{n^{1/2}\epsilon_3}{c_2\Lambda^{1/2} b_2K^{5/2}}\Big)
       +2\exp\Big(-\frac{\epsilon_4}{{\color{black} 4}\sum_{v=K+1}^{\infty}\lambda_v}\Big)\Big]\notag\\
\le\;& \{\overline F_\zeta(x_n)\}^{-1}\Big[(\epsilon_n^*)_{+}\cdot\!\!\!\max_{x'\in [x_n-(\epsilon_n^*)_{+},x_n]}p_\zeta(x') 
       +2e^{1/{\color{black} 12}}\exp\Big(-\frac{\epsilon_1}{{\color{black} 12}b_2^2C_{\theta}n^{-\theta}}\Big)
       +2\exp\Big\{-\frac{n\epsilon_2^2}{{\color{black} 48}\Lambda^2(b_2^2+1)^2}\Big\}\notag\\
     & +c_1n^{(1-3\theta)/2}\exp\Big\{-\frac{n^{1/2}\epsilon_3}{c_2\Lambda^{1/2} b_2n^{(1-3\theta)/2}}\Big\}
       +2\exp\Big(-\frac{\epsilon_4}{{\color{black} 4}C_{\theta}n^{-\theta}}\Big)\Big],
\end{align}
which we shall prove to converge to $0$ uniformly on $[-\Lambda,e_n n^{\theta}]$. 
The starting point for proving this are Equations (5) and (6) in \citet{zbMATH03179069},
which yield that the density $p_\zeta(x)$ and the survival function
$\overline F_\zeta(x):=\Pr(\zeta >x)$ of
$\zeta=\sum_{v=1}^{\infty}\lambda_v (\xi_v^2-1)$ satisfy
\begin{align*}
p_\zeta(x)&=\frac{\kappa}{2\lambda_1\cdot\Gamma(\mu_1/2)}\Big(\frac{x+\Lambda}{2\lambda_1}\Big)^{\mu_1/2-1}\exp\Big(-\frac{x+\Lambda}{2\lambda_1}\Big)\{1+o(1)\}\\
~~~\text{and}~~~\overline F_\zeta(x)&=\frac{\kappa}{\Gamma(\mu_1/2)}\Big(\frac{x+\Lambda}{2\lambda_1}\Big)^{\mu_1/2-1}\exp\Big(-\frac{x+\Lambda}{2\lambda_1}\Big)\{1+o(1)\}
\end{align*}
for $x>-\Lambda$ tending to infinity.  %Here $\zeta:=\sum_{v=1}^{\infty}\lambda_v
% (\xi_v^2-1)$.
Here $\mu_1$ is the multiplicity of the largest eigenvalue $\lambda_1$ and $\kappa:=\prod_{v= \mu_1+1}^{\infty}(1-\lambda_v/\lambda_1)^{-1/2}$.

Consider the first term in \eqref{eqn:fin}.  We claim that there exists
an absolute constant $C_{\zeta}^*>0$ such that, for all
$0<\epsilon\leq \lambda_1/2$,
\begin{equation}\label{eqn:Mills}
\sup_{x\ge -\Lambda}\Big|\{\overline F_\zeta(x)\}^{-1}\cdot\max_{x'\in [x-\epsilon,x]}p_\zeta(x')\Big|\leq C_{\zeta}^*. 
\end{equation}
%Notice that $p_\zeta(x)$ is monotonically decreasing for $x>-\Lambda$ if $\mu_1\le 2$, and for $x>\lambda_1(\mu_1-2)-\Lambda$ if $\mu_1\ge 3$. 
Indeed, we have ${p_\zeta(x)}/{\overline F_\zeta(x)}=(2\lambda_1)^{-1}\{1+o(1)\}$ { as $x\to\infty$}, and thus there exists { an absolute constant} $x_0>-\Lambda$ such that ${p_\zeta(x)}/{\overline F_\zeta(x)}\le \lambda_1^{-1}$ for all $x\ge x_0$.
%Without loss of generality, we assume that $x_0'>-\Lambda$ if $\mu_1\le 2$, and $x_0'>\lambda_1(\mu_1-2)-\Lambda$ if $\mu_1\ge 3$. 
Then for all $0<\epsilon\le \lambda_1/2$ and all $x\ge x_0+\epsilon$, 
\begin{equation*}
\frac{\max_{x'\in [x-\epsilon,x]} p_\zeta(x')}{\overline F_\zeta(x)}=\frac{p_\zeta(x-\epsilon')}{\overline F_\zeta(x)}\le\frac{p_\zeta(x-\epsilon')}{\overline F_\zeta(x-\epsilon') - \epsilon'\cdot p_\zeta(x-\epsilon')}=\frac{1}{{\overline F_\zeta(x-\epsilon')}/{p_\zeta(x-\epsilon')}-\epsilon'}\le \frac{2}{\lambda_1}
\end{equation*}
where $\epsilon'\in[0,\epsilon]$ is chosen
such that
$p_\zeta(x-\epsilon')=\max_{x'\in [x-\epsilon,x]} p_\zeta(x')$.  Now
\eqref{eqn:Mills} holds when taking
\[C_{\zeta}^*=\max\Big\{\frac{2}{\lambda_1},\{\overline F_\zeta(x_0+\lambda_1/2)\}^{-1}\cdot\max_{x'\in [-\Lambda,x_0+\lambda_1/2]}p_\zeta(x')\Big\}.\]

From \eqref{eqn:Mills}, to control the first term in \eqref{eqn:fin}, it remains to show that $(\epsilon_n^*)_{+}$ converges to $0$ uniformly on $[-\Lambda, e_n n^{\theta}]$ as $n\to\infty$.  
Choosing
\begin{equation}\label{eq:epsilons}
\epsilon_1={\color{black} 12}b_2^2C_{\theta}n^{-\theta}\Big(\frac{x_n+\Lambda}{2\lambda_1}+n^{\theta/2}\Big),~~~
\epsilon_2=n^{-\theta},~~~
\epsilon_3=n^{-\theta/2},~~~
\epsilon_4={\color{black} 4}C_{\theta}n^{-\theta}\Big(\frac{x_n+\Lambda}{2\lambda_1}+n^{\theta/2}\Big),
\end{equation}
we deduce that the first term in \eqref{eqn:fin} {converges uniformly to $0$ on $[-\Lambda,e_n n^{\theta}]$ as $n\to\infty$ by observing
that if $x_n+\sum_{v=1}^{K}\lambda_v\ge \epsilon_1+\epsilon_2+\epsilon_3^2$,
\begin{align*}
%\epsilon_n^*&\ge\epsilon_1+\epsilon_2+2\epsilon_3(\lambda_1-\epsilon_1-\epsilon_2)^{1/2}-\epsilon_3^2+\epsilon_4=\epsilon_1+2\epsilon_3(\lambda_1-\epsilon_1-\epsilon_2)^{1/2}+\epsilon_4\geq 0,\\
\epsilon_n^*&\le \epsilon_1+\epsilon_2+2\epsilon_3(x_n+\Lambda)^{1/2}-\epsilon_3^2+\epsilon_4\\
%&=\frac{6b_2^2C_{\theta}+2C_{\theta}}{\lambda_1}\Big(\frac{x_n+\Lambda}{n^{\theta}}\Big)+2\Big(\frac{x_n+\Lambda}{n^{\theta}}\Big)^{1/2}+\frac{6b_2^2C_{\theta}+2C_{\theta}}{\lambda_1}n^{-\theta/2}\\
&\le\frac{6b_2^2C_{\theta}+2C_{\theta}}{\lambda_1}\Big(e_n+\frac{\Lambda}{n^{\theta}}\Big)+2\Big(e_n+\frac{\Lambda}{n^{\theta}}\Big)^{1/2}+({12b_2^2C_{\theta}+4C_{\theta}})n^{-\theta/2},
\end{align*}
and otherwise
\[
\epsilon_n^*\le \epsilon_1+\epsilon_2+\epsilon_3^2+\epsilon_4
\le\frac{6b_2^2C_{\theta}+2C_{\theta}}{\lambda_1}\Big(e_n+\frac{\Lambda}{n^{\theta}}\Big)+({12b_2^2C_{\theta}+4C_{\theta}})n^{-\theta/2}+2n^{-\theta}.
\]
%for all $n$ large enough.
Recall that we consider a positive sequence $\{e_n\}$ tending to $0$}.

We then further verify that the other four terms in \eqref{eqn:fin} {also converge to $0$ uniformly on $[-\Lambda,e_n n^{\theta}]$} as $n\to\infty$. %This proves the case that there are infinite nonzero eigenvalues.
There exists some absolute constant $c_{\zeta}^*>0$ such that for all $x\ge{2\lambda_1-\Lambda}$,
\begin{equation}\label{eq:survlow}
\overline F_\zeta(x)\ge c_{\zeta}^*\,\frac{\kappa}{\Gamma(\mu_1/2)}\Big(\frac{x+\Lambda}{2\lambda_1}\Big)^{\mu_1/2-1}\exp\Big(-\frac{x+\Lambda}{2\lambda_1}\Big).
\end{equation}
We then have, by noticing $\theta< 1/3$, for all $n$ large enough and all $x_n\in[2\lambda_1-\Lambda,e_n n^{\theta}]$,
\begin{align}
&\{\overline F_\zeta(x_n)\}^{-1}\exp\Big(-\frac{\epsilon_1}{{\color{black} 12}b_2^2C_{\theta}n^{-\theta}}\Big)
\le \frac{\Gamma(\mu_1/2)}{c_{\zeta}^*\kappa}\Big(\frac{e_nn^{\theta}+\Lambda}{2\lambda_1}\Big)^{1/2}\exp(-n^{\theta/2}),\notag\\
&\{\overline F_\zeta(x_n)\}^{-1}\exp\Big\{-\frac{n\epsilon_2^2}{{\color{black} 48}\Lambda^2(b_2^2+1)^2}\Big\}
\le \frac{\Gamma(\mu_1/2)}{c_{\zeta}^*\kappa}\Big(\frac{e_nn^{\theta}+\Lambda}{2\lambda_1}\Big)^{1/2}\exp(-C'n^{1/3}),\notag\\
&\{\overline F_\zeta(x_n)\}^{-1}n^{(1-3\theta)/2}\exp\Big\{-\frac{n^{1/2}\epsilon_3}{c_2\Lambda^{1/2}b_2n^{(1-3\theta)/2}}\Big\}
\le \frac{\Gamma(\mu_1/2)}{c_{\zeta}^*\kappa}\Big(\frac{e_nn^{\theta}+\Lambda}{2\lambda_1}\Big)^{1/2}n^{(1-3\theta)/2}\exp(-C''n^{\theta}),\notag\\
&\{\overline F_\zeta(x_n)\}^{-1}\exp\Big(-\frac{\epsilon_4}{{\color{black} 4}C_{\theta}n^{-\theta}}\Big)
\le \frac{\Gamma(\mu_1/2)}{c_{\zeta}^*\kappa}\Big(\frac{e_nn^{\theta}+\Lambda}{2\lambda_1}\Big)^{1/2}\exp(-n^{\theta/2}).\label{eq:foursign}
\end{align}
Here $C'$ and $C''$ are some absolute positive constants.
{
The inequalities in \eqref{eq:foursign} hold for all sufficiently large $n$ and all $x_n\in[-\Lambda,2\lambda_1-\Lambda]$ with replacing ${\Gamma(\mu_1/2)}/{c_{\zeta}^*\kappa}\cdot\{(e_nn^{\theta}+\Lambda)/(2\lambda_1)\}^{1/2}$ by $\{\overline F_\zeta(2\lambda_1-\Lambda)\}^{-1}$, which together with \eqref{eq:foursign} concludes the uniform convergence.}

If there are only finitely many nonzero eigenvalues, a simple
modification to \eqref{eqn:fin} gives
\begin{align}\label{eqn:fin2}
     & \frac{\Pr\Big\{(n-1)\hat U_n >x_n\Big\}}{\Pr\Big\{\sum_{v=1}^{\infty}\lambda_v(\xi_v^2-1) >x_n\Big\}}-1\le\frac{1}{\overline F_\zeta(x)}\Big[(\epsilon_n^*)_{+}\cdot\max_{x'\in [x_n-(\epsilon_n^*)_{+},x_n]} p_\zeta(x')\notag\\
     &\mkern150mu +2\exp\Big\{-\frac{n\epsilon_2^2}{{\color{black} 48}\Lambda^2(b_2^2+1)^2}\Big\}
       +c_1K^{5/2}\exp\Big(-\frac{n^{1/2}\epsilon_3}{c_2\Lambda^{1/2} b_2K^{5/2}}\Big)\Big],
\end{align}
where $\epsilon_n^*:=x_n-[\{(x_n-\epsilon_2+\Lambda)_{+}^{1/2}-\epsilon_3\}_{+}^2-\Lambda]$
and $K$ is the number of nonzero eigenvalues. Choosing $\epsilon_2=n^{-1/3}$, $\epsilon_3=n^{-1/6}$, one can obtain that the right-hand side of \eqref{eqn:fin2} converges uniformly to $0$ on $[-\Lambda,e_n n^{\theta}]$ as $n\to\infty$. %This completes the proof of the setting that there are finite nonzero eigenvalues.

%We thus proved that
%% $\Pr\{(n-1)\hat U_n >x\}/\Pr\{\sum_{v=1}^{\infty}\lambda_v(\xi_v^2-1) >x\}-1\leq o(1)$. It can be similarly proven that 
%\[
%\frac{\Pr\Big\{(n-1)\hat U_n
%  >x_n\Big\}}{\Pr\Big\{\sum_{v=1}^{\infty}\lambda_v(\xi_v^2-1)
%  >x_n\Big\}}-1\le o(1)
%\]
%{for any positive sequence $\{x_n\}$ such that $x_n=o(n^{\theta})$.   It can be
%shown similarly that the left-hand side in the display is
%lower-bounded by an $o(1)$ term,
%}
%which
%concludes the proof of the case $m=2$.

We thus proved that
\[
\sup_{x_n\in[-\Lambda,e_n n^{\theta}]}\Bigg[\frac{\Pr\Big\{(n-1)\hat U_n
  >x_n\Big\}}{\Pr\Big\{\sum_{v=1}^{\infty}\lambda_v(\xi_v^2-1)
  >x_n\Big\}}-1\Bigg]\le o(1).
\]
{
For the lower bound, it can be shown similarly that if there are
infinitely many nonzero eigenvalues, then
\begin{align}\label{eqn:fin3}
     & \frac{\Pr\Big\{(n-1)\hat U_n >x_n\Big\}}{\Pr\Big\{\sum_{v=1}^{\infty}\lambda_v(\xi_v^2-1) >x_n\Big\}}-1\notag\\
\ge\;& \{\overline F_\zeta(x_n)\}^{-1}\Big[-(\epsilon_n^{**})_+\cdot\max_{x''\in [x_n,x_n+(\epsilon_n^{**})_+]}p_\zeta(x'') 
       -2e^{1/{12}}\exp\Big(-\frac{\epsilon_1}{{12}b_2^2C_{\theta}n^{-\theta}}\Big)\notag\\
     & -2\exp\Big\{-\frac{n\epsilon_2^2}{{48}\Lambda^2(b_2^2+1)^2}\Big\}
       -c_1n^{(1-3\theta)/2}\exp\Big\{-\frac{n^{1/2}\epsilon_3}{c_2\Lambda^{1/2} b_2n^{(1-3\theta)/2}}\Big\}
       -2\exp\Big(-\frac{\epsilon_4}{{4}C_{\theta}n^{-\theta}}\Big)\Big],
\end{align}
where
\begin{align*}
\epsilon_n^{**}
:=\;&\Big[\Big\{\Big(x_n+\epsilon_1+\epsilon_2+\sum_{v=1}^{K}\lambda_v\Big)_{+}^{1/2}+\epsilon_3\Big\}^2-\sum_{v=1}^{K}\lambda_v+\epsilon_4\Big]-x_n\\
=\;&\begin{cases}
\epsilon_1+\epsilon_2+2\epsilon_3(x_n+\epsilon_1+\epsilon_2+\sum_{v=1}^{K}\lambda_v)^{1/2}+\epsilon_3^2+\epsilon_4,
& \text{if }x_n+\sum_{v=1}^{K}\lambda_v\ge-\epsilon_1-\epsilon_2,\\
-x_n-\sum_{v=1}^{K}\lambda_v+\epsilon_3^2+\epsilon_4,
& \text{otherwise}.
\end{cases}
\end{align*}
We choose \eqref{eq:epsilons} as well.
To conclude the lower bound, it suffices to notice that there exists
an absolute constant $C_{\zeta}^{**}>0$ such that
\begin{equation*}
\sup_{x\ge -\Lambda}\Big|\{\overline F_\zeta(x)\}^{-1}\cdot\max_{x''\in [x,x+(\epsilon_n^{**})_{+}]}p_\zeta(x'')\Big|\leq C_{\zeta}^{**},
\end{equation*}
and $\epsilon_n^{**}$  converges uniformly to $0$ on $[-\Lambda,e_n n^{\theta}]$: if $x_n+\sum_{v=1}^{K}\lambda_v\ge-\epsilon_1-\epsilon_2$, then
\begin{align*}
0<\epsilon_n^{**}
%&\le \epsilon_1+\epsilon_2+2\epsilon_3(x_n+\Lambda+\epsilon_1+\epsilon_2)^{1/2}+\epsilon_3^2+\epsilon_4,\\
&\le \frac{6b_2^2C_{\theta}+2C_{\theta}}{\lambda_1}\Big(e_n+\frac{\Lambda}{n^{\theta}}\Big)+2\Big(e_n+\frac{\Lambda+2}{n^{\theta}}\Big)^{1/2}+({12b_2^2C_{\theta}+4C_{\theta}})n^{-\theta/2}+2n^{-\theta}
\end{align*}
for all $n$ large enough, 
and otherwise
\begin{align*}
0<\epsilon_1+\epsilon_2+\epsilon_3^2+\epsilon_4\le \epsilon_n^{**}
%&\le \Lambda-\sum_{v=1}^{K}\lambda_v+\epsilon_3^2+\epsilon_4\\
&\le \sum\nolimits_{v>\lfloor n^{(1-3\theta)/5}\rfloor}\lambda_v+\frac{2C_{\theta}}{\lambda_1}\Big(e_n+\frac{\Lambda}{n^{\theta}}\Big)+{4C_{\theta}}n^{-\theta/2}+n^{-\theta}.
\end{align*}
If there are only finitely many nonzero eigenvalues, one can obtain
\begin{align}\label{eqn:fin4}
     & \frac{\Pr\Big\{(n-1)\hat U_n >x_n\Big\}}{\Pr\Big\{\sum_{v=1}^{\infty}\lambda_v(\xi_v^2-1) >x_n\Big\}}-1\ge\frac{1}{\overline F_\zeta(x)}\Big[-(\epsilon_n^{**})_+\cdot\max_{x''\in [x_n,x_n+(\epsilon_n^{**})_+]}p_\zeta(x'') 
\notag\\
     &\mkern150mu -2\exp\Big\{-\frac{n\epsilon_2^2}{{48}\Lambda^2(b_2^2+1)^2}\Big\}
       -c_1K^{5/2}\exp\Big(-\frac{n^{1/2}\epsilon_3}{c_2\Lambda^{1/2} b_2K^{5/2}}\Big)\Big],
\end{align}
where $\epsilon_n^{**}:=[\{(x_n+\epsilon_2+\Lambda)_{+}^{1/2}+\epsilon_3\}^2-\Lambda]-x_n$
and $K$ is the number of nonzero eigenvalues. Choosing $\epsilon_2=n^{-1/3}$, $\epsilon_3=n^{-1/6}$, one can verify that the right-hand side of \eqref{eqn:fin4} converges uniformly to $0$ on $[-\Lambda,e_n n^{\theta}]$ as $n\to\infty$.
This completes the proof of the case $m=2$.
}

\smallskip

{\bf Step II. }
We use the Hoeffding decomposition and the exponential inequality for
bounded completely degenerate U-statistics of \cite{MR1235426} to
prove the general case $m\ge2$.  Write
\begin{equation*}%\label{eqn:Hdec}
\mbinom{m}{2}^{-1}(n-1)\hat U_n=(n-1)H_n^{(2)}(\cdot;\Pr_Z)+\sum_{\ell=3}^m\mbinom{m}{2}^{-1}\mbinom{m}{\ell}(n-1)H_n^{(\ell)}(\cdot;\Pr_Z).
\end{equation*}
Using Slutsky's argument, we have
%\begin{align}\label{eqn:slu3}
%     &\Pr\Big\{\mbinom{m}{2}^{-1}(n-1)\hat U_n >x\Big\}\notag\\
%\le\;&\Pr\Big\{(n-1)H_n^{(2)}(\cdot;\Pr_Z) >x-\epsilon_5\Big\}
%       +\sum_{\ell=3}^{m}\Pr\Big\{\mbinom{m}{2}^{-1}\mbinom{m}{\ell}(n-1)\cdot|H_n^{(\ell)}(\cdot;\Pr_Z)|\ge \epsilon_{5,\ell}\Big\},
%\end{align}
\begin{align}\label{eqn:slu3}
     & \frac{\Pr\Big\{\binom{m}{2}^{-1}(n-1)\hat U_n>x_n\Big\}}{\Pr\Big\{\sum_{v=1}^{\infty}\lambda_v(\xi_v^2-1)>x_n\Big\}}\notag\\
\le\;& \frac{\Pr\Big\{(n-1)H_n^{(2)}(\cdot;\Pr_Z)>x_n-\epsilon_{n}^{\#}\Big\}}{\Pr\Big\{\sum_{v=1}^{\infty}\lambda_v(\xi_v^2-1)>x_n\Big\}}
+\sum_{\ell=3}^{m}\frac{\Pr\Big\{\binom{m}{2}^{-1}\binom{m}{\ell}(n-1)\cdot|H_n^{(\ell)}(\cdot;\Pr_Z)|\ge \epsilon_{n,\ell}^{\#}\Big\}}{\Pr\Big\{\sum_{v=1}^{\infty}\lambda_v(\xi_v^2-1)>x_n\Big\}},
\end{align}
where $\{\epsilon_{n,\ell}^{\#}, \ell=3,\ldots,m\}$ are constants to be specified later and $\epsilon_{n}^{\#}:=\sum_{\ell=3}^m \epsilon_{n,\ell}^{\#}$.

We analyze the first term and the remaining terms on the right-hand
side of \eqref{eqn:slu3} separately. To bound the latter, we employ Proposition 2.3(c) in \citet{MR1235426}, which states that there exist absolute positive constants $C_\ell'$ and $C_\ell''$ such that for all $\epsilon_5>0$,
\begin{equation}\label{eqn:AG}
\Pr(n^{\ell/2}|H_n^{(\ell)}(\cdot;\Pr_Z)|\ge \epsilon_5)\le C_\ell'\exp\{-C_\ell''(\epsilon_5/\lVert h^{(\ell)}(\cdot;\Pr_Z)\rVert_{\infty})^{2/\ell}\},
\end{equation}
where $\lVert h^{(\ell)}(\cdot;\Pr_Z)\rVert_\infty\le 2^\ell b_1$ can be shown by the alternative formula of $h^{(\ell)}(z_1,\ldots,z_{\ell}; \Pr_{Z})$ as below:
\[
h^{(\ell)}(z_1,\ldots,z_{\ell}; \Pr_{Z})
=h_{\ell}(z_1,\ldots,z_{\ell};\Pr_{Z})+\sum_{k=1}^{\ell-1}(-1)^{\ell-k}\sum_{1\leq i_1<\cdots<i_k\leq\ell}h_{k}(z_{i_1},\ldots,z_{i_k};\Pr_{Z})+(-1)^{\ell}\E h.
\]
Plugging \eqref{eqn:AG} into each term in the sum on the right of
\eqref{eqn:slu3} implies, for $n\ge2$,
\begin{align}\label{eqn:scm}
     &\sum_{\ell=3}^{m}\frac{\Pr\Big\{\binom{m}{2}^{-1}\binom{m}{\ell}(n-1)\cdot|H_n^{(\ell)}(\cdot;\Pr_Z)|\ge \epsilon_{n,\ell}^{\#}\Big\}}{\Pr\Big\{\sum_{v=1}^{\infty}\lambda_v(\xi_v^2-1)>x_n\Big\}}\notag\\
\le\;&\sum_{\ell=3}^{m}\{\overline F_{\zeta}(x_n)\}^{-1} C_\ell'\exp\Big[-C_\ell''\Big\{n^{\ell/2-1}\mbinom{m}{2}\mbinom{m}{\ell}^{-1}\epsilon_{n,\ell}^{\#}\Big/\lVert h^{(\ell)}(\cdot;\Pr_Z)\rVert_{\infty}\Big\}^{2/\ell}\Big]\notag\\
\le\;&\begin{cases}
{\displaystyle\sum_{\ell=3}^{m}\frac{C'_{\ell}}{c_{\zeta}^*}\Big\{\Big(\frac{x_n+\Lambda}{2\lambda_1}\Big)^{\mu_1/2-1}\exp\Big(-\frac{x_n+\Lambda}{2\lambda_1}\Big)\Big\}^{-1}\exp\Big[-C_\ell'' n^{\theta}\Big\{\mbinom{m}{2}\mbinom{m}{\ell}^{-1}\epsilon_{n,\ell}^{\#}\Big/(2^\ell b_1)\Big\}^{2/\ell}\Big],}\mkern-180mu\\[-.5em]
&\text{for }x_n\in[2\lambda_1-\Lambda,e_n n^{\theta}],\\[-.5em]
{\displaystyle\sum_{\ell=3}^{m}{C'_{\ell}}\{\overline F_{\zeta}(2\lambda_1-\Lambda)\}^{-1}\exp\Big[-C_\ell'' n^{\theta}\Big\{\mbinom{m}{2}\mbinom{m}{\ell}^{-1}\epsilon_{n,\ell}^{\#}\Big/(2^\ell b_1)\Big\}^{2/\ell}\Big],}
&{\text{for }x_n\in[-\Lambda,2\lambda_1-\Lambda],}\\[-1em]
\end{cases}
\end{align}\\[-1em]
where the last step is due to \eqref{eq:survlow} and the fact that $\theta\le1/3\le 1-2/\ell$ for $\ell\ge3$. Taking 
\begin{equation*}
\epsilon_{n,\ell}^{\#}=b_1\mbinom{m}{2}^{-1}\mbinom{m}{\ell}\Big\{\frac{4}{C_\ell''n^{\theta}}\Big(\frac{x_n+\Lambda}{2\lambda_1}+n^{\theta/2}\Big)\Big\}^{\ell/2},
\end{equation*}
the sum on the right-hand side of \eqref{eqn:scm} is seen to be
$o(1)$. It remains to control the first term in \eqref{eqn:slu3}.  We
start by writing the term as 
\begin{align}\label{eqn:pro}
    &\frac{\Pr\Big\{(n-1)H_n^{(2)} >x_n-\epsilon_{n}^{\#}\Big\}}{\Pr\Big\{\sum_{v=1}^{\infty}\lambda_v(\xi_v^2-1)>x_n\Big\}}\notag\\
=\; &\frac{\Pr\Big\{(n-1)H_n^{(2)} >x_n-\epsilon_{n}^{\#}\Big\}}{\Pr\Big\{\sum_{v=1}^{\infty}\lambda_v(\xi_v^2-1)>x_n-\epsilon_{n}^{\#}\Big\}}
\cdot\frac{\Pr\Big\{\sum_{v=1}^{\infty}\lambda_v(\xi_v^2-1)>x_n-\epsilon_{n}^{\#}\Big\}}{\Pr\Big\{\sum_{v=1}^{\infty}\lambda_v(\xi_v^2-1)>x_n\Big\}}.
\end{align}
The first factor in \eqref{eqn:pro} converges uniformly to $1$ on $[-\Lambda,e_n n^{\theta}]$ by going through the same proof in Step I {while noticing that 
although $x_n-\epsilon_{n}^{\#}$ is not necessarily greater than or equal to $-\Lambda$, 
it holds for all $x_n\in[-\Lambda,e_n n^{\theta}]$ that 
\begin{equation}\label{eq:smalleps}
0<\epsilon_{n}^{\#}
=\sum_{\ell=3}^m\epsilon_{n,\ell}^{\#}
\le \sum_{\ell=3}^m b_1\mbinom{m}{2}^{-1}\mbinom{m}{\ell}\Big\{\frac{2}{C''\lambda_1}\Big(e_n+\frac{\Lambda}{n^{\theta}}\Big)+\frac{4}{C''}n^{-\theta/2}\Big\}^{\ell/2}.
\end{equation}
}For the second term in \eqref{eqn:pro}, we have
\begin{equation*}
1\le\frac{\overline F_{\zeta}(x_n-\epsilon_{n}^{\#})}{\overline F_{\zeta}(x_n)}
\le 1+\frac{\epsilon_{n}^{\#}\cdot \max_{x'\in [x_n-\epsilon_{n}^{\#},x_n]} p_\zeta(x')}{\overline F_{\zeta}(x_n)}
\le 1+C_{\zeta}^* \cdot\epsilon_{n}^{\#}
\end{equation*}
for $x_n>0$ and $\epsilon_{n}^{\#}\le\lambda_1/2$ by \eqref{eqn:Mills}. By \eqref{eq:smalleps} again, we have the second
term in \eqref{eqn:pro} uniformly converges to $1$ as well. Therefore, we obtain the
right-hand side of \eqref{eqn:slu3} is uniformly converges to $1$ on $[-\Lambda,e_n n^{\theta}]$ as $n\to\infty$.  Consequently,
\[
\sup_{x_n\in[-\Lambda,e_n n^{\theta}]}\Bigg[\frac{\Pr\Big\{\binom{m}{2}^{-1}(n-1)\hat U_n >x_n\Big\}}{\Pr\Big\{\sum_{v=1}^{\infty}\lambda_v(\xi_v^2-1) >x_n\Big\}}-1\Bigg]\le o(1).
\]  
Again a similar derivation yields a corresponding lower bound of order
$o(1)$, completing the proof of the general case $m\ge2$.
\end{proof}

\subsubsection{Proof of Theorem \ref{thm:distr}}

\begin{proof}[Proof of Theorem \ref{thm:distr}]
  Since the marginal distributions are assumed continuous, we may assume, without
  loss of generality, that they are uniform distributions on
  $[0,1]$. Theorem \ref{thm:md} can then directly apply to the
  studied kernel $h(\cdot)$ in view of Assumption
  \ref{assumption:key}.

  The main tool in this proof is Theorem 1 in \citet{MR972770}.
  Specifically, we use the version presented in Lemma C2 in
  \citet{MR3737306}.  We let
  $I:=\{(j,k) : 1\le j<k\le p\}$, and for all $u:=(j,k)\in I$, we define
  $B_u=\{(\ell,v)\in I : \{\ell,v\}\cap\{j,k\}\neq\varnothing\}$ and
\begin{equation*}
\eta_u:=\eta_{jk}:=\mbinom{m}{2}^{-1}(n-1)\hat U_{jk}.
\end{equation*}
Then the theorem yields that
\begin{equation}\label{eqn:arratia}
\Big|\Pr\Big(\max_{u\in I}\eta_u\le t\Big)-\exp(-L_n)\Big|\le  A_1+A_2+A_3,
\end{equation}
where $L_n=\sum_{u\in I}\Pr(\eta_u>t)$,
\begin{align*}
A_1&=\sum_{u\in I}\sum_{\beta\in B_u}\Pr(\eta_u>t)\Pr(\eta_\beta>t),
~~~~~~
A_2 =\sum_{u\in I}\sum_{\beta\in B_u\backslash\{u\}}\Pr(\eta_u>t,\eta_\beta>t),\\
~~~\text{and}~~~
A_3&=\sum_{u\in I}\E|\Pr\{\eta_u>t\mid\sigma(\eta_\beta:\beta\not\in B_u)\}-\Pr(\eta_u>t)|.
\end{align*}
We now choose an appropriate value of $t$ such that $L_n$ tends to a constant independent of $p$ as $n\to\infty$. Let
\begin{equation}\label{eq:choiceoft}
t=4\lambda_1\log p +\lambda_1(\mu_1-2)\log\log p -\Lambda+\lambda_1 y \;\asymp\; 4\lambda_1\log p=o(n^{\theta}).
\end{equation}
By Theorem \ref{thm:md}, 
\begin{equation}\label{eqn:C1}
L_n=\frac{p(p-1)}{2}\Pr(\eta_{12}>t)=\frac{p(p-1)}{2}\overline F_{\zeta}(t)\{1+o(1)\}.
\end{equation}
% On the other hand,
Using Example 5 in \citet{MR3377812}, we have for any $t>-\Lambda$,
\begin{equation}\label{eqn:C2}
\overline F_\zeta(t)=\frac{\kappa}{\Gamma(\mu_1/2)}\Big(\frac{t+\Lambda}{2\lambda_1}\Big)^{\mu_1/2-1}\exp\Big(-\frac{t+\Lambda}{2\lambda_1}\Big)[1+O\{(\log p)^{-1}\}].
\end{equation}
Combining \eqref{eqn:C1} and \eqref{eqn:C2} implies
\begin{align}
L_n&=\frac{p(p-1)}{2}\frac{\kappa}{\Gamma(\mu_1/2)}\Big(\frac{t+\Lambda}{2\lambda_1}\Big)^{\mu_1/2-1}\exp\Big(-\frac{t+\Lambda}{2\lambda_1}\Big)\{1+o(1)\}\notag\\
&=\frac{p(p-1)}{2}\frac{\kappa}{\Gamma(\mu_1/2)}(2\log p)^{\mu_1/2-1}\exp\Big\{-2\log p -\Big(\frac{\mu_1}{2}-1\Big)\log\log p -\frac{y}{2}\Big\}\{1+o(1)\}\notag\\
&=\frac{2^{\mu_1/2-2}\kappa}{\Gamma(\mu_1/2)}\exp\Big(-\frac{y}{2}\Big)\{1+o(1)\},\label{eq:limin}
\end{align}
where $\kappa:=\prod_{v=\mu_1+1}(1-{\lambda_v}/{\lambda_1})^{-1/2}$.

Next we bound $A_1$, $A_2$, and $A_3$ separately.  We have
\[
  A_1=\tfrac{1}{2}p(p-1)(2p-3)\{\Pr(\eta_{12}>t)\}^2.
\]
Moreover, since Hoeffding's $D$ is a rank-based U-statistic, Proposition
\ref{prop:easy}\ref{prop:easy-2} yields that $\eta_u$ is independent of $\eta_\beta$
for all $u\in I,\beta\in B_u\backslash\{u\}$.   Hence,
\begin{equation*}
A_2=\sum_{u\in I}\sum_{\beta\in B_u\backslash\{u\}}\Pr(\eta_u>t)\Pr(\eta_\beta>t)=p(p-1)(p-2)\{\Pr(\eta_{12}>t)\}^2.
\end{equation*}
Again, by Proposition \ref{prop:easy}\ref{prop:easy-3}, we have $A_3=0$. Accordingly,
\begin{equation}\label{eqn:b123}
A_1+A_2+A_3\le 2p(p-1)^2\{\Pr(\eta_{12}>t)\}^2 = \frac{2(2 L_n)^2}{p}=O\Big(\frac{1}{p}\Big).
\end{equation}
Let $L=\cdot{2^{\mu_1/2-2}\kappa}/{\Gamma(\mu_1/2)}\cdot\exp(-{y}/{2})$. Plugging \eqref{eq:choiceoft}, \eqref{eq:limin}, \eqref{eqn:b123} into \eqref{eqn:arratia} yields
\begin{align*}
&\Big|\Pr\Big\{\mbinom{m}{2}^{-1}(n-1)\max_{j<k}\hat U_{jk}-4\lambda_1\log p-\lambda_1(\mu_1-2)\log\log p+\Lambda\le \lambda_1 y\Big\}-\exp(-L)\Big|\notag\\
\le\;&\Big|\Pr\Big(\max_{u\in I}\eta_\alpha\le t\Big)-\exp(-L_n)\Big|+\Big|\exp(-L_n)-\exp(-L)\Big|=o(1).
\end{align*}
This completes the proof.
\end{proof}

%\begin{proof}[Proof of Corollary \ref{crl:size}]
%This is a direct consequence of Theorem \ref{thm:distr}.
%\end{proof}

\subsubsection{Proof of Corollary \ref{crl:size_eg}}

\begin{proof}[Proof of Corollary \ref{crl:size_eg}]
  We only give the proof for Hoeffding's $D$ here.  The proofs for the
  other two tests are very similar and hence omitted. As in the proof of
  Theorem \ref{thm:distr}, we may assume the margins to be uniformly
  distributed on $[0,1]$ without loss of generality. To employ Theorem
  \ref{thm:distr}, we only need to compute $\theta$. We claim that
\begin{equation}\label{eqn:asy}
\sum_{v=K+1}^{\infty}\lambda_v\asymp \frac{(\log K)^2}{K}.
\end{equation}
If this claim is true, then by the definition of $\theta$, one obtains $\theta=1/8-\delta$, where $\delta$ is an arbitrarily small pre-specified positive absolute constant. 

We now prove \eqref{eqn:asy}. Notice that the $K$ largest
eigenvalues are corresponding to the $K$ smallest products $ij$,
$i,j\in\Z^+$.  We begin by assuming that there exists an integer $M$ such that the number of pairs $(i,j)$ satisfying %$i,j\in\Z^+$ and 
$ij\le M$ is exactly $K$: 
\begin{equation}\label{eqn:squarepair}
2\sum_{i=1}^{\lfloor M^{1/2}\rfloor}\lfloor {M}/{i}\rfloor  -\lfloor M^{1/2}\rfloor^2=K. 
\end{equation}
To analyze $\sum_{v=K+1}^{\infty}\lambda_{v}$, we first quantify $M$.
An upper bound on $\sum_{i=1}^{\lfloor M^{1/2}\rfloor}\lfloor {M}/{i}\rfloor$ is
\begin{align*}
     \sum_{i=1}^{\lfloor M^{1/2}\rfloor}\Big\lfloor \frac{M}{i}\Big\rfloor
\le  \sum_{i=1}^{\lfloor M^{1/2}\rfloor} \frac{M}{i}
=    M\sum_{i=1}^{\lfloor M^{1/2}\rfloor} \frac{1}{i}
\le  M\Big(\log\lfloor M^{1/2}\rfloor+1\Big)
\le  M\Big(\frac12\log M+1\Big),
\end{align*}
and a lower bound is
\begin{align*}
     \sum_{i=1}^{\lfloor M^{1/2}\rfloor}\Big\lfloor \frac{M}{i}\Big\rfloor
\ge\; &\sum_{i=1}^{\lfloor M^{1/2}\rfloor} \Big(\frac{M}{i}-1\Big)
=    M\sum_{i=1}^{\lfloor M^{1/2}\rfloor} \frac{1}{i} - \lfloor M^{1/2}\rfloor\\
\ge\; &M\log\lfloor M^{1/2}\rfloor - \lfloor M^{1/2}\rfloor
\ge  M\log(M^{1/2}-1) - M^{1/2}.
\end{align*}
Thus we have $M\log M \asymp K$, which implies that $M\asymp K/\log K$. Then we obtain
\begin{align*}
\sum_{v=K+1}^{\infty}\lambda_v
\asymp\; &\sum_{i=1}^{\lfloor M^{1/2}\rfloor}\sum_{j=\lfloor {M}/{i}\rfloor+1}^{\infty}\frac{1}{i^2j^2}
+\sum_{j=1}^{\lfloor M^{1/2}\rfloor}\sum_{i=\lfloor {M}/{j}\rfloor+1}^{\infty}\frac{1}{i^2j^2}
+\sum_{i=\lfloor M^{1/2}\rfloor+1}^{\infty}\;\sum_{j=\lfloor M^{1/2}\rfloor+1}^{\infty}\frac{1}{i^2j^2}\\
\asymp\; &\sum_{i=1}^{\lfloor M^{1/2}\rfloor}\frac{1}{i^2(M/i)}
+\sum_{j=1}^{\lfloor M^{1/2}\rfloor}\frac{1}{(M/j)j^2}
+\frac{1}{(M^{1/2})(M^{1/2})}\\
\asymp\; &2\Big\{\frac{\log(M^{1/2})}{M}\Big\}+\frac{1}{M}\asymp\frac{(\log K)^2}{K}.
\end{align*}

If there is no integer $M$ such that \eqref{eqn:squarepair} holds, then we pick the largest integer $M_1$ and the smallest integer $M_2$ such that
\[
2\sum_{i=1}^{\lfloor M_1^{1/2}\rfloor}\Big\lfloor \frac{M_1}{i}\Big\rfloor  -\lfloor M_1^{1/2}\rfloor^2< K< 2\sum_{i=1}^{\lfloor M_2^{1/2}\rfloor}\Big\lfloor \frac{M_2}{i}\Big\rfloor  -\lfloor M_2^{1/2}\rfloor^2,
\]
and let $K_1$ and $K_2$ denote the left-hand side and the right-hand side, respectively. One can verify that $K_1>K/2$ and $K_2<2K$ for all sufficiently large $K$. Then we have
\begin{align*}
&\sum_{v=K+1}^{\infty}\lambda_v\le\sum_{v=K_1+1}^{\infty}\lambda_v \asymp\frac{(\log K_1)^2}{K_1}\le\frac{2(\log K)^2}{K}\\
\text{and}~~~&\sum_{v=K+1}^{\infty}\lambda_v\ge\sum_{v=K_2+1}^{\infty}\lambda_v \asymp\frac{(\log K_2)^2}{K_2}\ge\frac{(\log K)^2}{2K}.
\end{align*}
Therefore, the asymptotic result for $\sum_{v=K+1}^{\infty}\lambda_{v}$ given in \eqref{eqn:asy} still holds. 
%Finally, using Theorem \ref{thm:md}, uniformly over $x\in(0,o(n^{1/8-\delta}))$, we have
%\begin{equation*}
%\frac{\Pr\Big\{\frac{1}{10}(n-1)\hat D_{12}>x\Big\}}{\Pr\Big\{\sum_{i=1}^\infty\sum_{j=1}^\infty\frac{1}{10\pi^4i^2j^2}(Z_{i,j}^2-1)>x\Big\}}=1+o(1),
%\end{equation*}
%where $\{Z_{i,j}, i,j\in\N^+\}$ are i.i.d. standard Gaussian. We have finished the proof for $\hat D_{12}$.
\end{proof}

\subsubsection{Proof of Lemma \ref{lem:normal}}

\begin{proof}[Proof of Lemma \ref{lem:normal}]
  Again we only prove the claim for Hoeffding's $D$;
  Blum--Kiefer--Rosenblatt's $R$ and Bergsma--Dassios--Yanagimoto's
  $\tau^*$ can be treated similarly.
%that $D_{jk}\asymp \Sigma_{jk}^2$ and any $q$-dimensional normal distribution belongs to $\mathcal{D}(\gamma_D,q;h_D)$.
We first establish the fact that $D_{jk}\asymp \Sigma_{jk}^2$ as $\Sigma_{jk}\to 0$.
 Let $\{(X_{ij},X_{ik})^{\top}:i\in[5]\}$ be a collection of independent and identically distributed random vectors that follow a bivariate normal distribution with mean $(0,0)^{\top}$ and covariance matrix
\begin{equation*}
\bigg[\begin{matrix}1 & \Sigma_{jk}\\ \Sigma_{jk} & 1\end{matrix}\bigg].
\end{equation*}
We have 
\begin{equation*}
D_{jk}=\E_{jk}h_D=\int h_D(x_{1j},x_{1k},\dots,x_{5j},x_{5k})\phi(x_{1j},x_{1k},\dots,x_{5j},x_{5k};\Sigma_{jk})\prod_{i=1}^5 dx_{ij}\prod_{i=1}^5 dx_{ik},
\end{equation*}
where
\begin{equation*}
\phi(x_{1j},x_{1k},\dots,x_{5j},x_{5k};\Sigma_{jk})=\prod_{i=1}^5 \phi(x_{ij},x_{ik};\Sigma_{jk}),
\end{equation*}
and
\begin{equation*}
\phi(x_{ij},x_{ik};\Sigma_{jk})=\frac{1}{2\pi(1-\Sigma_{jk}^2)^{1/2}}\exp \Big\{-\frac{x_{ij}^2+x_{ik}^2-2\Sigma_{jk}x_{ij}x_{ik}}{2(1-\Sigma_{jk}^{2})}\Big\}
\end{equation*}
is the joint density of $(X_{ij},X_{ik})^{\top}$. Notice that $D_{jk}$ is smooth with respect to $\Sigma_{jk}$:
\begin{equation*}
\frac{\partial^{s}D_{jk}}{\partial\Sigma_{jk}^s}=\int h_D(x_{1j},x_{1k},\dots,x_{5j},x_{5k})\frac{\partial^s\phi(x_{1j},x_{1k},\dots,x_{5j},x_{5k};\Sigma_{jk})}{\partial\Sigma_{jk}^s}\prod_{i=1}^5 dx_{ij}\prod_{i=1}^5 dx_{ik}.
\end{equation*}
In order to prove $D_{jk}\asymp\Sigma_{jk}^2$, it suffices to establish that $D_{jk}=0$ when $\Sigma_{jk}=0$, the first derivative of $D_{jk}$ with respect to $\Sigma_{jk}$ is $0$ at $\Sigma_{jk}=0$, and the second derivative of $D_{jk}$ with respect to $\Sigma_{jk}$ is $5/\pi^2$ 
at $\Sigma_{jk}=0$, which can be confirmed by a lengthy but straightforward computation.

Now we turn to our claim.
%We have
%\begin{equation*}
%\E_{jk}\{h^{(1)}(\cdot;\Pr_{jk})\}^2=\int \{h^{(1)}(x_{1j},x_{1k};\Pr_{jk})\}^2\phi(x_{1j},x_{1k};\Sigma_{jk}) dx_{1j}dx_{1k},
%\end{equation*}
%where $\Pr_{jk}$ is totally determined by $\Sigma_{jk}$.
Recall that $\Var_{jk}\{h_D^{(1)}(\cdot;\Pr_{jk})\}=0$ when $\Sigma_{jk}=0$. We will show that the first-order term in the Taylor series of $\Var_{jk}\{h_D^{(1)}(\cdot;\Pr_{jk})\}$ with respect to $\Sigma_{jk}$ is also $0$. Suppose, for contradiction, the first-order coefficient (denoted by $a_1$) in the Taylor series of $\Var_{jk}\{h_D^{(1)}(\cdot;\Pr_{jk})\}$ is not $0$, then for $\Sigma_{jk}$ in a sufficiently small neighborhood of $0$, $\Var_{jk}\{h_D^{(1)}(\cdot;\Pr_{jk})\}<0$ for $\Sigma_{jk}<0$ if $a_1>0$, and for $\Sigma_{jk}>0$ if $a_1<0$, which contradicts the definition of $\Var_{jk}\{h_D^{(1)}(\cdot;\Pr_{jk})\}$. This together with $\E_{jk}h_D\asymp\Sigma_{jk}^2$ completes the proof.
\end{proof}

\subsubsection{Proof of Theorem \ref{thm:power}}

\begin{proof}[Proof of Theorem \ref{thm:power}]
It is clear that we only have to consider $\max_{j<k} U_{jk}=C_\gamma(\log p/n)$ for some sufficiently large $C_\gamma$. The main idea here is to bound $\max_{j<k} \hat U_{jk}-\max_{j<k} U_{jk}$ with high probability. To do this, we first construct a concentration inequality for $|\hat U_{jk}-U_{jk}|$.
The Hoeffding decomposition of the difference is
\begin{equation}\label{eqn:pow-H}
\hat U_{jk}-U_{jk}=\frac{m}n\sum_{i=1}^n h^{(1)}(\mX_{i,\{j,k\}};\Pr_{jk})+\sum_{\ell=2}^{m}\mbinom{m}{\ell}H_n^{(\ell)}(\cdot;\Pr_{jk}).
\end{equation}
For controlling the first term in \eqref{eqn:pow-H}, recall that $\lVert h\rVert_\infty\le b_1<\infty$, and then $h^{(1)}(\cdot;\Pr_{jk})=h_1(\cdot;\Pr_{jk})-\E h$ is bounded by $2b_1$ almost surely and $\E h^{(1)}(\cdot;\Pr_{jk})=0$. We then apply Bernstein's inequality, giving
\begin{equation}\label{eqn:pow-Bern}
\Pr \Big\{\frac{m}{n}\Big|\sum _{i=1}^{n}h^{(1)}(\cdot;\Pr_{jk})\Big|>t_1\Big\}\leq 
2\exp \Big(-\frac {n(t_1/m)^{2}}{2[\Var_{jk}\{h^{(1)}(\cdot;\Pr_{jk})\}+{2b_1(t_1/m)}/{3}]}\Big).
\end{equation}
By the definition of the distribution family $\mathcal{D}(\gamma,p;h)$, we have 
\[
\Var_{jk}\{h^{(1)}(\cdot;\Pr_{jk})\}\le \gamma \E_{jk}h=\gamma U_{jk}\le \gamma C_\gamma(\log p/n). 
\]
Plugging this into \eqref{eqn:pow-Bern} and taking $t_1=C_1(\log p/n)$, where $C_1$ is a constant to be specified later, yields 
\begin{align}\label{eqn:piece1}
      & \Pr \Big\{\frac{m}{n}\Big|\sum _{i=1}^{n}h^{(1)}(\cdot;\Pr_{jk})\Big|>C_1\frac{\log p}{n}\Big\}\notag\\
\le\; & 2\exp \Big\{-\frac {C_1^2\log p}{2(m^2\gamma C_\gamma+2m{b_1C_1}/{3})}\Big\}=2\Big(\frac{1}{p}\Big)^{{C_1^2}/(2m^2\gamma C_\gamma+ 4m{b_1C_1}/{3})}.
\end{align}
We then handle the remaining term.  By Proposition 2.3(c) in \citet{MR1235426}, there exist absolute constants $C_\ell',C_\ell''>0$ such that for all $t\in(0,1]$, $2\le\ell\le m$,
\begin{align*}
      &\Pr(|H_n^{(\ell)}(\cdot;\Pr_{jk})|\ge t)\le C_\ell'\exp\Big\{-C_{\ell}''n\Big(\frac{t}{\lVert h^{(\ell)}(\cdot;\P_{jk})\rVert_{\infty}}\Big)^{2/\ell}\Big\}\\
\le\; &C_\ell'\exp\Big\{-C_{\ell}''n\Big(\frac{t}{2^{\ell}b_1}\Big)^{2/\ell}\Big\}\le C_\ell'\exp\Big(-\frac{C_\ell''nt}{4b_1^{2/\ell}}\Big),
\end{align*}
which further implies that
\begin{align*}
\Pr\Big\{\Big|\sum_{\ell=2}^k\mbinom{m}{\ell}H_n^{(\ell)}(\cdot;\Pr_{jk})\Big|\ge t_2\Big\}
&\le \sum_{\ell=2}^{m}\Pr\Big\{\mbinom{m}{\ell}|H_n^{(\ell)}(\cdot;\Pr_{jk})|\ge t_2\frac{4b_1^{2/\ell}\binom{m}{\ell}/C_\ell''}{\sum_{\ell=2}^{m}4b_1^{2/\ell}\binom{m}{\ell}/C_\ell''}\Big\}\\
&\le\Big(\sum_{\ell=2}^{m}C_\ell'\Big)\exp\Big\{-\frac{nt_2}{\sum_{\ell=2}^{m}4b_1^{2/\ell}\binom{m}{\ell}/C_\ell''}\Big\}.
\end{align*}
Taking $t_2=C_2(\log p/n)$, where $C_2$ is another constant to be specified later, we have
\begin{equation}\label{eqn:piece2}
\Pr\Big\{\Big|\sum_{\ell=2}^k\mbinom{m}{\ell}H_n^{(\ell)}(\cdot;\Pr_{jk})\Big|\ge C_2(\log p/n)\Big\}\le \Big(\sum_{\ell=2}^{m}C_\ell'\Big)\Big(\frac1p\Big)^{C_2/\{\sum_{\ell=2}^{m}4b_1^{2/\ell}\binom{m}{\ell}/C_\ell''\}}.
\end{equation}
Putting \eqref{eqn:piece1} and \eqref{eqn:piece2} together, and choosing 
\[C_1=2mb_1+m(4b_1^2+6\gamma C_\gamma)^{1/2}
~~~\text{and}~~~
C_2=12\sum_{\ell=2}^m \frac{b_1^{2/\ell}\binom{m}{\ell}}{C_\ell''},
\]
 we deduce
\begin{equation*}
\Pr\Big[|\hat U_{jk}-U_{jk}|\ge \Big\{2mb_1+m(4b_1^2+6\gamma C_\gamma)^{1/2}+12\sum_{\ell=2}^m \frac{b_1^{2/\ell}\binom{m}{\ell}}{C_\ell''}\Big\}\frac{\log p}{n}\Big]
\le \Big(2+\sum_{\ell=2}^{m}C_\ell'\Big)\frac1{p^3}.
\end{equation*}
Then using Slutsky's argument gives
\begin{equation*}
\Pr\Big[\max_{j<k}|\hat U_{jk}-U_{jk}|\ge \Big\{2mb_1+m(4b_1^2+6\gamma C_\gamma)^{1/2}+12\sum_{\ell=2}^m \frac{b_1^{2/\ell}\binom{m}{\ell}}{C_\ell''}\Big\}\frac{\log p}{n}\Big]
\le \frac{2+\sum_{\ell=2}^{m}C_\ell'}{2}\cdot\frac1{p},
\end{equation*}
which implies that, with probability at least $1-(1+\sum_{\ell=2}^{m}C_\ell'/2)p^{-1}$,
\begin{equation*}
\max_{j<k}|\hat U_{jk}-U_{jk}|\le \Big\{2mb_1+m(4b_1^2+6\gamma C_\gamma)^{1/2}+12\sum_{\ell=2}^m \frac{b_1^{2/\ell}\binom{m}{\ell}}{C_\ell''}\Big\}\frac{\log p}{n}.
\end{equation*}
Hence for $n\ge 2$, we have with probability no smaller than $1-(1+\sum_{\ell=2}^{m}C_\ell'/2)p^{-1}$,
\begin{align*}
      &\max_{j<k}\hat U_{jk}\ge\max_{j<k} U_{jk}-\max_{j<k}|\hat U_{jk}-U_{jk}|\\
\ge\; &\Big\{C_\gamma-2mb_1-m(4b_1^2+6\gamma C_\gamma)^{1/2}-12\sum_{\ell=2}^m \frac{b_1^{2/\ell}\binom{m}{\ell}}{C_\ell''}\Big\}\frac{\log p}{n}\ge \frac{5\lambda_1\binom{m}{2}\log p}{n-1},
\end{align*}
where the last inequality is satisfied by choosing $C_\gamma$ large
enough. Accordingly, for any given $Q_\alpha$, the probability that
\begin{equation*}
\frac{n-1}{\lambda_1\binom{m}{2}}\max_{j<k}\hat U_{jk}\ge 5\log p > 4\log p+(\mu_1-2)\log\log p-\frac{\Lambda}{\lambda_1}+Q_\alpha
\end{equation*}
tends to $1$ as $p$ goes to infinity.  The proof is thus completed.
\end{proof}

%\begin{proof}[Proof of Corollary \ref{crl:power_eg}]
%This is a direct consequence of Theorem \ref{thm:power}. %The proof for test $T_{D,\alpha}$ is finished. The proofs for test $T_{R,\alpha}$ and $T_{\tau^*,\alpha}$ can be done similarly, and thus are omitted.
%\end{proof}

\subsubsection{Proof of Theorem \ref{crl:normal}}

\begin{proof}[Proof of Theorem \ref{crl:normal}]
In view of Corollary \ref{crl:power_eg}, the results follow from Lemma \ref{lem:normal} and the fact that $D_{jk}, R_{jk}, \tau^*_{jk}\asymp \Sigma_{jk}^2$ as $\Sigma_{jk}\to 0$, which has been shown in the proof of Lemma \ref{lem:normal}.
\end{proof}

\subsection{Proofs for Section \ref{sec:discussion} of the main paper}

\subsubsection{Proof of Theorem \ref{thm:ya}}

\begin{proof}[Proof of Theorem \ref{thm:ya}]
The proof of Theorem \ref{thm:ya} hinges on
the identity \eqref{eq:identity}, the fact that random vectors of
continuous margins almost surely have no ties among the values of each
coordinate, and that $\E h_D\geq 0$ and $\E h_R\geq 0$ (see \citet[p.~547]{MR0029139} and \citet[p.~490]{MR0125690}).    

The identity \eqref{eq:identity} now gives that $\E h_{\tau^*}\ge 0$ and that
$\E h_{\tau^*}=0$ if and only if  $\E h_D=\E h_R=0$, which in turn implies
independence of the considered pair of random variables.
\end{proof}

\subsubsection{Proof of Proposition \ref{prop:circ}}

\begin{proof}[Proof of Proposition \ref{prop:circ}]

The copula of $(X,Y)^\top$ is given by \citet[p.~56]{MR2197664}:
\begin{equation*}
C(u,v)=\begin{cases}
\min(u,v),                  &  \text{ if }|u-v|\ge \frac{1}{2},\\
\max(0,u+v-1),              &  \text{ if }|u+v-1|\ge \frac{1}{2},\\
\frac{u+v}{2}-\frac{1}{4},  &  \text{ otherwise}.
\end{cases}
\end{equation*}
We summarize the copula in Figure \ref{fig:copula1}.
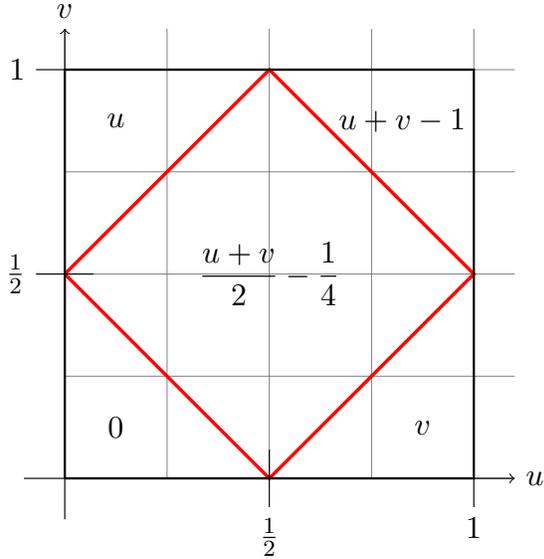
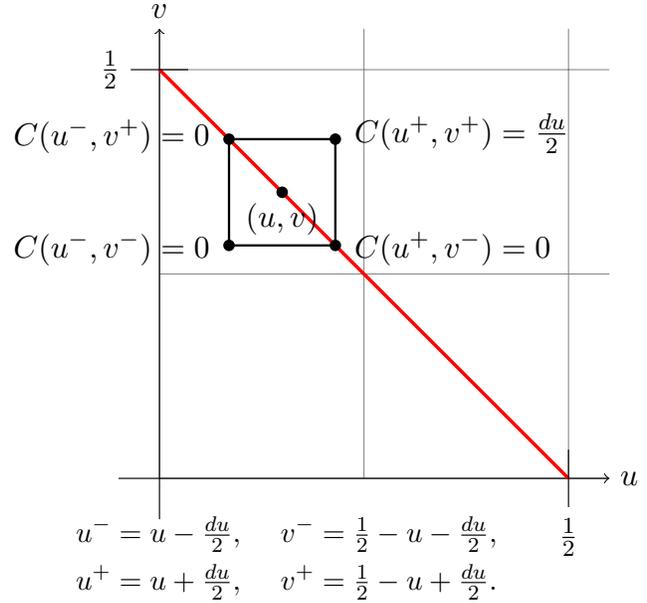
\begin{figure}[!htbp]
%\centering
\begin{subfigure}[t]{0.45\textwidth}
\resizebox{!}{3in}{
\begin{tikzpicture}[scale=5]
  \draw[style=help lines,step=0.25cm] (0,0) grid +(1.1,1.1);

  \draw[->] (-0.1,0) -- (1.1,0) node[right] {$u$};
  \draw[->] (0,-0.1) -- (0,1.1) node[above] {$v$};
  \draw[-,very thick,red] (0,0.5) -- (0.5,1) -- (1,0.5) -- (0.5,0) -- (0, 0.5);
  \draw[-,thick] (0,0) -- (0,1) -- (1,1) -- (1,0) -- (0,0);
  \draw (0.125,0.125) node {$0$};
  \draw (0.125,0.875) node {$u$};
  \draw (0.875,0.125) node {$v$};
  \draw (0.825,0.875) node {$u+v-1$};
  \draw (0.5,0.5) node {$\displaystyle \frac{u+v}{2}-\frac{1}{4}$};
  
  \foreach \x/\xtext in {0.5/\frac{1}{2}, 1/1}
  \draw[shift={(\x,0)}] (0pt,2pt) -- (0pt,-2pt) node[below] {$\xtext$};
  \foreach \y/\ytext in {0.5/\frac{1}{2}, 1/1}
  \draw[shift={(0,\y)}] (2pt,0pt) -- (-2pt,0pt) node[left] {$\ytext$};
\end{tikzpicture}}
\caption{The copula of the circular uniform distribution with its support marked in red.}\label{fig:copula1}
\end{subfigure}
~
\begin{subfigure}[t]{0.45\textwidth}
\resizebox{!}{3in}{
\begin{tikzpicture}[scale=10]
  \draw[style=help lines,step=0.25cm] (0,0) grid +(0.55,0.55);

  \draw[->] (-0.05,0) -- (0.55,0) node[right] {$u$};
  \draw[->] (0,-0.05) -- (0,0.55) node[above] {$v$};
  \draw[-,very thick,red] (0.5,0) -- (0, 0.5);
  
  \draw[-,thick] (0.085,0.285) -- (0.215,0.285) -- (0.215,0.415) -- (0.085,0.415) -- (0.085,0.285);
  \fill[black] (0.085,0.285) circle (.2pt);
  \draw (0.075,0.315) node[below left] {$C(u^{-},v^{-})=0$};
  \fill[black] (0.215,0.285) circle (.2pt);
  \draw (0.225,0.315) node[below right] {$C(u^+,v^{-})=0$};
  \fill[black] (0.085,0.415) circle (.2pt);
  \draw (0.075,0.385) node[above left] {$C(u^{-},v^{+})=0$};
  \fill[black] (0.215,0.415) circle (.2pt);
  \draw (0.225,0.385) node[above right] {$C(u^+,v^{+})=\frac{du}{2}$};
  \fill[black] (0.15,0.35) circle (.2pt);
  \draw (0.15,0.35) node[below] {$(u,v)$};
  
  \foreach \x/\xtext in {0.5/\frac{1}{2}}
  \draw[shift={(\x,0)}] (0pt,1pt) -- (0pt,-1pt) node[below] {$\xtext$};
  \foreach \y/\ytext in {0.5/\frac{1}{2}}
  \draw[shift={(0,\y)}] (1pt,0pt) -- (-1pt,0pt) node[left] {$\ytext$};
\end{tikzpicture}}
\begin{align*}
\\[-4.75em]
~~~u^{-}&=u-\tfrac{du}{2}, &\!\!\! v^{-}&=\tfrac{1}{2}-u-\tfrac{du}{2},\\
~~~u^{+}&=u+\tfrac{du}{2}, &\!\!\! v^{+}&=\tfrac{1}{2}-u+\tfrac{du}{2}.
\\[-2.6em]
\end{align*}
\caption{Integral on part of the support.}\label{fig:copula2}
\end{subfigure}
\caption{The copula of the circular uniform distribution.}
\end{figure}

Since both $X$ and $Y$ are continuous, by the arguments in \citet{MR619291}, we obtain
\begin{align*}
\E h_D&=30\int \{F(x,y)-F_1(x)F_2(y)\}^2 dF(x,y)\\
                          &=30\int \{C(u,v)-uv\}^2 dC(u,v)\\
\text{and}~~~
\E h_R&=90\int \{F(x,y)-F_1(x)F_2(y)\}^2 dF_1(x)dF_2(y)\\
                          &=90\int \{C(u,v)-uv\}^2 du dv.
\end{align*}

We first compute $\E h_D$. Notice that $\partial^2 C(u,v)/\partial u\partial v=0$ in $[0,1]\times[0,1]$ except for the support of $C(u,v)$ (marked in red in Figure \ref{fig:copula1}). Therefore, we only need to compute the integral on the support consisting of four line segments. In Figure \ref{fig:copula2}, we illustrate how to find $dC(u,v)$ on the line segment from $(0,1/2)$ to $(1/2,0)$ (denoted by $\mathcal{C}_1$). We have
\begin{equation*}
dC(u,v)=C(u^{+},v^{+})-C(u^{+},v^{-})-C(u^{-},v^{+})+C(u^{-},v^{-})=\frac{du}{2},
\end{equation*}
and thus the integral on the line segment $\mathcal{C}_1$ is given by
\begin{equation*}
30\int_{\mathcal{C}_1} \{C(u,v)-uv\}^2 dC(u,v)=30\int_0^{1/2} \Big\{0-u\Big(\frac{1}{2}-u\Big)\Big\}^2\frac{du}{2}=\frac{1}{64}.
\end{equation*}
We can evaluate the integral on the other three line segments (denoted
by $\mathcal{C}_2,\mathcal{C}_3,\mathcal{C}_4$, respectively)
similarly, and we find
\begin{equation*}
\E h_D=30\int_{\mathcal{C}_1+\mathcal{C}_2+\mathcal{C}_3+\mathcal{C}_4} \{C(u,v)-uv\}^2 dC(u,v)=\frac{1}{16}.
\end{equation*}

The computation of $\E h_R=90\int \{C(u,v)-uv\}^2 du dv=1/16$ is
standard, and we omit details. Finally, using the identity \eqref{eq:identity}, we deduce that $\E h_{\tau^*}=1/16$.
\end{proof}

\section{More comments on $\tau^*$}

First of all, we show that the identity \eqref{eq:identity} in the main paper may be false when ties
 exist. 
 
\begin{example}\label{remark:BDY}
If we take $\mz_i=(\lfloor(i+2)/3\rfloor,i)^{\top}$
 for $i\in[6]$, then
\begin{align*}
&\mbinom{6}{5}^{-1}\sum_{1\leq i_1<\cdots<i_5\leq 6}h_D(\mz_{i_1},\ldots,\mz_{i_5})=1/2, ~~~h_R(\mz_1,\ldots,\mz_6)=3/2, \\
\text{and}~~~ 
&\mbinom{6}{4}^{-1}\sum_{1\leq i_1<\cdots<i_4\leq 6}h_{\tau^*}(\mz_{i_1},\ldots,\mz_{i_4})=3/5.
\end{align*}
\end{example}

  In view of Example~\ref{remark:BDY}, the proof of Theorem
  \ref{thm:ya} cannot be directly extended to pairs consisting of both
  discrete and continuous random variables, and the question if
  Bergsma--Dassios's conjecture is correct remains open in
  that regard. However, %Theorem \ref{thm:ya} does imply that
  %$\E h_{\tau^*}> 0$ when $\mZ$ is uniformly distributed on the
  %unit-circle in $\R^2$, giving an affirmative answer to a comment
 % raised in \citet[p.~551]{MR3842884}; see also Proposition
  %\ref{prop:circ}.  Moreover, 
  by the
  Lebesgue decomposition theorem, in order to prove 
  Bergsma--Dassios's conjecture it suffices to  prove
  the case where the pair follows a mixture of discrete and singular
  measures.

We now provide a second proof  of Theorem~\ref{thm:ya} {for the absolute continuity case only. It} connects the correlation measures raised by
\citet{MR3178526} and the one in the
proof of Proposition~9 in \citet{zbMATH03366369}.  We believe the
resulting alternative representation of the population $\tau^*$ is of
independent interest, e.g., from the point of view of multivariate
extensions of $\tau^*$ as considered by \citet{MR3842884}.

\begin{proposition}\label{prop:future}
For any pair of absolutely continuous random variables $(X,Y)^\top\in\mathbb{R}^2$ with joint distribution function $F(x,y)$ and marginal distribution functions $F_1(x),F_2(y)$, we have 
\begin{align*}
                       &\frac{1}{18}\,\E h_{\tau^*}\\
\stackrel  {(i)}{=}\;  &\int F^2 d(F+F_1F_2)-\int F^2 d(FF_1)-2\int FF_1 d(FF_2)+\int FF_1d(F^2)+\frac1{18}\\
\stackrel {(ii)}{=}\,  &\int F^2 dF-2\int FF_1F_2dF+2\int F^2dF_1dF_2-\frac19\\
\stackrel{(iii)}{=}    &\int(F-F_1F_2)^2 dF + 2\int(F-F_1F_2)^2 dF_1dF_2\\
                 =\,\, &\frac{1}{30}\,\E h_D + \frac{1}{45}\,\E h_R,
\end{align*}
%Here $F,F_1,F_2$ are abbreviations of joint distribution function $F(x,y)$ and marginal distribution functions $F_1(x),F_2(y)$, respectively.
where the term on the righthand side of the identity (ii) is Yanagimoto's correlation measure.
\end{proposition}

\begin{proof}[Proof of Proposition~\ref{prop:future}]
We prove identities $(i)$--$(iii)$ sequentially. %The proof is mainly an application of elementary probability theory.
Let $\Psi_1,\Psi_2,\Psi_3$ denote the expressions on the right-hand side of identities $(i)$, $(ii)$, $(iii)$, respectively.\\

{\bf Identity $(i)$. }
Let $\{(X_i,Y_i)^\top\}_{i\in[4]}$ be four independent realizations of $(X,Y)^\top$. For Bergsma--Dassios--Yanagimoto's $\tau^*$, we have, by Equation (6) in \citet{MR3178526},
\begin{align}\label{eqn:BDiden}
\frac1{18}\,\E h_{\tau^*}=\; &\frac13\,\Pr\{\max(X_1,X_2)<\min(X_3,X_4)\textand\max(Y_1,Y_2)<\min(Y_3,Y_4)\}\notag\\
    &+\frac13\,\Pr\{\max(X_1,X_2)<\min(X_3,X_4)\textand\max(Y_3,Y_4)<\min(Y_1,Y_2)\}\notag\\
    &-\frac23\,\Pr\{\max(X_1,X_2)<\min(X_3,X_4)\textand\max(Y_1,Y_3)<\min(Y_2,Y_4)\}.
\end{align}
We study the three terms in \eqref{eqn:BDiden} separately, starting from the first term. Using Fubini's theorem, we get
\begin{align}\label{eqn:compu1}
    &\Pr\{\max(X_1,X_2)<\min(X_3,X_4)\textand\max(Y_1,Y_2\}<\min(Y_3,Y_4)\}\notag\\
=\; &\int\Pr\{\max(X_1,X_2)<x\textand\max(Y_1,Y_2)<y\}d\Pr\{\min(X_3,X_4)\le x\textand\min(Y_3,Y_4)\le y\}\notag\\
=\; &\int F(x,y)^2d\Pr\{\min(X_3,X_4)\le x\textand\min(Y_3,Y_4)\le y\},
\end{align}
where
%\begin{align*}
%    &\Pr\{\min(X_3,X_4)\le x\textand\min(Y_3,Y_4)\le y\}\\
%%=\; &\Pr\{(\underbrace{X_3\le x}_{A})\cup(\underbrace{X_4\le x}_{B})\}\cap \{(\underbrace{Y_3\le y}_{C})\cup(\underbrace{Y_4\le y}_{D})\}\\
%%=\; &\Pr[\{A\cap(C\cup D)\}\cup \{B\cap(C\cup D)\}]\\
%%=\; &\Pr\{(\underbrace{A\cap C}_{I})\cup (\underbrace{A\cap D}_{II})\cup (\underbrace{B\cap C}_{III})\cup (\underbrace{B\cap D}_{IV})\},
%=\; &\Pr\{(A\cup B)\cap(C\cup D)\}=\Pr(I\cup II\cup III\cup IV)
%\end{align*}
\begin{equation*}
\Pr\{\min(X_3,X_4)\le x\textand\min(Y_3,Y_4)\le y\}
=\Pr\{(A\cup B)\cap(C\cup D)\}=\Pr(I\cup II\cup III\cup IV)
\end{equation*}
and $A:=\{X_3\le x\}$, $B:=\{X_4\le X\}$, $C:=\{Y_3\le y\}$, $D:=\{Y_4\le y\}$,
\begin{align*}
%A  &:=\{X_3\le x\}, & B&:=\{X_4\le X\} & C&:=\{Y_3\le y\}, & D&:=\{Y_4\le y\},\\
I  &:=A\cap C=\{X_3\le x,Y_3\le y\},      & II &:=A\cap D=\{X_3\le x,Y_4\le y\},\\
III&:=B\cap C=\{X_4\le x,Y_3\le y\},      & IV &:=B\cap D=\{X_4\le x,Y_4\le y\}.
\end{align*}
From the inclusion--exclusion principle, we obtain
\begin{align}\label{eqn:incexc}
    &\Pr\{\min(X_3,X_4)\le x\textand\min(Y_3,Y_4)\le y\}\notag\\
=\; &\Pr(I)+\Pr(II)+\Pr(III)+\Pr(IV)\notag\\
    &-\Pr(I\cap II)-\Pr(I\cap III)-\Pr(I\cap IV)-\Pr(II\cap III)-\Pr(II\cap IV)-\Pr(III\cap IV)\notag\\
    &+\Pr(I\cap II\cap III)+\Pr(I\cap II\cap IV)+\Pr(I\cap III\cap IV)+\Pr(II\cap III\cap IV)\notag\\
    &-\Pr(I\cap II\cap III\cap IV)\notag\\
=\; &F+F_1F_2+F_1F_2+F-FF_2-FF_1-F^2-F^2-FF_1-FF_2+F^2+F^2+F^2+F^2-F^2\notag\\
=\; &2F+2F_1F_2-2FF_1-2FF_2+F^2.
\end{align}
Plugging \eqref{eqn:incexc} into \eqref{eqn:compu1} implies that
\begin{align}\label{eqn:term1}
    &\Pr\{\max(X_1,X_2)<\min(X_3,X_4)\textand\max(Y_1,Y_2\}<\min(Y_3,Y_4)\}\notag\\
=\; &\int F^2 d(2F+2F_1F_2-2FF_1-2FF_2+F^2).
\end{align}

The second term in \eqref{eqn:BDiden} can be written as
\begin{align}\label{eqn:compu2}
    &\Pr\{\max(X_1,X_2)<\min(X_3,X_4)\textand\max(Y_3,Y_4)<\min(Y_1,Y_2)\}\notag\\
=\; &\Pr\{\max(X_1,X_2)<\min(X_3,X_4)\}\notag\\
    &-\Pr\{\max(X_1,X_2)<\min(X_3,X_4)\textand\min(Y_1,Y_2)\le\max(Y_3,Y_4)\}\notag\\
=\; &\int \Pr\{\max(X_1,X_2)<x\}d\Pr\{\min(X_3,X_4)\le x\}\notag\\
    &-\Pr\{\max(X_1,X_2)<x\textand\min(Y_1,Y_2)\le y\}d\Pr\{\min(X_3,X_4)\le x\textand\max(Y_3,Y_4)\le y\},
\end{align}
where we have
\begin{align}\label{eqn:sos}
    \Pr\{\min(X_3,X_4)\le x\}=\Pr(A\cup B)=2F_1-F_1^2,
\end{align}
and
\begin{align}
    &\Pr\{\max(X_1,X_2)<x\textand\min(Y_1,Y_2)\le y\}\notag\\
=\; &\Pr\{\max(X_1,X_2)<x\textand Y_1\le y\}+\Pr\{\max(X_1,X_2)<x\textand Y_2\le y\}\notag\\
    &-\Pr\{\max(X_1,X_2)<x\textand\max(Y_1,Y_2)\le y\}\notag\\
=\;&2FF_1-F^2,
\end{align}
and
\begin{align}\label{eqn:sss}
    &\Pr(\min\{X_3,X_4\}\le x\textand\max\{Y_3,Y_4\}\le y)\notag\\
=\; &\Pr[\{(X_3\le x)\cup(X_4\le x)\}\cap \{(Y_3\le y)\cap(Y_4\le y)\}]\notag\\
=\; &\Pr[\{A\cap(C\cap D)\}\cup \{B\cap(C\cap D)\}]\notag\\
=\; &2FF_2-F^2.
\end{align}
Plugging \eqref{eqn:sos}--\eqref{eqn:sss} into \eqref{eqn:compu2} yields
\begin{align}\label{eqn:term2}
    &\Pr\{\max(X_1,X_2)<\min(X_3,X_4)\textand\max(Y_3,Y_4)<\min(Y_1,Y_2)\}\notag\\
=\; &\int F_1^2 d(2FF_1-F^2)-\int (2FF_1-F^2) d(2FF_2-F^2).
\end{align}

Next we handle the third term in \eqref{eqn:BDiden}. We have by symmetry that
\begin{align}\label{eqn:term3_sym1}
    &\Pr\{\max(X_1,X_2)<\min(X_3,X_4)\textand\max(Y_1,Y_3)<\min(Y_2,Y_4)\}\notag\\
=\; &\Pr\{\max(X_1,X_2)<\min(X_4,X_3)\textand\max(Y_1,Y_4)<\min(Y_2,Y_3)\}\notag\\
=\; &\Pr\{\max(X_2,X_1)<\min(X_3,X_4)\textand\max(Y_2,Y_3)<\min(Y_1,Y_4)\}\notag\\
=\; &\Pr\{\max(X_2,X_1)<\min(X_4,X_3)\textand\max(Y_2,Y_4)<\min(Y_1,Y_3)\}.
\end{align}
We also notice that
\begin{align}\label{eqn:term3_sym2}
    &\Pr\{\max(X_1,X_2)<\min(X_3,X_4)\notag\\
=\; &\Pr\{\max(X_1,X_2)<\min(X_3,X_4)\textand\max(Y_1,Y_2)<\min(Y_3,Y_4)\}\notag\\
    &+\Pr\{\max(X_1,X_2)<\min(X_3,X_4)\textand\max(Y_3,Y_4)<\min(Y_1,Y_2)\}\notag\\
    &+\Pr\{\max(X_1,X_2)<\min(X_3,X_4)\textand\max(Y_1,Y_3)<\min(Y_2,Y_4)\}\notag\\
    &+\Pr\{\max(X_1,X_2)<\min(X_3,X_4)\textand\max(Y_1,Y_4)<\min(Y_2,Y_3)\}\notag\\
    &+\Pr\{\max(X_1,X_2)<\min(X_3,X_4)\textand\max(Y_2,Y_3)<\min(Y_1,Y_4)\}\notag\\
    &+\Pr\{\max(X_1,X_2)<\min(X_3,X_4)\textand\max(Y_2,Y_4)<\min(Y_1,Y_3)\}
\end{align}
assuming marginal continuity of $(X,Y)^{\top}$.
Combining \eqref{eqn:term3_sym1} and \eqref{eqn:term3_sym2} gives
\begin{align}\label{eqn:term3}
    &\Pr\{\max(X_1,X_2)<\min(X_3,X_4)\textand\max(Y_1,Y_3)<\min(Y_2,Y_4)\}\notag\\
=\; &\frac14\Big[\Pr\{\max(X_1,X_2)<\min(X_3,X_4)\notag\\
    &-\Pr\{\max(X_1,X_2)<\min(X_3,X_4)\textand\max(Y_1,Y_2)<\min(Y_3,Y_4)\}\notag\\
    &-\Pr\{\max(X_1,X_2)<\min(X_3,X_4)\textand\max(Y_3,Y_4)<\min(Y_1,Y_2)\}\Big]\notag\\
=\; &\frac14\Big\{\int (2FF_1-F^2) d(2FF_2-F^2) - \int F^2 d(2F+2F_1F_2-2FF_1-2FF_2+F^2)\Big\}.
\end{align}
The identity $(i)$ follows by plugging \eqref{eqn:term1}, \eqref{eqn:term2}, and \eqref{eqn:term3} into \eqref{eqn:BDiden}.\\

{\bf Identity $(ii)$. }
Next we prove that $\Psi_1-\Psi_2=0$. A straightforward computation gives
\begin{align}\label{eqn:smart1}
    \Psi_1-\Psi_2
%    &\int F^2 d(F+F_1F_2)-\int F^2 d(FF_1)-2\int FF_1 d(FF_2)+\int FF_1d(F^2)+\frac1{18}\notag\\
%    &-\Big(\int F^2 dF-2\int FF_1F_2dF+2\int F^2dF_1dF_2-\frac19\Big)\notag\\
=\; &-\int F^2 d(F_1F_2)-\int F^2 d(FF_1)-2\int FF_1 d(FF_2)+\int FF_1d(F^2)+2\int FF_1F_2dF+\frac1{6}\notag\\
=\; &-\iint F^2\frac{\partial F_1}{\partial x}\frac{\partial F_2}{\partial y} dxdy
-\iint F^2\frac{\partial F_1}{\partial x}\frac{\partial F}{\partial y}dxdy
+\iint F^2F_1\frac{\partial^2 F}{\partial x\partial y}dxdy\notag\\
&-2\iint FF_1\frac{\partial F}{\partial x}\frac{\partial F_2}{\partial y}dxdy
+2\iint FF_1\frac{\partial F}{\partial x}\frac{\partial F}{\partial y}dxdy
+\frac16.
\end{align}
To further simplify \eqref{eqn:smart1}, notice that
\begin{align}\label{eqn:smart2}
 \iint\Big(F^2\frac{\partial F_1}{\partial x}\frac{\partial F}{\partial y}+2FF_1\frac{\partial F}{\partial x}\frac{\partial F}{\partial y}+F^2F_1\frac{\partial^2 F}{\partial x\partial y}\Big)dxdy
=\iint\frac{\partial^2 (F^3F_1/3)}{\partial x\partial y}dxdy
=\frac13.
\end{align}
Adding \eqref{eqn:smart1} and \eqref{eqn:smart2} together yields
\begin{align*}
\Psi_1-\Psi_2=\; &-\Big(\iint F^2\frac{\partial F_1}{\partial x}\frac{\partial F_2}{\partial y} dxdy
+2\iint FF_1\frac{\partial F}{\partial x}\frac{\partial F_2}{\partial y}dxdy\Big)
-2\iint F^2\frac{\partial F_1}{\partial x}\frac{\partial F}{\partial y}dxdy
+\frac12\\
=\; &-\int\frac{\partial F_2}{\partial y}\int \frac{\partial (F^2F_1)}{\partial x} dxdy
-2\int \frac{\partial F_1}{\partial x}\int F^2\frac{\partial F}{\partial y}dydx
+\frac12\\
=\; &-\int\frac{\partial F_2}{\partial y}F_2^2 dy
-2\int \frac{\partial F_1}{\partial x}\frac{F_1^3}{3} dx
+\frac12\\
=\; &-\frac{1}{3}-2\Big(\frac{1}{12}\Big)+\frac12=0,
\end{align*}
which completes the proof of identity $(ii)$.\\

{\bf Identity $(iii)$. } 
This identity was discovered by \citet{zbMATH03366369}.
To see this, it suffices to show that
\begin{equation*}
\Psi_3-\Psi_2=\int F_1^2F_2^2dF - 4\int FF_1F_2dF_1dF_2 + 2\int F_1^2F_2^2dF_1dF_2 +\frac{1}{9}=0.
\end{equation*}

We start from the identity
\begin{align}\label{eqn:iden3_1}
1=\iint\frac{\partial^2 (FF_1^2F_2^2)}{\partial x\partial y}dxdy
=\; &\iint F_1^2F_2^2\frac{\partial^2 F}{\partial x\partial y}dxdy
  +\iint 4FF_1F_2\frac{\partial F_1}{\partial x}\frac{\partial F_2}{\partial y}dxdy\notag\\
 &+\iint 2F_1^2F_2\frac{\partial F}{\partial x}\frac{\partial F_2}{\partial y}dxdy
  +\iint 2F_1F_2^2\frac{\partial F_1}{\partial x}\frac{\partial F}{\partial y}dxdy.
\end{align}
We also note that
\begin{align}
    &\iint 2FF_1F_2\frac{\partial F_1}{\partial x}\frac{\partial F_2}{\partial y}dxdy+\iint F_1^2F_2\frac{\partial F}{\partial x}\frac{\partial F_2}{\partial y}dxdy\notag\\
=\; &\int F_2\frac{\partial F_2}{\partial y}\int \frac{\partial (FF_1^2)}{\partial x}dx dy\notag\\
=\; &\int F_2^2\frac{\partial F_2}{\partial y} dy=\frac13,
\end{align}
and
\begin{align}
    &\iint 2FF_1F_2\frac{\partial F_1}{\partial x}\frac{\partial F_2}{\partial y}dxdy+\iint F_1F_2^2\frac{\partial F_1}{\partial x}\frac{\partial F}{\partial y}dxdy\notag\\
=\; &\int F_1\frac{\partial F_1}{\partial x}\int \frac{\partial (FF_2^2)}{\partial y}dy dx\notag\\
=\; &\int F_1^2\frac{\partial F_1}{\partial x} dx=\frac13,
\end{align}
and
\begin{align}\label{eqn:iden3_2}
\iint 2F_1^2F_2^2\frac{\partial F_1}{\partial x}\frac{\partial F_2}{\partial y}dxdy=2\int F_1^2\frac{\partial F_1}{\partial x}dx\int F_2^2\frac{\partial F_2}{\partial y}dy=2\Big(\frac{1}{3}\Big)\Big(\frac{1}{3}\Big)=\frac{2}{9}.
\end{align}
Combining \eqref{eqn:iden3_1}--\eqref{eqn:iden3_2} concludes the claim.
\end{proof}

{
\section{Additional simulation results}\label{sec:additional-sim}

%\nb{
%In the following, we perform more Monte Carlo experiments 
%concerning the sizes of the three tests DHS$_{D}$, DHS$_{R}$, DHS$_{\tau^*}$.
%The values reported below are based on $5,000$ simulations at the
%nominal significance level of $0.05$, with sample size
%$n\in\{100,200,400,800\}$ and dimension $p\in\{n/2,n\}$.  
%All data sets are generated as an i.i.d.~sample 
%from a standard Gaussian distribution $\mX\sim N_p(0,\fI_p)$.
%
%The relationships between empirical sizes and sample sizes $n$ 
%for each choice of dimension $p$ are shown in Figure~\ref{fig:sim-size}, 
%which are seen to confirm Corollary \ref{crl:size_eg}.
%}

First, we report the sizes and powers of the proposed
tests with simulation-based critical values ($M=5,000$)
as shown in Table \ref{tab:size-power}. The table shows
results only for {Examples~\ref{eg:sim-size}, \ref{eg:sim-power3}, and \ref{eg:sim-power4}} as
the simulated powers under Example~\ref{eg:sim-power2} were all perfectly one.  
It can be observed that all sizes are now well controlled, with powers of the proposed
tests only slightly different from the ones without using simulation.

{
\renewcommand{\tabcolsep}{2.5pt}
\renewcommand{\arraystretch}{1.05}
\begin{table}[t]
\centering
\caption{Empirical sizes and powers of simulation-based rejection threshold in Examples \ref{eg:sim-size}--\ref{eg:sim-power4} (The powers under Example~\ref{eg:sim-power2} are all perfectly $1.000$ and hence omitted)}{
\footnotesize
\begin{tabular}{ccC{0.375in}C{0.375in}C{0.375in}c
                  C{0.375in}C{0.375in}C{0.375in}c
                  C{0.375in}C{0.375in}C{0.375in}c
                  C{0.375in}C{0.375in}C{0.375in}}
$n$ & $p$ &  DHS$_D$ & DHS$_R$ & DHS$_{\tau^*}\!$ && 
             DHS$_D$ & DHS$_R$ & DHS$_{\tau^*}\!$ && 
             DHS$_D$ & DHS$_R$ & DHS$_{\tau^*}\!$ && 
             DHS$_D$ & DHS$_R$ & DHS$_{\tau^*}\!$ \\
    &     &  \multicolumn{3}{c}{Example \ref{eg:sim-size}\ref{sim:1a}}   &&
             \multicolumn{3}{c}{Example \ref{eg:sim-power3}\ref{sim:4a}} &&
             \multicolumn{3}{c}{Example \ref{eg:sim-power3}\ref{sim:4b}} &&
             \multicolumn{3}{c}{Example \ref{eg:sim-power3}\ref{sim:4c}} \\
100 &  50 & 0.053 & 0.053 & 0.053 && 0.964 & 0.964 & 0.965 && 0.746 & 0.651 & 0.694 && 0.639 & 0.591 & 0.611 \\
    & 100 & 0.051 & 0.051 & 0.050 && 0.955 & 0.954 & 0.955 && 0.731 & 0.636 & 0.676 && 0.638 & 0.581 & 0.607 \\
    & 200 & 0.045 & 0.045 & 0.044 && 0.943 & 0.944 & 0.945 && 0.698 & 0.602 & 0.643 && 0.609 & 0.549 & 0.580 \\
    & 400 & 0.045 & 0.046 & 0.046 && 0.930 & 0.931 & 0.932 && 0.674 & 0.577 & 0.624 && 0.592 & 0.524 & 0.557 \\
    & 800 & 0.054 & 0.051 & 0.051 && 0.921 & 0.921 & 0.923 && 0.651 & 0.548 & 0.594 && 0.567 & 0.490 & 0.526 \\
200 &  50 & 0.050 & 0.053 & 0.051 && 0.991 & 0.991 & 0.991 && 0.896 & 0.853 & 0.872 && 0.822 & 0.800 & 0.810 \\
    & 100 & 0.048 & 0.048 & 0.047 && 0.984 & 0.985 & 0.985 && 0.874 & 0.824 & 0.847 && 0.803 & 0.775 & 0.787 \\
    & 200 & 0.046 & 0.045 & 0.044 && 0.983 & 0.984 & 0.984 && 0.852 & 0.794 & 0.820 && 0.785 & 0.757 & 0.769 \\
    & 400 & 0.051 & 0.058 & 0.055 && 0.983 & 0.984 & 0.984 && 0.842 & 0.778 & 0.805 && 0.766 & 0.738 & 0.751 \\
    & 800 & 0.042 & 0.044 & 0.046 && 0.978 & 0.978 & 0.979 && 0.809 & 0.746 & 0.776 && 0.741 & 0.708 & 0.727 \\
\vspace{-.5em}\\
    &     &   \multicolumn{3}{c}{Example \ref{eg:sim-power4}\ref{sim:5a}} 
          &&  \multicolumn{3}{c}{Example \ref{eg:sim-power4}\ref{sim:5b}}  
          &&  \multicolumn{3}{c}{Example \ref{eg:sim-power4}\ref{sim:5c}} \\
100 &  50 & 0.081 & 0.085 & 0.080 && 0.096 & 0.094 & 0.096 && 0.121 & 0.124 & 0.126 \\
    & 100 & 0.079 & 0.074 & 0.077 && 0.074 & 0.074 & 0.074 && 0.088 & 0.090 & 0.092 \\
    & 200 & 0.052 & 0.059 & 0.056 && 0.067 & 0.069 & 0.068 && 0.072 & 0.072 & 0.074 \\
    & 400 & 0.064 & 0.064 & 0.064 && 0.059 & 0.057 & 0.056 && 0.059 & 0.058 & 0.065 \\
    & 800 & 0.051 & 0.048 & 0.048 && 0.058 & 0.054 & 0.052 && 0.061 & 0.064 & 0.059 \\
200 &  50 & 0.099 & 0.099 & 0.098 && 0.110 & 0.114 & 0.112 && 0.115 & 0.120 & 0.115 \\
    & 100 & 0.060 & 0.064 & 0.063 && 0.081 & 0.084 & 0.080 && 0.090 & 0.091 & 0.087 \\
    & 200 & 0.066 & 0.067 & 0.071 && 0.046 & 0.046 & 0.044 && 0.080 & 0.070 & 0.079 \\
    & 400 & 0.058 & 0.062 & 0.058 && 0.060 & 0.070 & 0.069 && 0.059 & 0.058 & 0.058 \\
    & 800 & 0.045 & 0.049 & 0.050 && 0.052 & 0.050 & 0.050 && 0.061 & 0.060 & 0.062 \\
\end{tabular}}
\label{tab:size-power}
%Results are averaged over $5000$ simulated data sets.
\end{table} 

}

Next, in order to interpret the power in Examples~\ref{eg:sim-power2}--\ref{eg:sim-power4}, 
we consider the following example.

\begin{customexample}{\ref{eg:sim-power2}--\ref{eg:sim-power4} (continued)}
We consider modified data drawn as
\[\mX_{\alpha}=\alpha\mX+(1-\alpha)\bm{\cE}\]
where $\alpha\in[0,1]$ represents the level of a desired signal, $\mX$ is the same as that in Examples \ref{eg:sim-power2}--\ref{eg:sim-power4}, respectively, and $\bm{\cE}\sim N_{p}(0,\fI_{p})$ is independent of $\mX$.
\end{customexample}

The relationships between empirical powers (5,000 replicates) based on observations from $\mX_{\alpha}$ and the value $\alpha$ for Examples \ref{eg:sim-power2}--\ref{eg:sim-power4} (continued) are summarized in Figures \ref{fig:sim-power2}--\ref{fig:sim-power5}.
%The empirical powers for Examples \ref{eg:sim-power4}\ref{sim:5a}--\ref{sim:5b} (continued) are around $0.05$ in most cases and hence omitted here.
As expected, the power of each test is monotonically increasing in
$\alpha$, i.e., as the signal increases.
Similar patterns as we discussed for Examples \ref{eg:sim-power2}--\ref{eg:sim-power4} can be found here. It can be noticed that, 
%however, that LD$_{\tau^*}$ and YZS perform generally better than the three proposed tests in Example \ref{eg:sim-power1} (continued) since it is still a relatively dense setting. On the other hand, 
the three proposed tests, followed by LD$_{\tau^*}$, uniformly dominate the other tests in Examples \ref{eg:sim-power2} and \ref{eg:sim-power3} (continued) that are sparse settings. 
}

\begin{landscape}

\begin{figure}
\centering
\includegraphics[height=3in]{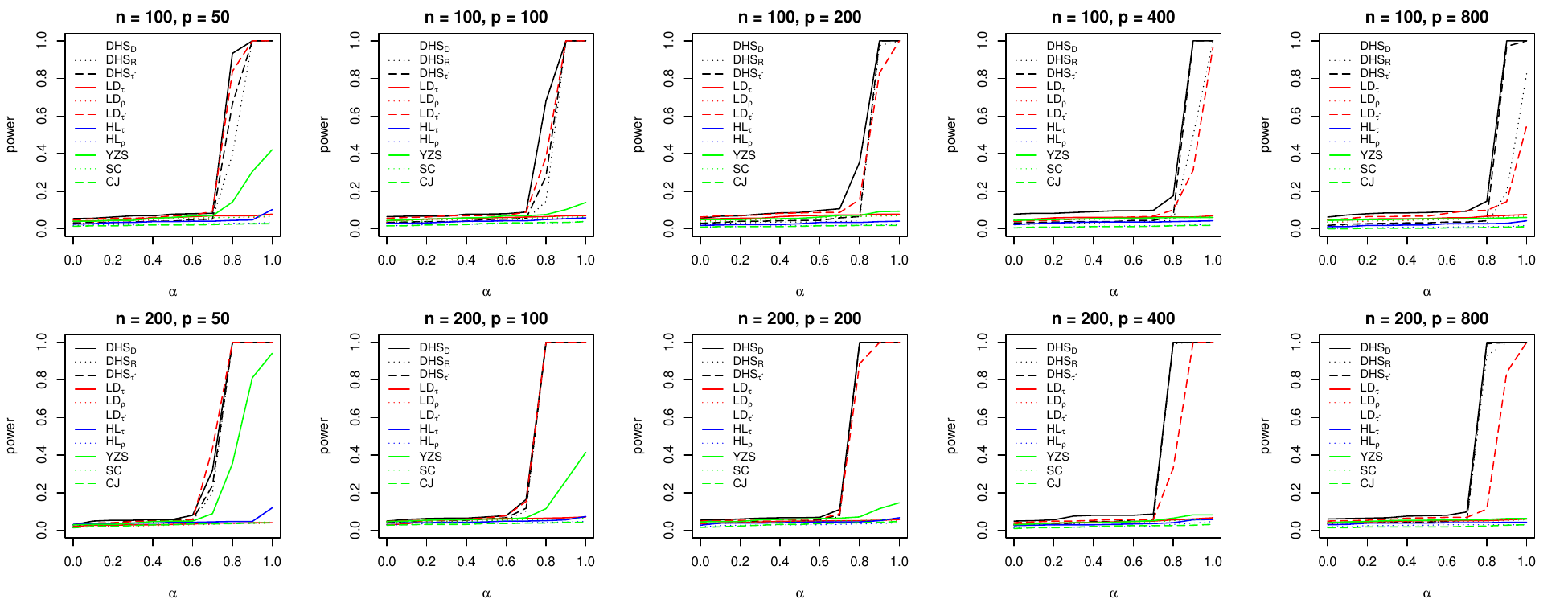}\\[-1em]
\rule{8in}{0.4pt}
\includegraphics[height=3in]{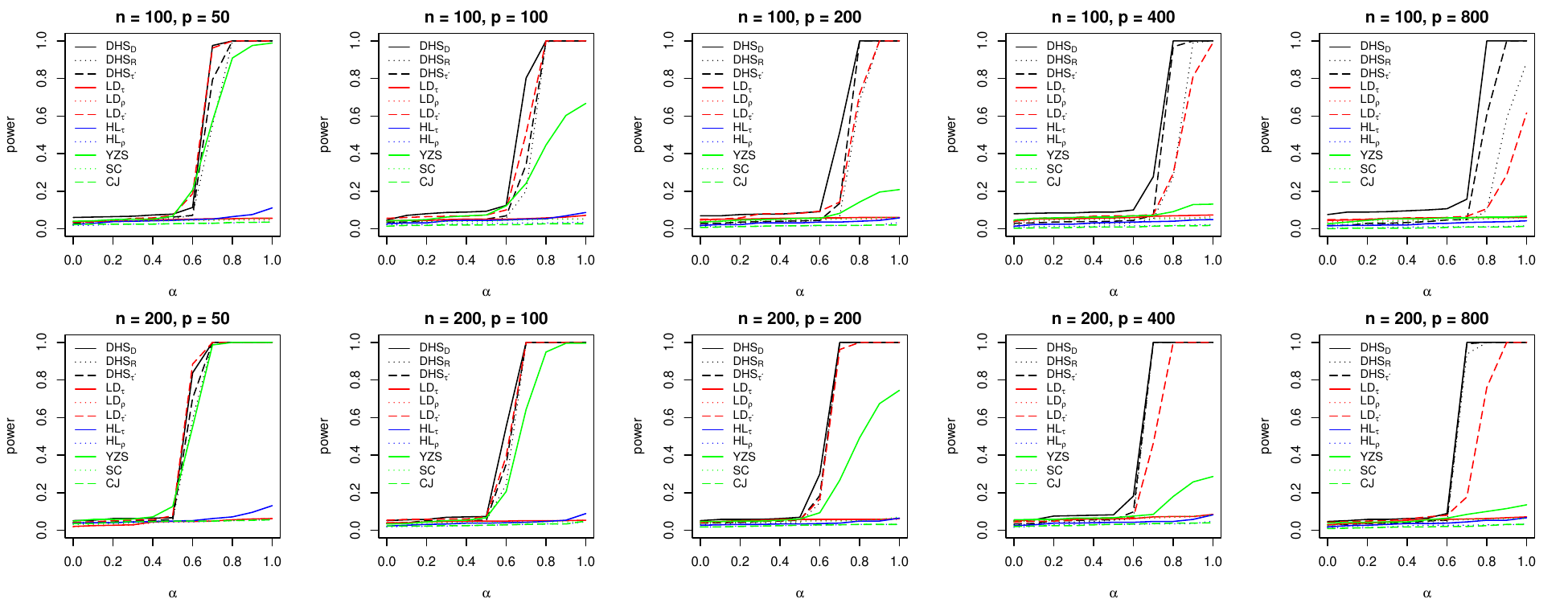}
\caption{Empirical powers of the eleven competing tests in continued Example~\ref{eg:sim-power2}\ref{sim:3a'} (first two rows) and continued Example~\ref{eg:sim-power2}\ref{sim:3b'} (last two rows). 
The y-axis represents the power based on 5,000 replicates and 
the x-axis represents the level of a desired signal.}\label{fig:sim-power2}
\end{figure}

\begin{figure}
\centering
\includegraphics[height=3in]{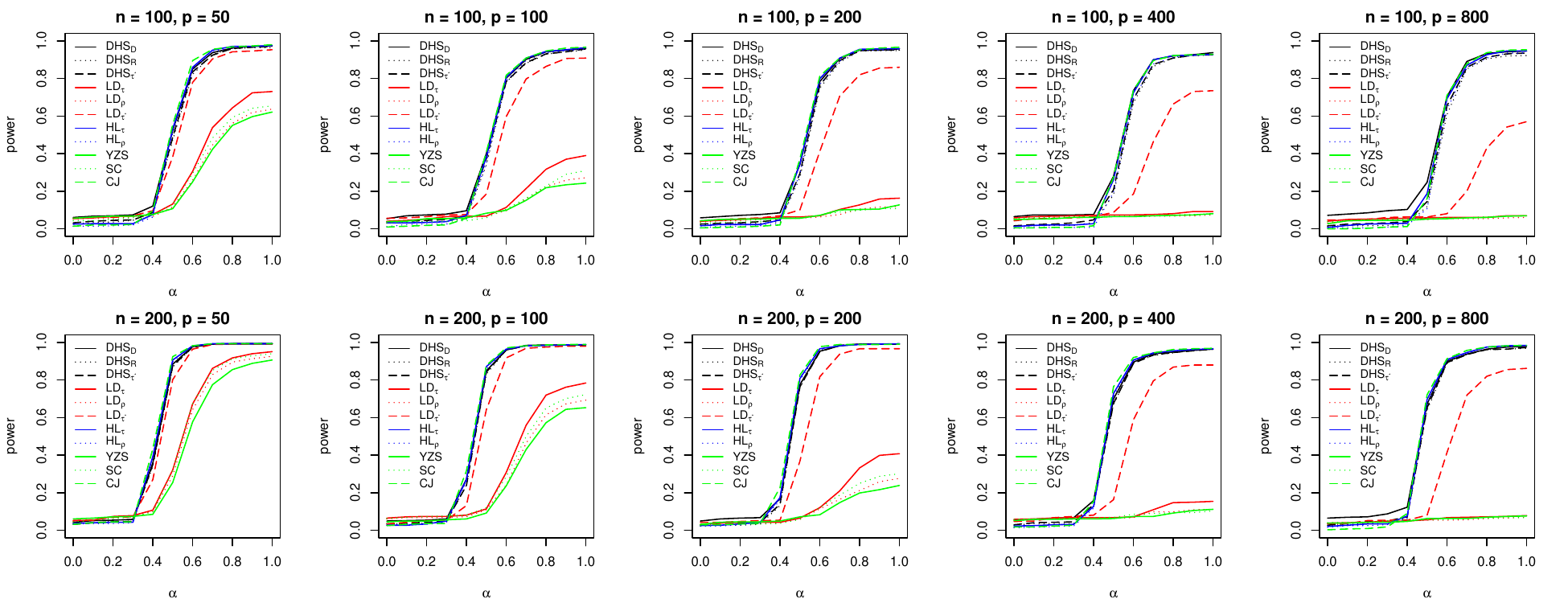}\\[-1em]
\rule{8in}{0.4pt}
\includegraphics[height=3in]{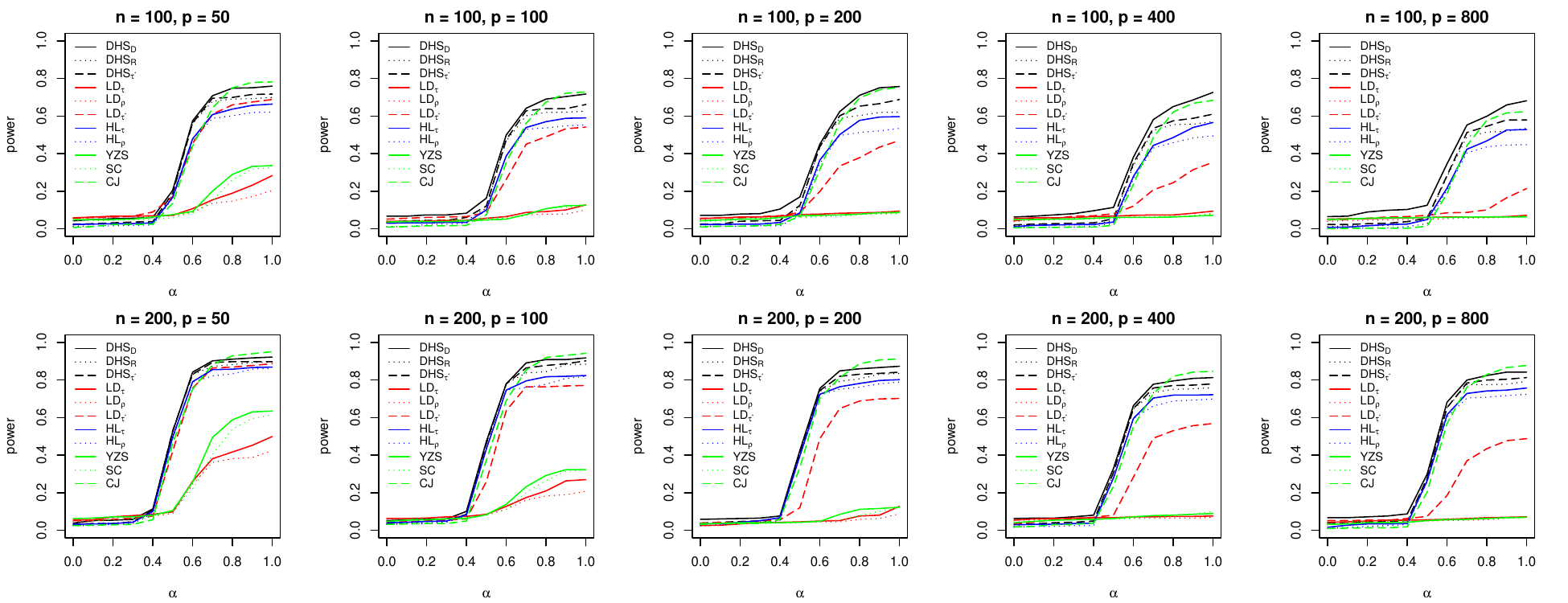}
\caption{Empirical powers of the eleven competing tests in continued Example~\ref{eg:sim-power3}\ref{sim:4a} (first two rows) and continued Example~\ref{eg:sim-power3}\ref{sim:4b} (last two rows). 
The y-axis represents the power based on 5,000 replicates and 
the x-axis represents the level of a desired signal.}\label{fig:sim-power3}
\end{figure}

\begin{figure}
\centering
\includegraphics[height=3in]{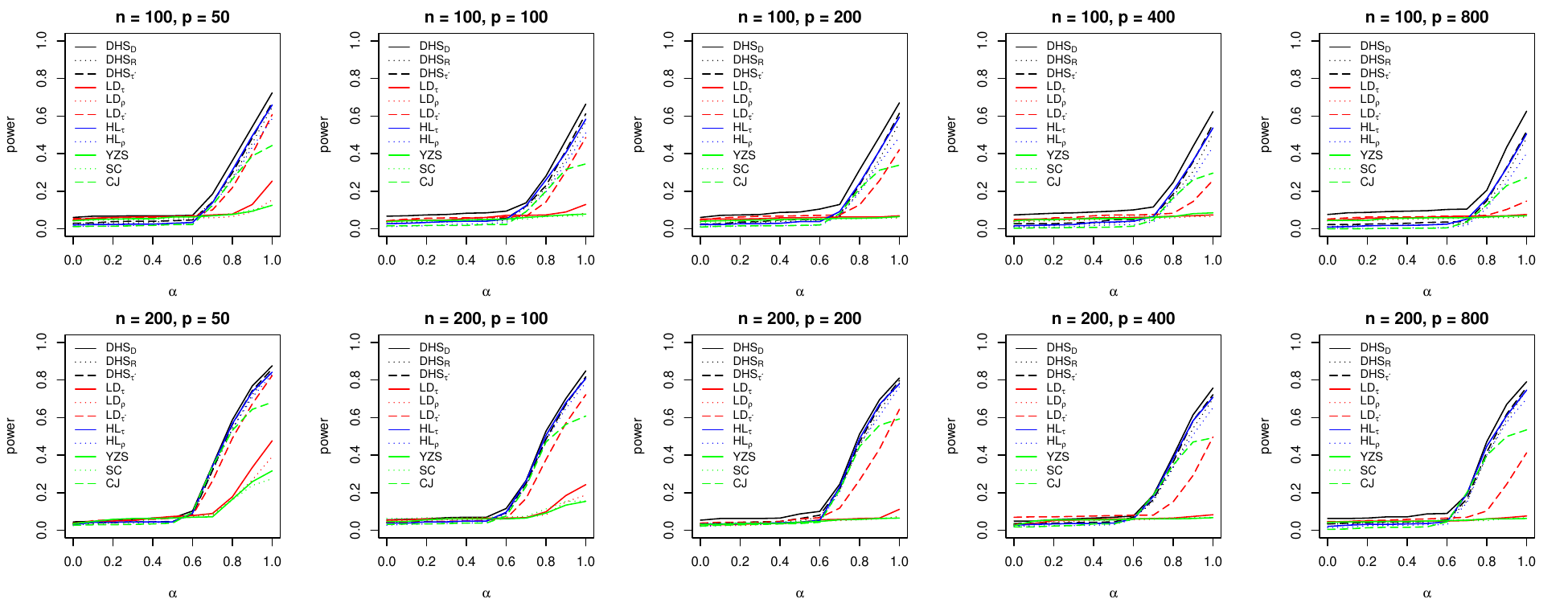}\\[-1em]
\rule{8in}{0.4pt}
\includegraphics[height=3in]{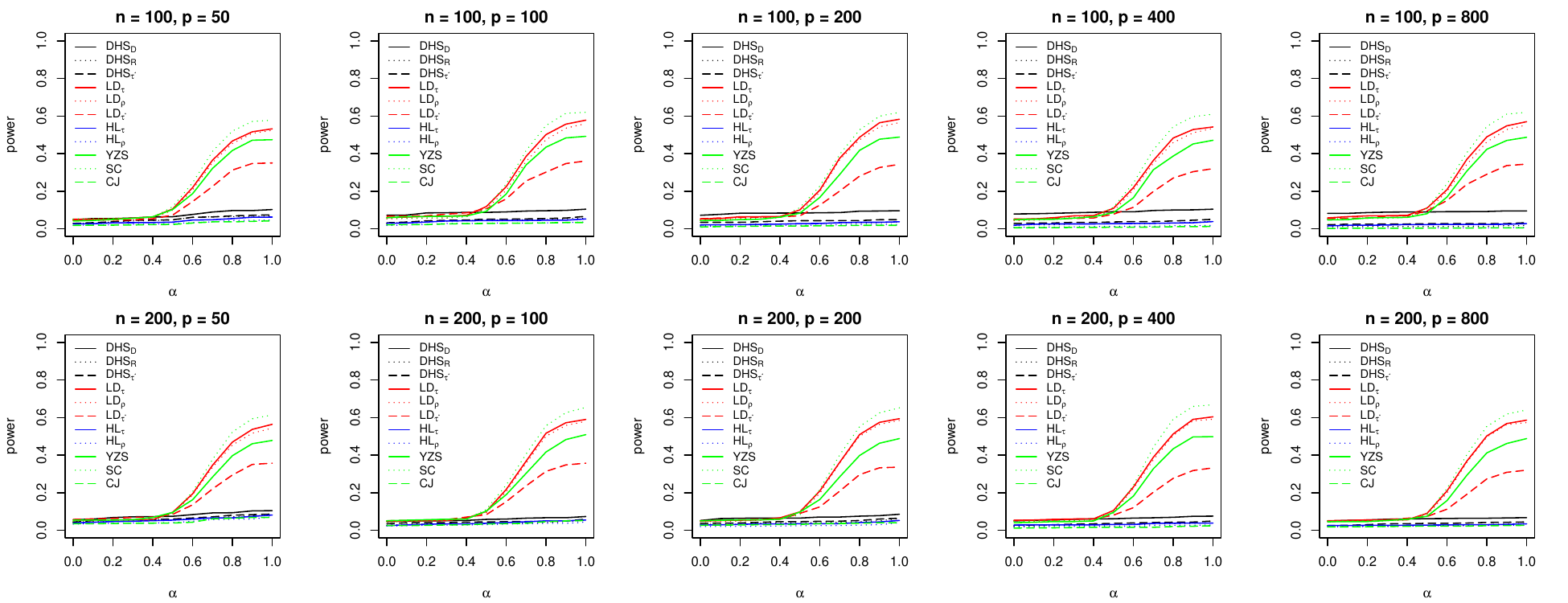}
\caption{Empirical powers of the eleven competing tests in continued Example~\ref{eg:sim-power3}\ref{sim:4c} (first two rows) and continued Example~\ref{eg:sim-power4}\ref{sim:5a} (last two rows). 
The y-axis represents the power based on 5,000 replicates and 
the x-axis represents the level of a desired signal.}\label{fig:sim-power4}
\end{figure}

\begin{figure}
\centering
\includegraphics[height=3in]{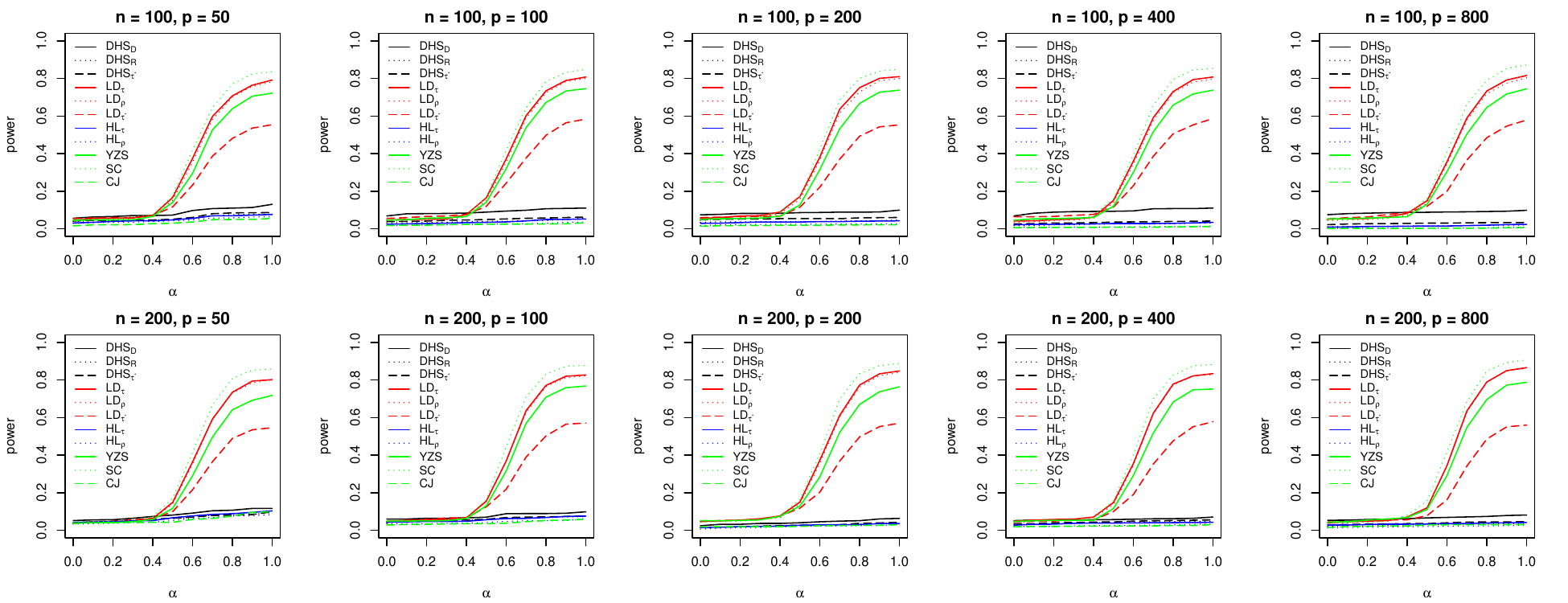}\\[-1em]
\rule{8in}{0.4pt}
\includegraphics[height=3in]{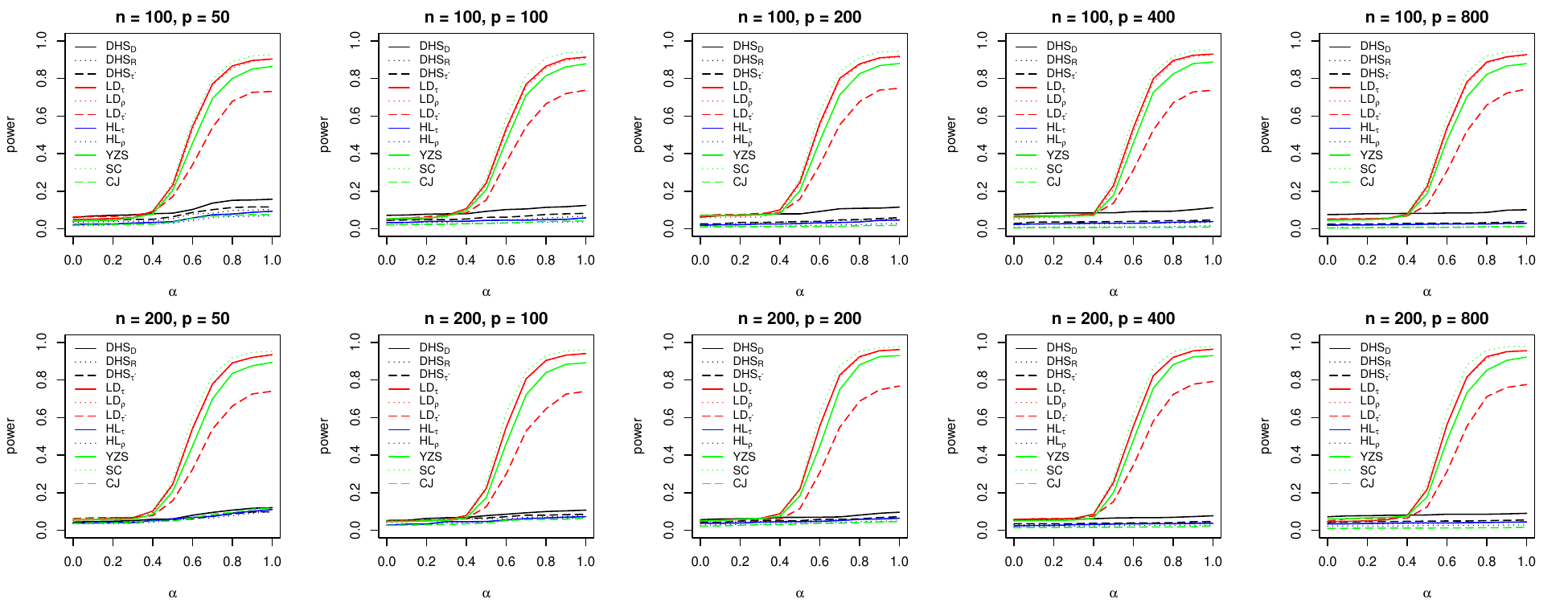}
\caption{Empirical powers of the eleven competing tests in continued Example~\ref{eg:sim-power4}\ref{sim:5b} (first two rows) and continued Example~\ref{eg:sim-power4}\ref{sim:5c} (last two rows). 
The y-axis represents the power based on 5,000 replicates and 
the x-axis represents the level of a desired signal.}\label{fig:sim-power5}
\end{figure}

\end{landscape}

%\newpage
%\bibliographystyle{apalike}
%\bibliography{DHS_ams}

\end{document}